\documentclass[10pt]{amsart}

\usepackage{varioref}
\usepackage[french,english]{babel}
\usepackage[utf8]{inputenc}	
\usepackage[T1]{fontenc}

\usepackage[normalem]{ulem}
\usepackage{latexsym}

\usepackage{lipsum}

\usepackage{amsmath}			
\usepackage{amsfonts}			
\usepackage{amssymb}			
\usepackage{amsthm}
\usepackage{thmtools}			
\usepackage{amsbsy}				
\usepackage{wasysym}
\usepackage{dsfont}			
\usepackage{mathrsfs}
\usepackage{setspace}

\usepackage{graphicx}
\usepackage{float}	
\usepackage{comment}
\usepackage{layout}
\usepackage{color}			

\usepackage[all]{xy}			
\usepackage{tikz}				
\usetikzlibrary{cd,calc,shapes,shapes.geometric,decorations.pathreplacing,decorations.pathmorphing,arrows} 

\DeclareSymbolFont{Shuffle}{U}{shuffle}{m}{n}
\DeclareFontFamily{U}{shuffle}{}
\DeclareFontShape{U}{shuffle}{m}{n}{%
	<-8>shuffle7%
	<8->shuffle10%
}{}
\DeclareMathSymbol\shuffle{\mathbin}{Shuffle}{"001}
\DeclareMathSymbol\cshuffle{\mathbin}{Shuffle}{"002}

\usepackage{enumitem}		
\setitemize[0]{font=\textbf, label=--}

\usepackage{color}
\definecolor{Chocolat}{rgb}{0.36, 0.2, 0.09}
\definecolor{BleuTresFonce}{rgb}{0.215, 0.215, 0.36}

\usepackage[colorlinks,final,backref=page,hyperindex]{hyperref}
\hypersetup{citecolor=BleuTresFonce, linkcolor=Chocolat}

\usepackage{cleveref}

\theoremstyle{plain}
\newtheorem{thm}{Theorem}[section]
\newtheorem*{thm*}{Theorem}
\newtheorem{lem}[thm]{Lemma}
\newtheorem{prop}[thm]{Proposition}
\newtheorem*{prop*}{Proposition}
\newtheorem{cor}[thm]{Corollary}

\theoremstyle{definition}
\newtheorem{defi}[thm]{Definition}
\newtheorem*{defi*}{Definition}
\newtheorem{defiprop}[thm]{Definition/Proposition}

\newtheorem{nota}[thm]{Notation}
\newtheorem*{nota*}{Notation}
\newtheorem{rem}[thm]{Remark}

\newcommand{\Z}{\mathbb{Z}}
\newcommand{\N}{\mathbb{N}}

\newcommand{\K}{\mathcal{K}}

\newcommand{\rmT}{\mathrm{T}}

\newcommand{\calC}{\mathcal{C}}
\newcommand{\calP}{\mathcal{P}}

\newcommand{\scrA}{\mathscr{A}}

\newcommand{\id}{\mathrm{id}}

\newcommand{\scrF}{\mathscr{F}}

\renewcommand{\phi}{\varphi}

\newcommand{\antish}{{\mbox{\footnotesize{!`}}}}

\newcommand{\As}{\mathcal{A}s}
\newcommand{\Com}{\mathcal{C}om}

\newcommand{\Lie}{\mathcal{L}ie}
\newcommand{\DLie}{{\mathcal{DL}ie}}
\newcommand{\DCom}{\mathcal{DC}om}
\newcommand{\DPois}{\mathcal{DP}ois}
\renewcommand{\Bar}{\mathsf{B}}
\newcommand{\Cobar}{\mathsf{\Omega}}

\newcommand{\Ob}{\mathrm{Ob}\,}
\newcommand{\op}{^{\mathrm{op}}}
\newcommand{\sh}{^{\mathrm{sh}}}

\newcommand{\Val}{^{\mathrm{Val}}}
\newcommand{\aug}{^{\mathrm{aug}}}
\newcommand{\coaug}{^{\mathrm{coaug}}}

\newcommand{\Set}{\mathsf{Set}}

\newcommand{\Ho}{\mathsf{Ho}}
\newcommand{\Ch}{\mathsf{Ch}}
\newcommand{\Alg}{\mathsf{Alg}}
\newcommand{\DGA}{\mathsf{Dga}}
\newcommand{\CDGA}{\mathsf{Cdga}}
\newcommand{\Smod}{\mathfrak{S}\mbox{-}\mathsf{mod}}
\newcommand{\Smodred}{\mathfrak{S}\mbox{-}\mathsf{mod}^{\mathrm{red}}}

\newcommand{\Sbimodred}{\mathfrak{S}\mbox{-}\mathsf{bimod}^{\mathrm{red}}}

\newcommand{\Fin}{\mathsf{Fin}}
\newcommand{\Ord}{\mathsf{Ord}}

\newcommand{\Hom}{\mathrm{Hom}}
\newcommand{\sHom}{\mathrm{\underline{hom}}}
\newcommand{\sDer}{\mathrm{\underline{Der}}}

\newcommand{\Func}{\mathrm{Func}}
\newcommand{\Wall}{\mathcal{W}^{\mathrm{conn}}}
\newcommand{\Colo}{\mathcal{C}\mathrm{ol}}
\newcommand{\Ind}{\mathrm{Ind}}
\newcommand{\Res}{\mathrm{Res}}

\newcommand{\sgn}{\widetilde{\mathrm{sgn}}}

\newcommand{\End}{\mathrm{End}}
\newcommand{\sEnd}{\mathrm{\underline{End}}}
\newcommand{\GL}{\mathrm{GL}}

\newcommand{\commu}{\circlearrowleft}

\newcommand{\DB}[2]{\{\hspace{-3pt}\{#1,#2\}\hspace{-3pt}\}}
\newcommand{\TB}[3]{\{\hspace{-3pt}\{#1,#2,#3\}\hspace{-3pt}\}}

\newcommand{\Ibox}{I_{\boxtimes}}

\usepackage{xr-hyper}
\title{Protoperads II: Koszul duality}
\address{LAGA, Universit\'e Paris 13, 99 Avenue Jean Baptiste Cl\'ement 93430, Villetaneuse, France}
\email{leray@math.univ-paris13.fr}
\author{Johan \textsc{Leray}} 
\date{\today}
\thanks{	This article is the homotopical part of the  PhD thesis of the author, supported by the project "Nouvelle \'Equipe", convention n$^\circ$2013-10203/10204 between La R\'egion des Pays de Loire and the University of Angers.  The author thanks the Centre Henri Lebesgue ANR-11-LABX-0020-01 for its stimulating mathematical research programs. This paper was finished at the University Paris 13, where the author was financed by a postdoctoral allocation given by DIM Math Innov. The author is indebted to G. Powell who has carefully read and corrected the first version of this paper. The author also thanks E. Hoffbeck and B. Vallette for our useful discussions.}

\begin{document}
	
\begin{abstract}
	In this paper, we construct a bar-cobar adjunction and a Koszul duality theory for protoperads, which are an operadic type notion encoding faithfully some categories of bialgebras with diagonal symmetries, like double Lie algebras ($\DLie$). We give a criterion to show that a binary quadratic protoperad is Koszul and we apply it successfully to the protoperad $\DLie$.  As a corollary, we deduce that the properad $\DPois$ which encodes double Poisson algebras is Koszul. This allows us to describe the homotopy properties of double Poisson algebras which play a key role in non commutative geometry.
\end{abstract}

\maketitle
	
\section*{Introduction}
This paper develops the  Koszul duality theory for \emph{protoperads}, defined in \cite{Ler18i}, which are an analog of properads (see \cite{Val03,Val07}) with more less symmetries. The main application of this theory is the proof of the Koszulness of the properad which encodes double Lie algebras, from which it follows that the properad encoding double Poisson algebras is Koszul. 

The motivation for this work is to determine what is a double Poisson bracket up to homotopy. A double Poisson structure, as defined by Van den Bergh in  \cite{VdB08}, gives a Poisson structure in noncommutative algebraic geometry (see \cite{Gin05,VdB08-2}) under the \emph{Kontsevich-Rosenberg principle}, i.e. if $A$ is a double Poisson algebra, then the associated affine representation schemes $\mathrm{Rep}_n(A)$ have (classical) Poisson structures.

In order to determine the homotopical properties of a family of algebras, we use the classical strategy, which was already used to understand, for example, the homotopical properties of Gerstenhaber algebras (and also the homotopic properties of associative, commutative, Lie, Poisson, etc, algebras). The idea is to go to the upper level and understand the homological properties of the algebraic object that encodes the structure, such as the  operad $\mathcal{G}\mathrm{erst}$ for Gerstenhaber algebras. In the good case where the operad (or the properad) satisfies good properties, we can use Koszul duality in order to have a minimal cofibrant replacement of our operad. We can then go down to the level of algebras. Thanks to this cofibrant replacement (in the case of Gerstenhaber algebras, the operad $\mathcal{G}_\infty$), we obtain the associated notion of algebra up to homotopy: for example, Gerstenhaber algebras up to homotopy are encoded by  $\mathcal{G}_\infty$ (see \cite{Gin04} or \cite[Sect. 2.1]{GTV12}). This structure has a good homotopical behaviour at the  algebras' level, the homotopy transfer theorem (see  \cite[Sect. 10.3]{LV12} for algebras over an operad), etc.

Double Poisson structures are \emph{properadic} in nature as they are made up of operations with multiple inputs and multiples outputs. They are encoded by the properad $\DPois$, which is constructed with the properads $\As$ and $\widetilde{\DLie}$  (see \Cref{lem::loi_remplacement_compatible}), where the properad $\widetilde{\DLie}$ encodes double Lie structure and the properad  $\As$ encodes associative algebra structure. The properad $\widetilde{\DLie}$ is a quadratic properad defined by generators and relations, with the generator $V_\DLie$ concentrated in arity $(2,2)$:
\[
V_\DLie:
\begin{tikzpicture}[scale=0.2,baseline=-3]
\draw (0,0.5) node[above] {$\scriptscriptstyle{1}$};
\draw (2,0.5) node[above] {$\scriptscriptstyle{2}$};
\draw[fill=black] (-0.3,-0.5) rectangle (2.3,0);
\draw[thin] (0,-1) -- (0,0.5);
\draw[thin] (2,-1) -- (2,0.5);
\draw (0,-1) node[below] {$\scriptscriptstyle{1}$};
\draw (2,-1) node[below] {$\scriptscriptstyle{2}$};
\end{tikzpicture}
\ = - \
\begin{tikzpicture}[scale=0.2,baseline=-3]
\draw (0,0.5) node[above] {$\scriptscriptstyle{2}$};
\draw (2,0.5) node[above] {$\scriptscriptstyle{1}$};
\draw[fill=black] (-0.3,-0.5) rectangle (2.3,0);
\draw[thin] (0,-1) -- (0,0.5);
\draw[thin] (2,-1) -- (2,0.5);
\draw (0,-1) node[below] {$\scriptscriptstyle{2}$};
\draw (2,-1) node[below] {$\scriptscriptstyle{1}$};
\end{tikzpicture}
\]
and the relation in arity $(3,3)$
\[
R_{\mathcal{DJ}}:
\begin{tikzpicture}[scale=0.2,baseline=-3]
\draw (0,3) node[below] {$\scriptscriptstyle{1}$};
\draw (2,3) node[below] {$\scriptscriptstyle{2}$};
\draw (4,3) node[below] {$\scriptscriptstyle{3}$};
\draw[fill=black] (1.7,0.5) rectangle (4.3,1);
\draw[fill=black] (-0.3,-0.5) rectangle (2.3,0);
\draw[thin] (0,-1) -- (0,1.5);
\draw[thin] (2,-1) -- (2,1.5);
\draw[thin] (4,-1) -- (4,1.5);
\draw (0,-1) node[below] {$\scriptscriptstyle{1}$};
\draw (2,-1) node[below] {$\scriptscriptstyle{2}$};
\draw (4,-1) node[below] {$\scriptscriptstyle{3}$};
\end{tikzpicture}
~~+~~
\begin{tikzpicture}[scale=0.2,baseline=-3]
\draw (0,3) node[below] {$\scriptscriptstyle{2}$};
\draw (2,3) node[below] {$\scriptscriptstyle{3}$};
\draw (4,3) node[below] {$\scriptscriptstyle{1}$};
\draw[fill=black] (1.7,0.5) rectangle (4.3,1);
\draw[fill=black] (-0.3,-0.5) rectangle (2.3,0);
\draw[thin] (0,-1) -- (0,1.5);
\draw[thin] (2,-1) -- (2,1.5);
\draw[thin] (4,-1) -- (4,1.5);
\draw (0,-1) node[below] {$\scriptscriptstyle{2}$};
\draw (2,-1) node[below] {$\scriptscriptstyle{3}$};
\draw (4,-1) node[below] {$\scriptscriptstyle{1}$};
\end{tikzpicture}
~~+~~
\begin{tikzpicture}[scale=0.2,baseline=-3]
\draw (0,3) node[below] {$\scriptscriptstyle{3}$};
\draw (2,3) node[below] {$\scriptscriptstyle{1}$};
\draw (4,3) node[below] {$\scriptscriptstyle{2}$};
\draw[fill=black] (1.7,0.5) rectangle (4.3,1);
\draw[fill=black] (-0.3,-0.5) rectangle (2.3,0);
\draw[thin] (0,-1) -- (0,1.5);
\draw[thin] (2,-1) -- (2,1.5);
\draw[thin] (4,-1) -- (4,1.5);
\draw (0,-1) node[below] {$\scriptscriptstyle{3}$};
\draw (2,-1) node[below] {$\scriptscriptstyle{1}$};
\draw (4,-1) node[below] {$\scriptscriptstyle{2}$};
\end{tikzpicture}
\ .
\]
Thus double Lie bracket on a chain complex $A$ is given by a morphism of properads $\widetilde{\DLie} \rightarrow \End_A$ where $\End_A$ is the properad of endomorphisms of $A$ (see \cite{Val07} for the definition).

The theory of properads is the good general algebraic framework to encode operations with several inputs and outputs. In certain cases, this framework can be simplified. For example, algebraic structures with several inputs and \emph{one} output, like associative, commutative or Lie algebras, are encoded by \emph{operads} (see \cite{LV12}). In a certain sense, the operadic framework is the minimal one to study such structures. In this smaller framerwork, homotopical properties are much easier to study.

Similarly, protoperads form a special class of properads, which provide the appropriate framework for studying the double Lie properads. In the first article \cite{Ler18i}, we have developed  this minimal framework, such that there exists a protoperad $\DLie$ which encodes the double Lie structure.
In \cite{Ler18i}, we proved the existence of the free protoperad functor and gave an explicit combinatorial description of this, in terms of \emph{bricks and walls}. An important property of protoperads is their compatibility with properads via the induction functor (see \Cref{def::functor induction}). 

In this paper, we develop the homological algebra for protoperads. With the monoidal exact functor of induction, we prove the existence of a bar-cobar adjunction in the case of protoperads:
\[
	\Cobar  : \mathsf{coprotoperads}\coaug_k \leftrightarrows \mathsf{protoperads}\aug_k : \Bar .
\]
We obtain also the following theorem, the protoperadic analogue of the criterion of Koszul of the properads \cite[Th. 149]{Val03},\cite{Val07}.

\begin{thm*}[Koszul criterion -- \textup{(cf. \Cref{thm::critere_Koszul})}]
Let $\calP$ be a connected weight-graded protoperad. The following are equivalent:
\begin{enumerate}
	\item the protoperad $\calP$ is Koszul;
	\item the inclusion $\calP^\antish \hookrightarrow \Bar \calP$ is a quasi-isomorphism;
	\item the morphism of protoperads $\Cobar \calP^\antish \rightarrow \calP$ is a quasi-isomorphism, where $\calP^\antish$ is the Koszul dual of $\calP$ (see \Cref{thm::duale_de_koszul})
\end{enumerate}
\end{thm*}

We give a useful criterion to show that a binary quadratic protoperad (i.e. a quadratic protoperad generated by a $\mathfrak{S}$-module concentrated in arity $2$)  is Koszul. Take a binary quadratic protoperad $\calP$ given by generators and relations.
We associate to $\calP$ a family of associative algebras $\mathscr{A}(\calP,n)$, for $2\geqslant n$ in $\N$. The algebra $\mathscr{A}(\calP,n)$ is constructed so that its bar construction 
splits and such that one of these factors is the $n$-th arity of the normalized simplicial bar construction of the protoperad $\calP $.
	
\begin{thm*}[Koszul criterion (see \Cref{thm::critere_koszulite})]
	Let $\calP $ be a binary quadratic protoperad. If, for all integers $n\geqslant 2$, the quadratic algebra $\scrA(\calP,n)$ is Koszul, then the protoperad $\calP $ is Koszul.
\end{thm*}

This is a useful criterion because the study of the Koszulness of algebras is easier than for pro(to)perads. Many tools are available, such as PBW or Gr\"obner bases, or rewriting methods (see \cite[Chap. 4]{LV12}). 

We use this criterion to show that the protoperad $\mathcal{DL}ie$ and therefore the properad $\widetilde{\DLie}$, which encodes double Lie algebras, are Koszul.
\begin{thm*}[see \Cref{thm::proto_DLie_Koszul} and \Cref{thm::prop_DLie_Koszul}]
	The protoperad $\DLie$ is Koszul. So, the properad $\widetilde{\DLie}$ is Koszul.
\end{thm*}
This theorem is very important: it is the first example of a Koszul properad with a generator not in arity $(1,2)$ or $(2,1)$.

And so, with an argument of distributive law, we deduce the main theorem of this paper.
\begin{thm*}[see \Cref{thm::DPois_Koszul}]
	The properad $\DPois$ is Koszul.
\end{thm*}

In an future article, we will explain the homotopy transfert theorem for properadic algebras and we will use this in an other future work, where we will study the implications of \Cref{thm::DPois_Koszul} in derived noncommutative algebraic geometry \`a la Berest et al. (see \cite{BCER12,BFPRW14,BFR12,CEEY15}). In particular, we will link it to pre-Calabi Yau structures as in \cite{Yeu18,IK18}.
We will also look at the cohomological theory of double Poison algebras. Indeed, the work of Merkulov and Vallette gives the notion of deformation theory of $\calP$-algebras, for $\calP$ a properad. We want to link the deformation complex defined in \cite{MV09II} with the work of Pichereau et al. who defined the cohomology of \emph{differential} double Poisson algebra (see \cite{PV08}).

\subsection*{Organization of the paper}
After a review of definitions and some properties of protoperads (see \cite{Ler18i}) in  \Cref{sect::review_protoperads}, following the results on properads (see \cite{Val03,Val07,MV09I}), we introduce the notion of shuffle protoperads in  \Cref{sect::shuffle protoperad}. In \Cref{sect::Koszul duality}, we define the Koszul duality of protoperads. We transpose a part of the results on properads obtained by Vallette in \cite{Val03,Val07} to the protoperatic framework thanks to the exactness of the induction functor $\Ind$ (see \cite[\Cref{ProtoI-prop::Ind_exact}]{Ler18i}). In \Cref{sect::simplicial}, we define the simplicial bar construction and the normalized one for protoperads and we described the levelization morphism (see \Cref{def::levelizationmorphism}). In \Cref{sect::criterion}, we give a criterion to prove that a binary quadratic protoperad is Koszul and we use it to prove that the protoperad $\DLie$ is Koszul. Finally, in \Cref{sect::DPoisKoszul}, we use results of Vallette on distributive laws to prove that the properad $\DPois$ is Koszul.


\setcounter{tocdepth}{1} 
\tableofcontents

\section*{Notations}
	We write $\N^*$ for the set $\N\backslash \{0\}$. In all this paper, $k$ is a field with characteristic different to $2$. We denote by $\Fin$, the category with finite sets as objects and bijections as morphisms and $\Set$, the category of all sets and all maps. For two integers $a$ and $b$, we note by $[\![a,b]\!]$ the set $[a,b]\cap \mathbb{Z}$, and, for $n\in \N^*$, $\mathfrak{S}_n$ is the automorphism group of $[\![1,n]\!]$, i.e. $\mathfrak{S}_n=\mathrm{Aut}_{\Fin}\left( [\![1,n]\!] \right)$. We denote by $\Ch_k$ the category of $\Z$-graded chain complexes over the field $k$.
	
\section{Recollections on pro(to)perads}\label{sect::review_protoperads}
We briefly recall the definition of protoperads and some results of \cite{Ler18i}.  We denote by $\Smodred_k$, the category of contravariant functors from $\Fin$ to the category of chain complexes $\Ch_k$ such that $P(\varnothing)=0$.

\subsection{Combinatorial functors}\label{subsect::combinatorial_functors}

We recall two important functorial combinatorial constructions which are described in \cite[\Cref{ProtoI-sect::bricks_and_walls}]{Ler18i}: the functors $\Wall$ and $\mathcal{X}^{\mathrm{conn}}$.
\begin{nota*}
	For a poset $(P,\leqslant_P)$, we denote by $\mathrm{Succ}(P)$, the set of pairs $(r,s) \in P\times P$ such that $r<_P s$ and there does not exist $t\in P$ such that $r<_P t <_P t$.
\end{nota*}

The functor $\mathcal{W}_n: \Fin\op  \longrightarrow \Fin\op $ 
is given, for all finite sets $S$, by 
\[
\mathcal{W}_n(S):=
	\left\{ \big( W=\{W_\alpha\}_{\alpha\in A}\big)~~\left|~~ 
	\begin{array}{l}
		|A|=n; \ \cup_\alpha W_\alpha=S; \ \forall \alpha\in A, W_\alpha \ne \varnothing\\
		\forall s\in S,~\Gamma^W_s:=\{W_\alpha|s\in W_\alpha\} \\
		\quad \mbox{is totally ordered (by }\leqslant_s) \\
		\forall s,t\in S, ~ \forall a,b\in\Gamma_s\cap \Gamma_t,~~a\leqslant_s b 
	\end{array} \ 
	\right. \right\} \ ;
\]
for $W$ in $\mathcal{W}(S)$, the collection of partial orders $\{\leqslant_s\}_{s\in S}$ defines a canonical partial order on $W$ (see \cite[\Cref{ProtoI-lem::ordre_partiel_canonique}]{Ler18i}). The action of $\sigma \in \mathrm{Aut}(S)$ on $\big(\{W_\alpha\}_{\alpha\in A},\leqslant\big)$ in $\mathcal{W}_n$ is induced by the  action on $S$, i.e.
\[
	\big(\{W_\alpha\}_{\alpha\in A},\leqslant\big)\cdot\sigma =\big(W=\{W_\alpha\cdot \sigma\}_{\alpha\in A},\leqslant^\sigma\big)
\]
where $\leqslant^\sigma$ is induced by the total orders of $\Gamma^{W\cdot \sigma}_{s\cdot \sigma}:=\{W_\alpha\cdot \sigma|s\cdot \sigma\in W_\alpha\cdot \sigma\}$. We also define the functor 
\[
\begin{array}{rccc}
\mathcal{W} : &\Fin\op  & \longrightarrow &\Set\op  \\
& S & \longmapsto & \coprod_{n\in\N^*} \mathcal{W}_n(S)
\end{array}~~.
\]
Let  $(W=\{W_\alpha\}_{\alpha\in A},\leqslant)$ be a wall in $\mathcal{W}(S)$. We define on $W$  \emph{the equivalence relation of connectedness} $\overset{conn.}{\sim}$: for two elements $a$ and $b$ of $A$, we say $W_a\overset{conn.}{\sim} W_b$ if there exist an integer $n\geqslant 2$  and a sequence $W_0,W_1,\ldots,W_{n-1},W_n$ of elements of $W$ with $W_0=W_a$ and $W_n=W_b$ such that, for all $i$ in $[\![0,n-1]\!]$, 
\[
W_i\cap W_{i+1}\ne \varnothing \ \mbox{and} \ (W_i,W_{i+1})\in \mathrm{Succ}(W) \ \mbox{or} \ (W_{i+1},W_{i})\in \mathrm{Succ}(W).
\]
\begin{defi}[Projection $\mathcal{K}$] \label{def::projection_K}
	We define the natural projection $\mathcal{K}$ as follows: for a finite set $S$, we have
	\[
	\begin{array}{cccc}
	\mathcal{K}_S : & \mathcal{W}(S) & \longrightarrow & \mathcal{Y}(S) \subset \mathcal{W}(-) \\
	& W & \longmapsto & \left\{ \left. \bigcup_{B_\alpha\in\pi^{-1}([B])} B_\alpha \  \right| \ [B]\in \pi(W) \right\}
	\end{array} \ ,
	\]
	where $\pi$ is the projection of $W$ to its quotient by $\overset{conn.}{\sim}$.
\end{defi} 

We also have the subfunctor $\Wall_n \hookrightarrow \mathcal{W}_n$ of connected walls 
$\Wall_n: \Fin\op  \longrightarrow \Fin\op $ which is given, for all finite sets $S$, by 
\[
\Wall_n(S):=
\left\{ 
	\big( W=\{W_\alpha\}_{\alpha\in A},\leqslant\big) 
	\in \mathcal{W}_n(S) 
 \left| 
\mathcal{K}_S(W)=\{S\} 
\right. \right\} \ .
\]
We also define the functor 
\[
\begin{array}{rccc}
\Wall : &\Fin\op  & \longrightarrow &\Set\op  \\
& S & \longmapsto & \coprod_{n\in\N^*} \Wall_n(S)
\end{array}~~.
\]
An element $W$ of $\Wall(S)$ is called a connected \emph{wall} over $S$, and an element of a wall $W$ is called a \emph{brick} of $W$. Hence, we have other important subfunctors of $\Wall$: for all finite sets $S$, we have
\begin{itemize} 
\item The functor $\mathcal{Y}:\Fin\op\rightarrow\Set\op$ given by
\[
	\mathcal{Y}(S):=
		\left\{
		\big( \{K_\alpha\}_{\alpha\in A},\leqslant\big) \in \mathcal{W}(S) 
		\ \left| \
		\forall s\in S, \ |\{K_\alpha \ | \ s\in K_\alpha \}| =1		
	\right. \right\} \ ;	
\]
an element of $\mathcal{Y}(S)$ is a non-ordered partition of $S$;
\item The functor $\mathcal{X}=\mathcal{Y}\times \mathcal{Y}$, given by
\[
	\mathcal{X}(S)\cong
	\left\{
		\big( \{K_\alpha\}_{\alpha\in A},\leqslant\big) \in \mathcal{W}(S) 
		\ \left| \
		\forall s\in S, \ |\{K_\alpha \ | \ s\in K_\alpha \}| =2		
	\right. \right\} \ ;
\]
an element $K$ of $\mathcal{X}(S)$ is an ordered pair of unordered partitions of the finite set $S$, so we also denote by $(I,J)$ such a $K$;
\item the functor $\mathcal{X}^{\mathrm{conn}}$, which is a subfunctor of $\mathcal{X}$, given by
\[
	\mathcal{X}^{\mathrm{conn}}(S):=
	\left\{
		\big( \{K_\alpha\}_{\alpha\in A},\leqslant\big) \in \Wall(S) 
		\ \left| \
		 \forall s\in S, \ |\{K_\alpha \ | \ s\in K_\alpha \}| =2		
		\right. 
	\right\} .
\]
\end{itemize}
The functor $\mathcal{X}^{\mathrm{conn}}$ encodes a new monoidal structure on the category of reduced $\mathfrak{S}$-modules, the connected composition product, as we will see in \Cref{def::prod_connexe_Smod}.
\subsection{Monoidal structures and the induction functor}
We have three monoidal structures on $\Smodred_k$.
\begin{defiprop}[Composition product]
The \emph{composition product} is the bifunctor
\[
	- \square - : \Smodred_k \times \Smodred_k \longrightarrow \Smodred_k
\]
defined, for $P$, $Q$ two reduced $\mathfrak{S}$-modules and $S$ a finite set, by
\[
	\big(P\:\square\: Q\big) (S) :=  P(S) \otimes Q(S).
\]
This bi-additive bifunctor gives $\Smodred$ a symmetric monoidal structure, with identity $I_\square$, defined, for all non empty sets $S$, by $I_\square(S):=k$ concentrated in degree $0$.
\end{defiprop}

\begin{defiprop}[Concatenation product]
The \emph{concatenation product} is the bifunctor 
\[
-\otimes^{conc}- : \Smodred_k \times \Smodred_k \longrightarrow \Smodred_k
\]
defined, for all finite sets $S$ and all reduced $\mathfrak{S}$-modules $P$ and $Q$, by:
\[
\big(P\otimes^{conc} Q\big) (S) := 
\underset{{\substack{
			S',S''\in \Ob\Fin \\ 
			S'\amalg S''=S
}}}{\bigoplus} P(S') \otimes Q(S'').
\]
This product is symmetric monoidal without unit (since we are working with reduced $\mathfrak{S}$-modules).
\end{defiprop}
\begin{nota}\label{nota::S}
	We denote by $\mathbb{S}$, the functor which sends a reduced $\mathfrak{S}$-module $V$ to the free symmetric monoid \emph{without unit} $\mathbb{S}(V)$ for the concatenation product (see \cite[\Cref{ProtoI-sect::free_monoid}]{Ler18i}).
\end{nota}
\begin{rem}
We can extend the concatenation product:
\begin{equation}\label{eq::extension_concatenation} 
-\otimes^{\mathrm{conc}}-:\Smod_k \times \Smodred_k \longrightarrow \Smodred_k.
\end{equation}
This extension is induced by the equivalence of categories 
\[
\Smod_k\cong \Ch_k\times\Smodred_k,
\]
by the injection $(-)^{\mathfrak{S}}:\Ch_k  \hookrightarrow  \Smod_k $ defined, for all chain complexes $C$ and all finite sets $S$, by
\[
(C)^{\mathfrak{S}}(S):=\left\{
\begin{array}{cc}
C & \mbox{if } S=\varnothing, \\
0 & \mbox{otherwise;}
\end{array}\right.
\]
and by the action of the category $\Ch_k $ on $\Smodred_k$ defined, for all chain complexes $C$ and all finite sets $S$, by 
\[
\big(C\otimes^{\mathrm{conc}} V\big)(S):= C \otimes V(S).
\]
This extension allows us to define the suspension of a $\mathfrak{S}$-module.
\end{rem}
\begin{defi}[Suspension of a $\mathfrak{S}$-module (see {\cite[\Cref{ProtoI-def::suspension_Smod}]{Ler18i}}]\label{def::suspension_Smod}
	Let $\Sigma$ (respectively $\Sigma^{-1}$) be the chain complex $k$ concentrated in degree $1$ (resp. in degree $-1$). For $V$ a reduced $\mathfrak{S}$-module, the \emph{suspension of $V$} (resp. \emph{desuspension of $V$}) is the reduced $\mathfrak{S}$-module  $\Sigma V \overset{not.}{:=} \Sigma\otimes^{\mathrm{conc}} V$ (resp. $\Sigma^{-1} V \overset{not.}{:=} \Sigma^{-1}\otimes^{\mathrm{conc}} V$).
\end{defi}

\begin{defiprop}[Connected composition product of $\mathfrak{S}$-modules (see {\cite[\Cref{ProtoI-def::prod_connexe_Smod}]{Ler18i}})] \label{def::prod_connexe_Smod}
	The \emph{connected composition product} of reduced $\mathfrak{S}$-modules is the bifunctor 
	\[
	- \boxtimes_c - : \Smodred_k \times \Smodred_k \longrightarrow \Smodred_k
	\]
	defined, for all reduced $\mathfrak{S}$-modules $P$, $Q$ and for all non empty finite sets $S$, by:	  
	\[
	~P \boxtimes_c Q (S):= \underset{(I,J)\in \mathcal{X}^{\mathrm{conn}}(S)}{\bigoplus}~\underset{\alpha}{\bigotimes}~P(I_\alpha) \otimes \underset{\beta}{\bigotimes}~Q(J_\beta),
	\]
	where $\bigotimes_\alpha~P(I_\alpha) \otimes \bigotimes_\beta~Q(J_\beta)$ with $(I,J)$ in $\mathcal{X}^{\mathrm{conn}}(S)$ is the notation for  
	\[
	\bigotimes_{i=1}^r P(I_i) \otimes \bigotimes_{j=1}^s Q(J_j)\bigg/_\sim
	\] where the relation $\sim$ identifies $(p_1\otimes\ldots\otimes p_r)\otimes(q_1\otimes\ldots\otimes q_s)$ with
	\[
	(-1)^{|\sigma(p)|+|\tau(q)|}(p_{\sigma^{-1}(1)}\otimes\ldots\otimes p_{\sigma^{-1}(r)})\otimes(q_{\tau^{-1}(1)}\otimes\ldots\otimes q_{\tau^{-1}(s)})
	\]
	for all $\sigma$ in $\mathfrak{S}_r$, $\tau$ in $\mathfrak{S}_s$ with $(-1)^{|\sigma(p)|}, (-1)^{|\tau(q)|}$, the Koszul signs induced by permutations. We also denote by $\Ibox$, the $\mathfrak{S}$-module given by
	\[
	\Ibox(S) \overset{\mathrm{def}}{:=} 
	\left\{\begin{array}{cl}
	k & \mbox{if } |S|=1, \\
	0 & \mbox{otherwise},
	\end{array}\right. 
	\]
	which is the unit of the product $\boxtimes_c$. The category $(\Smodred_k,\boxtimes_c,\Ibox)$ is a (symmetric) monoidal category. The monoids for this product are called \textbf{protoperads}. 
\end{defiprop}
We have a compatibility between these monoidal structures.
\begin{prop}[Compatibility between monoidal structures (see {\cite[\Cref{ProtoI-prop::S_permute_prod_connexe}]{Ler18i}})]
	\label{prop::S_permute_prod_connexe}
	Let $P$ and $Q$ be two reduced $\mathfrak{S}$-modules. There is a  natural isomorphism of $\mathfrak{S}$-modules: 
	\[
	\mathbb{S}(P \boxtimes_c Q) \cong \mathbb{S}P \:\square\: \mathbb{S}Q.
	\]
	In particular, for  a protoperad $\calP $, the $\mathfrak{S}$-module $\mathbb{S}\calP $ is a monoid for the product $\square$.
\end{prop}
We have a notion of free protoperad. The combinatorics of the free protoperad is described by the functor of connected walls $\Wall$.

\begin{prop}[Free protoperad {(see \cite[\Cref{ProtoI-prop::proto_libre}]{Ler18i})}]\label{prop::freeproto}
	Let $V$ be a reduced $\mathfrak{S}$-module and $\rho$ be a positive integer. There exists  a free protoperad on $V$, denoted by $\scrF(V)$. For a finite set $S$, there is an isomorphism of weight-graded right $\mathrm{Aut}(S)$-modules, given on each weight $\rho$, by
	\[
	\scrF^\rho(V)(S) \cong \bigoplus_{{\substack{(\{K_\alpha\}_{\alpha\in A},\leqslant)\\ \in \mathcal{W}^{\mathrm{conn}}_\rho(S)}}} \bigotimes_{\alpha\in A} V(K_\alpha),
	\]
	where $\Wall : \Fin\op  \rightarrow \Set\op $ is the weight-graded functor of connected walls. The functor $\scrF$ is the left adjoint to the forgetful functor
	\[
	\scrF: \Smodred_k \rightleftarrows \mathsf{protoperads}_k : \mathrm{For}.
	\]
\end{prop}

The notion of protoperad is compatible with the notion of properad, defined by Vallette in  \cite{Val03,Val07} and \cite{MV09I}, via the induction functor.

\begin{defiprop}[Properad -- Free properad (see \cite{Val03,Val07})]
	The category of reduced $\mathfrak{S}$-bimodules, i.e. the category of functors $P:\Fin\times \Fin\op \rightarrow \Ch_k$ such that, for all finite set $S$, $P(S,\varnothing)=0=P(\varnothing,S)$, is monoidal for the connected composition product denoted by $\boxtimes_c^{\mathrm{Val}}$. The monoids for this product are called \emph{properads}. We have the \emph{free properad functor}, which is denoted by $\scrF\Val$ which is the left adjoint to the forgetful functor:
	\[
		\scrF\Val:\Sbimodred_k\rightleftarrows \mathsf{properads}_k : \mathrm{For}.
	\]
%
\end{defiprop}
We define a monoidal adjunction between these categories.
\begin{defiprop}[Induction functor (see {\cite[\Cref{ProtoI-defi::foncteur Ind}]{Ler18i}})]\label{def::functor induction}
	We define the \emph{induction functor} $\Ind : \Smodred_k \longrightarrow \Sbimodred_k$ which is given, for all reduced $\mathfrak{S}$-modules $V$ and, for all finite sets $S$ and $E$, by:
	\begin{align*}
		\big(\Ind \;V \big)(S,E)\cong & 
		\left\{ \begin{array}{cl}
		0 & \mbox{ if } S\not\cong E \\ 
		k[\mathrm{Aut}(S)]\otimes V(S)  & \mbox{ otherwise.}
		\end{array} \right.
		~~\\
	\end{align*}
	This functor is exact, has a right adjoint which is the \emph{functor of restriction} $\Res$, and is monoidal. 
	Hence, we have the functor
	\[
		\Ind : \mathsf{protoperads} \longrightarrow \mathsf{properads}.
	\]
	Moreover, the induction functor commutes with the free monoid constructions, formally by adjunction, i.e. we have the natural isomorphism of reduced $\mathfrak{S}$-bimodules:
	\[
		\Ind(\scrF(-)) \cong \scrF^{\mathrm{Val}}(\Ind(-)).
	\]
	Then, for a protoperad $\calP $ defined by generators and relations, i.e. $\calP =\scrF(V)/\langle R \rangle$, the properad $\Ind(\calP )$ is given by
	\[
		\Ind(\calP )\cong \scrF^{\mathrm{Val}}(\Ind V) /\langle \Ind R \rangle.
	\]
\end{defiprop}

\subsection{Shuffle protoperad}\label{sect::shuffle protoperad}
We denote $\Ord$, the category of totally ordered finite sets, with bijections. As in \cite[\Cref{ProtoI-sect::functors of walls}]{Ler18i}, we define  the combinatorial functors, $\mathcal{Y}\sh:\Ord\op  \rightarrow \Set\op $ and $\mathcal{X}\sh:\Set\op  \rightarrow \Fin\op $ which encode \emph{shuffle protoperads}. The shuffle framework corresponds to choosing a representative for each wall. We define the functor $\mathcal{Y}\sh$ as follows: for all finite, totally ordered  sets $S$, we set
\[
	\mathcal{Y}\sh_R(S) := \left\{ I=(I_j)_{j\in [\![1,R]\!]} ~\left|~
	\begin{array}{c}
		\forall r\in [\![1,R]\!], \ I_r \ne \varnothing ; \
		\cup_{r\in [\![1,R]\!]} I_r =S \\
		\forall r\ne s \in [\![1,R]\!], I_r \cap I_s = \varnothing \\
		\mathrm{min}(I_1)<\mathrm{min}(I_2)<\ldots<\mathrm{min}(I_R)
	\end{array} 
	\right.\right\}.
\]
and $\mathcal{Y}\sh(S) := \coprod_{r\in\N^*}\mathcal{Y}\sh_r(S)$; we also have
\[
	\mathcal{X}\sh_r(S):= \coprod_{i+j=r} \mathcal{Y}\sh_i(S) \times \mathcal{Y}\sh_j(S)~~\mbox{ and }~~\mathcal{X}\sh(S):=\coprod_{r\in\N^*} \mathcal{X}\sh_r(S).
\]
We have the following natural isomorphism of functors $\mathcal{X}\sh \cong \mathcal{Y}\sh\times\mathcal{Y}\sh$.
\begin{lem}\label{lem::combinatoire_oubli}
	We have the following commutative diagrams of functors up to natural isomorphims
	\[
	\begin{tikzcd}[column sep=small]
		\Ord\op \ar[rr,"(-)^f"] \ar[rd, bend right=15, "\mathcal{Y}\sh"'] & \ar[d,phantom,"\commu" description] & \Fin\op \ar[ld, bend left=15, "\mathcal{Y}"]\\
		& \Set\op &
	\end{tikzcd}
	\ \mathrm{and} \
	\begin{tikzcd}[column sep=small]
		\Ord\op \ar[rr,"(-)^f"] \ar[rd, bend right=15, "\mathcal{X}\sh"'] & \ar[d,phantom,"\commu" description] & \Fin\op \ar[ld, bend left=15, "\mathcal{X}"]\\
		& \Set\op &
	\end{tikzcd}
	\]
	where $\mathcal{X}$ and $\mathcal{Y}$ are the functors defined in \Cref{subsect::combinatorial_functors} (see also \cite[\Cref{ProtoI-sect::functors of walls}]{Ler18i}).
\end{lem}

\begin{defi}[Projection $\mathcal{K}\sh$] 
	We define the projection $\mathcal{K}\sh$ as follows: for a totally ordered finite set $(S,<_S)$, we have
	\[
	\begin{array}{cccc}
	\mathcal{K}\sh_S : & \mathcal{X}\sh(S) & \longrightarrow & \mathcal{Y}\sh(S) \\
	& W & \longmapsto & \left\{ \left. \bigcup_{B_\alpha\in\pi^{-1}([B])} B_\alpha \  \right| \ [B]\in \pi(W) \right\}
	\end{array} \ ,
	\]
	where $\pi$ is the projection of $W$ to its quotient by $\overset{conn.}{\sim}$ (cf. \cite[\Cref{ProtoI-subsect::connected wall}]{Ler18i}), and the set $\mathcal{K}\sh_S(W)$ is totally ordered by the order $<$ defined as follows: for $B_\alpha$ and $B_\beta$ in $\mathcal{K}\sh_S(W)$, $B_\alpha<B_\beta$ if $\mathrm{min}(B_\alpha) <_S \mathrm{min}(B_\beta)$.  
\end{defi} 

As for $\mathcal{K}$ (cf.  \cite[\Cref{ProtoI-lem::K_associative}]{Ler18i}), the product $\mathcal{K}\sh$ on $\mathcal{Y}\sh$ is associative. Let $M$ and  $S$ be two totally ordered finite sets. Every monotone injection $j : M \hookrightarrow S$ induces a morphism
\[
	\iota_j : \mathcal{Y}\sh(M) \longrightarrow \mathcal{Y}\sh(S)
\]
such that, for all $W=\{W_a\}_{a\in[\![1,r]\!]}$ in $\mathcal{Y}_r\sh(M)$, 
\[
	\iota_j(W) = \big\{ j(W_a) \big\}_{a\in[\![1,r]\!]}\amalg \coprod_{s\in S\backslash j(M)} \{s\}.
\]
Let $M,N$ and $S$ be three totally ordered finite sets and $\phi$ be the diagram of monotone injections $\phi :=\big(i:M\hookrightarrow S \hookleftarrow N:j\big)$ such that
\[
	\left\{
	\begin{array}{rl}
		\mathrm{im}(i)\cup \mathrm{im}(j) & = S \\
		\mathrm{im}(i)\cap\mathrm{im}(j) & \ne \emptyset
	\end{array}
	\right.~
\]
then, we have the product
\[
\mu_\phi : \mathcal{Y}\sh(M)\times \mathcal{Y}\sh(N) \longrightarrow \mathcal{Y}\sh(S),
\]
given by the union of the images by $i$ and $j$ of the partitions of $M$ and $N$, extended by singletons, i.e. defined by the following composition
\[
\begin{tikzcd}
	\mathcal{Y}\sh(M)\times \mathcal{Y}\sh(N) 
	\ar[r, "\iota_i\times \iota_j"]
	\ar[rd, dotted, bend right=10, "\mu_\phi:="']
	&  \mathcal{Y}\sh(S)\times \mathcal{Y}\sh(S) 
	\ar[d, "\mathcal{K}\sh_S"'] \\
	& \mathcal{Y}\sh(S)
\end{tikzcd} \ .
\]
We have the following commutative diagram:
\[
\begin{tikzcd}
	\mathcal{X}\sh(M)\times \mathcal{X}\sh(N) 
	\ar[r,"\K\sh_{M}\times \K\sh_{N}"] 
	\ar[d,"\mu_\phi^{\times 2}"'] & 
	\mathcal{Y}\sh(M)\times \mathcal{Y}\sh(N) \ar[d,"\mu_\phi"] \\
	\mathcal{X}\sh(S) \ar[r,"\K\sh_{S}"'] & \mathcal{Y}\sh(S)
\end{tikzcd}\ .
\]	
Finally, we define the functor $\mathcal{X}^{\mathrm{conn,sh}} : \Ord\op  \rightarrow \Fin\op $, for all totally ordered finite sets $S$, by 
\[
\mathcal{X}^{\mathrm{conn,sh}}(S):=\Big\{ (I,J)\in \mathcal{X}\sh(S)~\big|~ \mathcal{K}_{S}\sh(I,J)=\big\{S\big\} \Big\}.
\]

\begin{defiprop}[Connected shuffle product]
	The \textit{connected shuffle product} is the bifunctor
	\[
		-\boxtimes_c^\shuffle - : \Func(\Ord\op,\Ch_k) \times \Func(\Ord\op ,\Ch_k) \longrightarrow \Func(\Ord\op ,\Ch_k)
	\]
	defined, for two objects $P$ and $Q$ of $\Func(\Ord\op,\Ch_k)$ and a finite totally ordered $S$, by
	\[
		\big(P \boxtimes_c^\shuffle Q\big) (S):= \bigotimes_{(K,L)\in \mathcal{X}^{\mathrm{conn,sh}}(S)} \bigotimes_{i=1}^{m} P(K_i) \otimes \bigotimes_{j=1}^n Q(L_j).
	\]
\end{defiprop}
\begin{prop}\label{lem::Smodsh_ana_scinde}
	The product $\boxtimes_c^\shuffle$ is associative. Also, for all objects $A$ and $B$ in the category $\Func(\Ord\op ,\Ch_k)$, the endofuncteur
	\[	
	\begin{array}{rccc}
	\Phi_{A,B} : &\Func(\Ord\op ,\Ch_k) & \longrightarrow & \Func(\Ord\op ,\Ch_k) \\
	& X & \longmapsto & A \boxtimes_c^\shuffle X \boxtimes_c^\shuffle B
	\end{array}
	\]
	is split analytic in the sense of \cite{Val07,Val09}. The category  
	\[
	\big(\Func(\Ord\op ,\Ch_k), \boxtimes_c^\shuffle, \Ibox\big)
	\]
	is an abelian (symmetric) monoidal category and the monoidal product preserves  reflexive coequalizors and sequential colimits. 
\end{prop}
\begin{proof}
	Similar to the proof of \cite[\Cref{ProtoI-lem::Smod_ana_scinde}]{Ler18i} and   \cite[\Cref{ProtoI-prop::prop_fond_du_produit_connexe}]{Ler18i}.
\end{proof}

\begin{defi}[Shuffle protoperad]
	 The monoids of $\big(\Func(\Ord\op ,\Ch_k),\boxtimes_c^\shuffle,\Ibox\big)$ are called \emph{shuffle protoperads}, and we denote $\mathsf{protoperads}^{\mathrm{sh}}$, the category of shuffle protoperads.
\end{defi}

The forgetful functor $\underline{(-)}: \Ord \rightarrow \Fin$ induces the functor $(-)\sh$ from  $\Smodred_k$ to $\Func(\Ord\op ,\Ch_k)$.

\begin{prop}\label{prop::oubli_monoidal_sym}
	The functor 
	\[
	(-)\sh:\big(\Smodred_k,\boxtimes_c,\Ibox\big) \longrightarrow \big(\Func(\Ord\op ,\Ch_k),\boxtimes_c^\shuffle\Ibox\big)
	\]
	is (strongly) monoidal and is exact. In particular, it preserves quasi-isomorphisms.
\end{prop}
 
\begin{proof}
	Let $S$ be a totally ordered finite set and $P$ and $Q$ be two reduced $\mathfrak{S}$-modules. We have the following isomorphisms: 
	\begin{eqnarray*}
		\big(P\boxtimes_c Q\big)\sh(S) &= & \ \big(P\boxtimes_c Q\big)(\underline{S}) \\
		&\cong & \ \bigoplus_{{\substack{(K,L) \\ \in \mathcal{X}^{\mathrm{conn}}(\underline{S})}}} \bigotimes_{\alpha\in A} P(K_\alpha) \otimes \bigotimes_{\beta\in B} Q(L_\beta) \\
		& \cong & \ \bigoplus_{{\substack{(\widetilde{K},\widetilde{L})\\ \in \mathcal{X}^{\mathrm{conn,sh}}(S)}}}
		P(\widetilde{K}_1) \otimes \ldots \otimes P(\widetilde{K}_m) \otimes Q(\widetilde{L}_1)\otimes \ldots \otimes Q(\widetilde{L}_n)\\
		&\cong &~\big(P\sh\boxtimes_c^\shuffle Q\sh \big)(S).
	\end{eqnarray*}
	Consider a short exact sequence of reduced $\mathfrak{S}$-modules 
	\[
		0 \rightarrow P \rightarrow Q \rightarrow R \rightarrow 0,
	\]
	then, for all finite ordered sets $S$, the  following
	\[
		0 \rightarrow P(\underline{S}) \rightarrow Q(\underline{S}) \rightarrow R(\underline{S}) \rightarrow 0
	\]
	is a short exact sequence. So the functor $(-)\sh$ is exact.
\end{proof}

As for the case of protoperads, we have a combinatorial description of the free shuffle protoperad.
\begin{defiprop}
	Let $V$ be a functor in $\Func(\Ord\op ,\Ch_k)$. The \emph{free shuffle protoperad} $\scrF_{\shuffle}(V)$ is given, for all totally ordered sets $S$, by: 
	\[
	\scrF_{\shuffle}(V)(S) := \bigoplus_{W\in \Wall(\underline{S})} \bigotimes_{\alpha\in A} V(W_\alpha).
	\]
	The functor $\scrF_\shuffle$ is the left adjoint to the forgetful functor
	\[
	\scrF_{\shuffle}: \Func(\Ord\op ,\Ch_k) \rightleftarrows \mathsf{shuffle-protoperads}_k : \mathrm{For}.
	\]
\end{defiprop}
By \Cref{prop::oubli_monoidal_sym}, we have the following.
\begin{cor}\label{cor::shuffle}
	Let $V$ be a reduced $\mathfrak{S}$-module. There is a natural isomorphism of shuffle protoperads
	\[
	\mathscr{F}_\shuffle(V\sh) \cong \big(\mathscr{F}(V) \big)\sh.
	\]
	Also, for $\langle R \rangle \subset \mathscr{F}(V)$ an ideal, $\langle R \rangle\sh$ is an ideal of $\mathscr{F}_\shuffle(V\sh)$, and there is a natural isomorphism of shuffle protoperads
	\[
	\mathscr{F}_\shuffle(V\sh)/ \langle R \rangle\sh \cong \big(\mathscr{F}(V)/ \langle R \rangle \big)\sh.
	\]
\end{cor}

\section{Koszul duality of protoperads}\label{sect::Koszul duality}

In this section, we adapt the  constructions of \cite[Sect. 3]{MV09I} and \cite{Val03,Val07} for properads to the protoperadic framework.

\subsection{(Co)Augmentation, infinitesimal (co)bimodule and (co)derivation}
\begin{defi}[Augmented protoperad]
	An \emph{augmentation} of a protoperad $\calP$ is a morphism of protoperads $\epsilon : \calP \rightarrow \Ibox$ , where $\Ibox$ is the unit of the product $\boxtimes_c$. A protoperad with an augmentation is called \emph{augmented}. We denote by $\mathsf{protoperads}_k^{\mathrm{aug}}$, the category of augmented protoperads. To an augmented protoperad $(\calP,\epsilon)$, we associate its \emph{augmentation ideal} $\overline{\calP}$, defined as the kernel of the augmentation $\epsilon$,  i.e. $\overline{\calP}:=\mathrm{Ker}(\epsilon)$.
\end{defi}
For two reduces $\mathfrak{S}$-modules $M$ and $P$, the $\mathfrak{S}$-module $P\boxtimes_c (P\oplus M)$ has a weight-grading, which we denote
\[
P\boxtimes_c (P\oplus M)= \bigoplus_{r\in\N} \left( P\boxtimes_c (P\oplus M) \right) ^{(r)_M}.
\]
Let $(\calP,\epsilon)$ be an augmented protoperad. Then, we have the isomorphism of reduced $\mathfrak{S}$-modules $\calP\cong \Ibox\oplus \overline{\calP}$. Moreover, by the bigrading given by \cite[\Cref{ProtoI-lem::bigrading_proto}]{Ler18i}, we can decompose the connected composition product  

	\[  
		\mu=\bigoplus_{(r,s)\in(\N^*)^2} \mu^{(r,s)}  \ \mathrm{by} \ 
		\mu^{(r,s)} : \big((\Ibox \oplus \overline{\calP})\boxtimes_c (\Ibox\oplus \overline{\calP})\big)^{(r,s)_{\overline{\calP}}} \longrightarrow \overline{\calP}.
	\]

\begin{defi}[Partial composition product]\label{def::prod_compo_partiel}
	Let $(\calP,\epsilon)$ be an augmented protoperad. The \emph{partial composition product} is the restriction of the product $\mu:\calP\boxtimes_c \calP\rightarrow \calP$ to
	\[
		\mu^{(1,1)} :  \big((\Ibox \oplus \overline{\calP})\boxtimes_c (\Ibox \oplus \overline{\calP})\big)^{(1,1)_{\overline{\calP}}} \longrightarrow \overline{\calP}.
	\]
\end{defi}

Using the partial composition, we introduce the notion of an infinitesimal bimodule over a protoperad.

\begin{defi}[Infinitesimal bimodule]
	Let $(\calP,\mu)$ be a protoperad. A $\mathfrak{S}$-module $M$ is  a $\calP$-\emph{infinitesimal bimodule} if $M$ has two morphisms of  $\mathfrak{S}$-modules, respectively called the \emph{left and right actions}:
	 \[
	 	\lambda : \big(\calP\boxtimes_c (\calP\oplus M)\big)^{(1)_M} \rightarrow M ~~\mbox{ and }~~ \rho : \big((\calP\oplus M)\boxtimes_c \calP\big)^{(1)_M} \rightarrow M
	 \] 
	 such that the following compatibility diagrams commute: 
	 \begin{enumerate}
	 	\item associativity of the left action $\lambda$:
	 		\[
	 			\begin{tikzcd}
	 				\big(\calP  \boxtimes_c \calP  \boxtimes_c (\calP \oplus M)\big)^{(1)_M} 
	 				\ar[r, "\calP \boxtimes_c (\lambda+ \mu)"] 
	 				\ar[d, "\mu\boxtimes_c (\calP \oplus M)"'] 
	 				\ar[rd,phantom, "\commu" description]
	 				&  \big(\calP  \boxtimes_c (\calP \oplus M)\big)^{(1)_M} 
	 				\ar[d, "\lambda"]\\
	 				\big( \calP  \boxtimes_c (\calP \oplus M)\big)^{(1)_M} 
	 				\ar[r, "\lambda"']
	 				& M 
	 			\end{tikzcd} \ ;
	 		\]
	 	\item associativity of the right action $\rho$:
			\[
	 			\begin{tikzcd}
	 				\big((\calP \oplus M) \boxtimes_c \calP  \boxtimes_c \calP \big) ^{(1)_M} 
	 				\ar[r, "(\rho+\mu)\boxtimes_c \calP"] 
	 				\ar[d,"(\calP  \oplus M)\boxtimes_c \mu"']
	 				\ar[rd, phantom, "\commu" description]
	 					&\big((\calP \oplus M)\boxtimes_c \calP \big) ^{(1)_M} \ar[d,"\rho"]\\
	 				\big((\calP \oplus M)  \boxtimes_c \calP \big) ^{(1)_M} 
	 				\ar[r,"\rho"'] 
	 				& M
	 			\end{tikzcd} \ ;
	 		\]
	 	\item the left and right actions commute:
	 		\[
	 			\begin{tikzcd}
	 				\big(\calP  \boxtimes_c(\calP \oplus M)\boxtimes_c \calP \big) ^{(1)_M} 
	 				\ar[r,"(\lambda+\mu)\boxtimes_c \calP"]
	 				\ar[d, "\calP \boxtimes_c(\rho+\mu)"']
	 				\ar[rd,phantom, "\commu" description]
	 					& \big((\calP  \oplus M)\boxtimes_c \calP \big)^{(1)_M} \ar[d,"\rho"] \\
	 				\big(\calP  \boxtimes_c (\calP \oplus M)\big)^{(1)_M} \ar[r,"\lambda"']& M
	 			\end{tikzcd} \ .
	 		\]
	 \end{enumerate}
\end{defi}
\begin{rem}
	We also have the dual definitions of co-augmented coprotoperad, partial coproduct and infinitesimal cobimodule.(see \cite{MV09I} for properadic definition). 
\end{rem}

\begin{rem}
	For an augmented protoperad $(\calP, \mu ,\epsilon:\calP \rightarrow \Ibox)$, the following definition is equivalent to the data of two actions $\lambda' : (\calP \boxtimes_c M_+)^{(1)_{\overline{\calP }},(1)_M}\rightarrow M$ and $\rho' :(M_+\boxtimes_c \calP )^{(1)_M,(1)_{\overline{\calP }}}\rightarrow M$ where $M_+=M\oplus \Ibox$, compatible with the product of the protoperad $\calP $. In fact, if we consider the left action on $M$,  then the injection $\Ibox \hookrightarrow \calP $ induces the following morphism of $\mathfrak{S}$-modules
	\[
		\begin{tikzcd}
		\big(\calP \boxtimes_c M_+\big)^{(1)_{\overline{\calP }}(1)_M}  \ar[dr, dotted, bend left=15, "\lambda'"] \ar[d, hook]&	\\
		\big(\calP \boxtimes_c (\calP \oplus M)\big)^{(1)_M} \ar[r,"\lambda"']  & M 
		\end{tikzcd} \ ,
	\]
	compatible with the product $\mu$ of $\calP $. Conversely, if we consider a $\mathfrak{S}$-module $M$ with a morphism 
	\[
		\lambda' : \big(\calP \boxtimes_c (M_+)\big)^{(1)_{\overline{\calP }}(1)_M} \longrightarrow M
	\]
	compatible with the product $\mu$ of $\calP $, i.e. the following diagram commutes:
	\[
		\begin{tikzcd}[column sep=2.5cm]
			\big(\calP  \boxtimes_c \calP  \boxtimes_c M_+\big)^{(1)_{\overline{\calP }},(1)_{\overline{\calP }},(1)_M} 
			\ar[r,"\calP \boxtimes_c (\lambda'+ \mu^{(1,1)})"]
			\ar[d, "\mu^{(1,1)}\boxtimes_c M_+"']
			\ar[rd, phantom, "\commu" description]
			&  \big(\calP  \boxtimes_c M_+ \big)^{(1)_{\overline{\calP }},(1)_M} \ar[d,"\lambda'"]\\
			\big( \calP  \boxtimes_c M_+\big)^{(1)_{\overline{\calP }},(1)_M} \ar[r,"\lambda'"']& M
	 	\end{tikzcd} \ .
	\]
	This compatibility and associativity of the product $\mu$ allows the extension of the morphism $\lambda'$ to a morphism $ \lambda : \big(\calP \boxtimes_c (\calP \oplus M)\big)^{(1)_M} \rightarrow M$, which is the expected morphism. We have a similar equivalence for $\rho$ and $\rho'$.
\end{rem}
For the definition of infinitesimal bimodule in the properadic case, which is similar to the protoperadic case, the reader can refer to \cite{MV09I}.
\begin{lem}
	Let $\calP $ be a protoperad and $M$ be an infinitesimal $\calP $-bimodule. The $\mathfrak{S}$-bimodule $\Ind (M)$ is an infinitesimal $\Ind (\calP )$-bimodule.
\end{lem}

\begin{proof}
 	The functor $\Ind $ is monoidal for the products $\boxtimes_c$ and $\otimes^{conc}$ (see \cite[\Cref{ProtoI-prop_ind_mon_sym_comp}, \Cref{ProtoI-thm::Ind_monoidal_connexe}]{Ler18i}) and is additive, i.e. $ \Ind (V\oplus W) \cong \Ind (V)\oplus \Ind (W)$, so preserves the weight grading:
 	\[
 		\Ind \Big(\big(\calP \boxtimes_c (\calP \oplus M)\big)^{(1)_M}\Big) \cong \big(\Ind \calP \boxtimes_c (\Ind \calP \oplus \Ind M)\big)^{(1)_{\Ind M}}.
 	\]
\end{proof}

\begin{defi}[Derivation, coderivation]
	Let $(\calP ,\epsilon)$ be an augmented protoperad and $(M,\lambda,\rho)$ be an infinitesimal $\calP $-bimodule. A morphism of $\mathfrak{S}$-modules $d: \calP  \rightarrow M$ of homological degree $n$ is called a \emph{homogemeous derivation} if the following diagram commutes:
	\[
		\xymatrix@C=4pc{
			\big( \calP \boxtimes_c \calP  \big)^{(1,1)}\ar[r]^{\mu^{(1,1)}} \ar[d]_{d\boxtimes_c \calP  + \calP \boxtimes_c d}
				& \calP  \ar[d]^d \\
			M\boxtimes_c \calP  \oplus \calP \boxtimes_c M \ar[r]_(0.7){\rho + \lambda} & M
		}
	\]
	i.e., for all $p$ and $q$ in $\calP $:
	\[
		d\circ \mu^{(1,1)}(p,q) = \rho(d(p),q) + (-1)^{n|p|}\lambda(p,d(q)).
	\]
	We denote $\sDer_n(\calP ,M)$, the $k$-module of derivations from $\calP $ to $M$ of homological degree $n$ and the derivation complex is denoted by $\sDer_\bullet(\calP ,M)$, with the differential $[\partial,-]$ defined, for $\delta$ in $\sDer_n(\calP ,N)$, by $[\partial,\delta]:=\partial_N\circ\delta -(-1)^{|n|} \delta\circ \partial_P$.

	Let $(\calC ,\nu)$ be a coaugmented coprotoperad and $(N,\lambda,\rho)$, an infinitesimal $\calC $-cobimodule. A morphism  of $\mathfrak{S}$-modules $d: N \rightarrow \calC $ of homological degree $n$ is a \emph{homogeneous coderivation} if the following diagram commutes:
	\[
		\begin{tikzcd}[column sep= 2cm]
			N\ar[r,"d"] \ar[d,"\lambda+\rho"']& \calC   \ar[d,"\Delta^{(1,1)}"] \\
			 (N\boxtimes_c \calC )\oplus (\calC \boxtimes_c N) \ar[r, "d\boxtimes_c \calC + \calC \boxtimes_c d"'] & \calC \boxtimes_c \calC 
		\end{tikzcd} \ .
	\]
	We denote $\underline{\mathrm{Coder}}_n(\calC ,N)$, the $k$-module of homogeneous coderivations from $\calC $ to $N$ of degree $n$ and $\underline{\mathrm{Coder}}_\bullet(\calC ,N)$, the coderivation complex.
\end{defi}

\begin{prop}\label{prop::Ind_deriv_coderiv}
	Let $(\calP ,\epsilon)$ be an augmented protoperad, $(\calC ,\nu)$ be a coaugmented coprotoperad, $M$ be an infinitesimal $\calP $-bimodule and $N$ be an infinitesimal $\calC $-cobimodule. We have the following natural isomorphisms:
	\[
		\Ind \Big( \sDer^\bullet(\calP ,M) \Big) \cong \sDer^{\bullet}\big(\Ind (\calP ),\Ind (M)\big);
	\]
	\[
		\Ind \Big(\underline{\mathrm{Coder}}^\bullet(\calC ,N)\Big)\cong \underline{\mathrm{Coder}}^\bullet\big(\Ind (\calC ),\Ind (N)\big).
	\]
\end{prop}
\begin{proof}
	The functor $\Ind $ is additive monoidal and respects the grading on $M$ (see \Cref{def::functor induction}, or \cite[\Cref{ProtoI-thm::Ind_monoidal_connexe}]{Ler18i}). 
\end{proof}

\begin{lem}
	Let $\scrF(V)$ be the free protoperad on the $\mathfrak{S}$-module $V$. For a homogeneous morphism $\theta : V \rightarrow \scrF(V)$ of degree $|\theta|$, there exists a unique homogeneous derivation $d_\theta : \scrF(V) \rightarrow \scrF(V)$ of degree $|\theta|$,  such that its restriction to $V$ is $\theta$: we have
	\[
		\sDer_{|\theta|}(\scrF(V),\scrF(V))\cong \Hom_{|\theta|}(V,\scrF(V)).
	\]
	Moreover, if $\theta(V)\subset \scrF^{(\rho)}(V)$ then we have $d_\theta(\scrF^{(s)}(V))\subset \scrF^{(s+\rho-1)}(V)$.
\end{lem}
\begin{proof}
	Let $\bigotimes_{j=1}^{n} (v_1^j \otimes \ldots \otimes v_{r_j}^j)$ be a representative  of a class of $V_n=(V\oplus \Ibox)^{\boxtimes_c  n}$ with each $v_\lambda^j$ in  $V\oplus \Ibox$. We define the application $d_\theta$ by $	d_\theta\Big( \bigotimes_{j=1}^{n} (v_1^j \otimes \ldots \otimes v_{r_j}^j)\Big):=$ 
	\[
		\sum_{{\substack{
					s\in[\![1,n]\!] \\ 
					i\in[\![1,r_s]\!]
		}}}\!(-1)^{\lambda_{s,i}} 
	\bigotimes_{j=1}^{s-1} (v_1^j \otimes \ldots \otimes v_{r_j}^j)
	\otimes (v_1^s \otimes \ldots \otimes \theta( v_{i}^s)\otimes \ldots\otimes v_{r_s}^s) \otimes  \bigotimes_{j=s+1}^{n} (v_1^j \otimes \ldots \otimes v_{r_j}^j)
	\]
	where $\lambda_{s,i}=\big(\sum_{j=1}^{s-1}\sum_{l=1}^{r_j} |v_l^j| + |v_1^s|+\ldots+|v_{i-1}^s|\big) |\theta|$ and where we extend $\theta$ to $V\oplus \Ibox$  by $\theta(\Ibox)=0$. The morphism $d_\theta$ is constant on the equivalence class  of $\bigotimes_{j=1}^{n} (v_1^j \otimes \ldots \otimes v_{r_j}^j)$. We just need to verify that for $n=1$ and for the transposition which sends $v_j$ to $v_{j+1}$: 
	\begin{align*}
		d_\theta \big( v_1 \otimes & \ldots\otimes v_{j+1}\otimes v_j\otimes \ldots \otimes v_{r} \big) = \\
		= &~ \sum_{ i\in[\![1,j-1]\!]} (-1)^{|\theta|\sum_{l=1}^{i-1}|v_l|} v_1 \otimes \ldots \otimes \theta( v_{i})\otimes \ldots \otimes  v_{j+1}\otimes v_j\otimes \ldots \otimes v_{r}  \\
		&~+(-1)^{|\theta|\sum_{l=1}^{j-1}|v_l|}v_1 \otimes \ldots \otimes \theta(v_{j+1})\otimes v_j\otimes \ldots \otimes v_{r}  \\
		&~+(-1)^{|\theta|\sum_{l=1}^{j-1}|v_l|+|\theta||v_{j+1}|} v_1 \otimes \ldots  \otimes  v_{j+1}\otimes \theta( v_j)\otimes \ldots \otimes v_{r} \\
		&~+\sum_{ i\in[\![j+2,r]\!]} (-1)^{|\theta|\sum_{l=1}^{i-1}|v_l|} v_1 \otimes \ldots \otimes  v_{j+1}\otimes v_j\otimes \ldots\otimes \theta( v_{i})\ldots\otimes v_{r} \\
		\sim (-1)^{|v_j||v_{j+1}|}&\bigg( \sum_{ i\in[\![1,j-1]\!]} (-1)^{|\theta|\sum_{l=1}^{i-1}|v_l|} v_1 \otimes \ldots \otimes \theta( v_{i})\otimes \ldots \otimes  v_{j}\otimes v_{j+1}\otimes \ldots \otimes v_{r}  \\
		&~+(-1)^{|\theta|\sum_{l=1}^{j-1}|v_l|+|\theta||v_j|}v_1 \otimes \ldots \otimes v_j \otimes \theta(v_{j+1})\otimes \ldots \otimes v_{r}  \\
		&~+(-1)^{|\theta|\sum_{l=1}^{j-1}|v_l|} v_1 \otimes \ldots  \otimes \theta( v_j) \otimes v_{j+1} \otimes \ldots \otimes v_{r} \\
		&~+\sum_{ i\in[\![j+2,r]\!]} (-1)^{|\theta|\sum_{l=1}^{i-1}|v_l|} v_1 \otimes \ldots \otimes  v_{j}\otimes v_{j+1}\otimes \ldots\otimes \theta( v_{i})\ldots\otimes v_{r} \bigg) \\
		= &~(-1)^{|v_j||v_{j+1}|}d_\theta \big( v_1 \otimes \ldots\otimes v_{j}\otimes v_{j+1}\otimes \ldots \otimes v_{r} \big).
	\end{align*}

Moreover, $d_\theta$ factorizes through $\widetilde{V}_n$ (see \cite[\Cref{ProtoI-subsect::monoide_libre}, \cref{ProtoI-eq::def_VnTilde}]{Ler18i} for the definition); similarly, we show that, on the elements of the form $\bar{v}:=(v_1\otimes v_2 \otimes v_3)\otimes (v_1'\otimes 1 \otimes v_3') - (-1)^{|v_2|(|v_3|+|v_1'|)} (v_1\otimes 1 \otimes v_3)\otimes (v_1'\otimes v_2 \otimes v_3') $, we have $d_\theta(\bar{v})=0$. Hence, we use the same arguments that the properadic case (cf. \cite[Lem. 87]{Val03}): the surjectivity of the product of the free protoperad $\scrF(V)$ gives us the uniqueness of the derivation $d_\theta$ and that all derivation are as above.
\end{proof}

Dually we have the following lemma (which is the protoperadic analogue of \cite[Lem. 88]{Val03}).

\begin{lem}\label{lem::coderivation}
	Let $\scrF^c(V)$ be the connected cofree coprotoperad on the $\mathfrak{S}$-module $V$. For all homogeneous morphisms of $\mathfrak{S}$-module $\theta : \scrF^c(V) \rightarrow V$ of homological degree $|\theta|$, there exists a unique homogeneous coderivation with the same degree $d_\theta : \scrF^c(V) \rightarrow \scrF^c(V)$ such that the composition
	\[
		\scrF^c(V) \overset{d_\theta}{\longrightarrow} \scrF^c(V) \overset{\mathrm{proj}}{\longrightarrow} V
	\]
is equal to $\theta$. This correspondance is bijective; moreover, if $\theta$ is null on each weight component $\scrF^{c,(s)}(V)$ for $s\ne r$, then $d_\theta(\scrF^{c,(s+r-1)}(V) ) \subset \scrF^{c,(s)}(V) $, for $s>0$.
\end{lem}

\begin{defi}[Quasi-free protoperad/coprotoperad] 
	A protoperad $(\scrF(V),\partial=\partial_V+d_\theta)$ (resp. coprotoperad $(\scrF^c(V),\partial=\partial_V+d_\theta)$) is called \emph{quasi-free}. 
\end{defi}
\begin{prop}
	The projection $\scrF(V) \rightarrow V$ of a quasi-free (co)protoperad on to its indecomposables is a morphism of $\mathfrak{S}$-modules if and only if $\theta(V)\subset \bigoplus_{r\geqslant 2} \scrF^{(r)}(V)$.
\end{prop}
\begin{proof}
	cf. \cite[Prop. 89]{Val03}.
\end{proof}


\subsection{Bar-cobar adjunction}\label{subsect::adjonction_barcobar}
We introduce the bar construction of a protoperad. 
We denote by $s$ the generator of the $\mathfrak{S}$-module $\Sigma$, the suspension (see \Cref{def::suspension_Smod}). Let $(\calP ,\mu, \epsilon)$ be an augmented protoperad. The partial product $\mu^{(1,1)}$ of $\calP $ induces a homogeneous morphism of $\mathfrak{S}$-modules of homological degree $-1$:
	\[
		s\mu_2 : \scrF^{c,(2)}(\Sigma \overline{\calP })\longrightarrow \Sigma \overline{\calP }
	\]
given by
	\[
		s\mu_2\big( s.p_1 \otimes  s.p_2 \big):=~ (-1)^{|p_1|}s.\mu^{(1,1)}\big( p_1 \otimes  p_2 \big).
	\]

By  \Cref{lem::coderivation}, we can associate to $s\mu_2$, a homogeneous  coderivation $d_{s\mu_2} : \scrF^c(\Sigma \overline{\calP }) \rightarrow \scrF^c(\Sigma \overline{\calP })$, of homological degree $-1$. We consider the coderivation $ \partial:= \partial_P+d_{s\mu_2} : \scrF^{c}(\Sigma \overline{\calP }) \rightarrow \scrF^{c}(\Sigma \overline{\calP })$ with $\partial_P$ the coderivation induced by the internal differential of $\calP $. We show that $\partial^2=0$, which is equivalent to showing that $\partial_P  d_{s\mu_2}+d_{s\mu_2} \partial_P +d_{s\mu_2}^2=0$, because $\partial_P$ is a differential. By \Cref{prop::Ind_deriv_coderiv}, $\Ind (d_{s\mu_2})$ is a coderivation of homological degree $-1$ (in the properadic sense). As the functor $\Ind$ commutes with the free monoid functor $\scrF(-)$, with the suspension and $\Ind$ is exact (so commutes with the functor $\overline{(-)}$), we have the following isomorphism
\[
	\Ind \Big(\big( \scrF^c(\Sigma\overline{\calP }), \partial_P+d_{s\mu_2}\big)\Big)\cong \big( \scrF^c(\Sigma\overline{\Ind (\calP )}), \partial_{\Ind (\calP )}+\Ind (d_{s\mu_2})\big).
\]
As the coderivation $d_{s\mu_2}$ is the suspension of the partial product $\mu^{(1,1)}$, and the functor $\Ind$ is compatible with the weight-bigrading in $\calP $  of $\calP \boxtimes_c \calP $ and commutes with the suspension, we have directly that $\Ind (d_{s\mu_2^\calP })$ is equal to $d_{s\mu_2^{\Ind (\calP )}}$, the coderivation induced by the partial product of the properad $\Ind (\calP )$.

This lends to the definition of the bar construction of a protoperad.
\begin{defiprop}[Bar construction]
	Let $(\calP ,\mu,\partial_\calP, \epsilon)$ be an augmented protoperad. The \emph{bar construction of $\calP $} is the following quasi-cofree coaugmented coprotoperad: 
	\[
		\big(\Bar \calP,\partial\big):=\big( \scrF^c(\Sigma\overline{\calP }), \partial_P+d_{s\mu_2}\big),
	\]
	which gives the functor
	\[
		\Bar  : \mathsf{protoperads}\aug_k \longrightarrow \mathsf{coprotoperads}\coaug_k \ .
	\]
	Moreover, the respective bar constructions commute with the induction functor:
	\[
		\Ind \big( \Bar (-)\big) \cong \Bar \Val\big( \Ind (-)\big),
	\]
	where the functor $\Bar \Val$ is the bar construction for properads defined in \cite{Val03,Val07}.
\end{defiprop}

\begin{prop}
	Let $V$ be a reduced $\mathfrak{S}$-module concentrated in homological degree $0$. Then the homology of the chain complex given by the bar construction of the free protoperad over $V$ is acyclic, i.e.
	\[
	\mathrm{H}_\bullet\big(\Bar \scrF(V), \partial_{\mathrm{bar}} \big) \cong  \Sigma V,
	\]
	where $\Sigma$ is the shifting of homological degree one.
\end{prop}
\begin{proof}
	For this proof, we use the notion of \emph{colouring} of a wall $W$ and the  \emph{colouring complex} associated to $W$, defined in \cite[\Cref{ProtoI-sect::coloring}]{Ler18i}. Let $S$ be a totally ordered finite set. We have the following isomorphisms of chain complexes:
	\[
	\big(\Bar \scrF(V)\big)(S)  \cong \bigoplus_{K\in\Wall(S)}~\bigoplus_{\phi\in Col(K)}~\Sigma^{\#\phi(K)}\Big(\bigotimes_{\alpha\in A} V(K_\alpha),\phi \Big)
	\]
	where the differential on the right side acts on the colouring as in the colouring complex. So we have:
	\begin{align*}
	\big(\Bar \scrF(V)\big)(S)  \cong &~\bigoplus_{K\in\Wall(S)}~\bigoplus_{\phi\in Col(K)}~\Sigma^{\#\phi(K)}\Big(\bigotimes_{\alpha\in A} V(K_\alpha),\phi \Big) \\
	\cong &~\Sigma V \oplus \bigoplus_{{\substack{K\in\Wall(S)\\ \#K\geqslant 2}}}~\bigoplus_{\phi\in Col(K)}~\Sigma^{\#\phi(K)}\Big(\bigotimes_{\alpha\in A} V(K_\alpha),\phi \Big) \\
	\cong &~\Sigma V \oplus \bigoplus_{{\substack{K\in\Wall(S)\\ \#K\geqslant 2}}} \big( C^{\Colo}_\bullet(K)\big)^{\oplus (\#K)\cdot \mathrm{dim}_kV}
	\end{align*}	
	then, by \cite[\Cref{ProtoI-prop::cpx_colo_acyclique}]{Ler18i}, $\Bar \scrF(V) \simeq \Sigma V$.
\end{proof}

We also have the cobar construction.
\begin{defiprop}[Cobar construction]
	Let $(\calC ,\Delta,\partial_\calC,\nu)$ be a coaugmented coprotoperad. The \emph{cobar construction of $\calC $} is the following quasi-free augmented protoperad:
	\[
		\big(\Cobar \calC ,\partial\big):=\big( \scrF(\Sigma^{-1}\overline{\calC }), \partial_C+d_{s^{-1}\Delta_2}\big),
	\]
	which gives the functor:
	\[
		\Cobar  : \mathsf{coprotoperads}\coaug_k \longrightarrow \mathsf{protoperads}\aug_k \ .
	\]
	Moreover, the respective cobar constructions commute with the induction functor:
	\[
		\Ind \big( \Cobar (-)\big) \cong \Cobar \Val\big( \Ind (-)\big),
	\]
	where the functor $\Cobar \Val$ is the cobar construction for properads defined in \cite{Val03,Val07}.
\end{defiprop} 
By the exactness of the functor $\Ind $, we directly have the adjunction between bar and cobar construction.
\begin{prop}
	The functors $\Bar $ and $\Cobar $ form a pair of adjoint functors:
	\[
		\xymatrix{
			\Cobar  : \mathsf{coprotoperads}\coaug_k \ar@<0.5ex>[r] 
				& \mathsf{protoperads}\aug_k : \Bar  \ar@<0.5ex>[l]
		} .
	\]
\end{prop}
\begin{proof}
By the properties of the functor $\Ind $ (see \cite[\Cref{ProtoI-prop::Ind_exact}, \Cref{ProtoI-prop::Ind_commute_F}]{Ler18i}) and \cite[Prop. 17]{MV09I}.
\end{proof}

\subsection{Koszul duality}\label{subsect::dualite_de_koszul}

The result of this section  are inspired by \cite[Chap. 7]{Val03}: as the results are very similar, we try to use the same notation as in \cite{Val03}.
\subsubsection{Definition of the Koszul dual}
Let $(\calP ,\mu,\epsilon)$ be an augmented protoperad, with a weight grading, $\calP =\bigoplus_{n\in\N}\calP ^{[n]}$. This grading induced a new one on the bar construction of $\calP $:
\[
	\Bar _{(r)}\calP  = \underset{\rho\in\N}{\bigoplus} \Bar _{(r)}(\calP) ^{[\rho]},
\]
where $\Bar _{(r)}\calP = \scrF^{(r)}(\Sigma \overline{\calP })$ is the grading described in \cite[\Cref{ProtoI-prop::proto_libre}]{Ler18i}. We interpret $r$ as the number of elements of $\calP $ and $\rho$ as the total weight induced by the weight of each element of $\calP $. As the product $\mu$ of $\calP $ respects the weight grading, $d_{s\mu_2}$ respects the induced grading on $\Bar \calP$; so we have
\[
	d_{s\mu_2}\Big(\Bar _{(r)}(\calP) ^{[\rho]}\Big) \subset \Bar _{(r-1)}(\calP )^{[\rho]}.
\]
Thus we have the following lemma.
\begin{lem}\label{lem::description_bar}
	Let $\calP $ (respectively $\calC$) be a weight-graded, connected, protoperad (resp. coprotoperad), i.e. $\calP ^{[0]}=\Ibox$ (resp. $\calC^{[0]})=\Ibox$). Then we have:
	\[
		\Bar _{(\rho)}(\calP )^{[\rho]}= \scrF^{c,(\rho)}(\Sigma \overline{\calP }^{[1]})
		~~\mbox{ and }~~ \Bar _{(r)}(\calP )^{[\rho]}= 0 \mbox{ for } r>\rho
	\]
	\[
		\left(\mbox{resp. }\Cobar ^{(\rho)}(\calC)^{[\rho]}= \scrF^{(\rho)}(\Sigma^{-1} \overline{\calC}^{[1]})
		~~\mbox{ and }~~ \Cobar ^{(r)}(\calC)^{[\rho]}= 0 \mbox{ for } r>\rho \right).
	\]
\end{lem}
\begin{proof}
	cf. \cite[Sect. 7.1]{Val03}.
\end{proof}
\begin{defi}[Koszul dual]
	Let $\calP $ (respectively $\calC$) be a  weight-graded, connected protoperad (resp. coprotoperad). We define the \emph{Koszul dual} of $\calP $ (resp. of $\calC$), denoted by $\calP ^\antish$ (resp. $\calC^\antish$) by the weight-graded $\mathfrak{S}$-module:
	\[
		\calP ^{\antish [\rho]}:= H_{(\rho)}\big(\Bar _*(\calP )^{[\rho]},d_{s\mu_2}\big) 
	\]
	\[
		\Big(\mathrm{resp.}~\calC^{\antish[\rho]}:= H_{(\rho)}\big(\Cobar _*(\calC)^{[\rho]},d_{s^{-1}\Delta_2}\big)\Big).
	\]
\end{defi}
By \Cref{lem::description_bar}, we have the equalities:
\[
	\calP ^{\antish[\rho]}=\mathrm{Ker}\Big(d_{s\mu_2}:\Bar _{(\rho)}(\calP )^{(\rho)} \rightarrow \Bar _{(\rho-1)}(\calP )^{(\rho)}\Big)
\]
and
\[
	\calC^{\antish [\rho]}=\mathrm{Coker}\Big(d_{s^{-1}\Delta_2}:\Cobar _{(\rho-1)}(\calC)^{(\rho)} \rightarrow \Cobar _{(\rho)}(\calC)^{(\rho)}\Big).
\]
Moreover, if the protoperad $\calP $ is concentrated in homological degree $0$, then we have
\[
	\big(\Bar _{(r)}(\calP )^{[\rho]}\big)_m =
	\left\{ \begin{array}{cl}
		\Bar _{(r)}(\calP )^{[\rho]} & \mbox{if } m=r,\\
		0 & \mbox{otherwise.}
	\end{array}\right. 
\]
The dual coprotoperad $\calP ^\antish$ is not concentrated in $0$ degree, but satisfies:
\[
	\big(\calP ^{\antish[\rho]}\big)_m =
	\left\{ \begin{array}{cl}
		\calP ^{\antish[\rho]} & \mbox{if } m=\rho,\\
		0 & \mbox{otherwise.}
	\end{array}\right. 
\]
\begin{prop}
	The functor $(-)^\antish : \mathsf{(co)protoperades}_k\aug \rightarrow \Smod_k^{gr}$ commutes with the functor $\Ind $.
\end{prop}
\begin{proof}
	By the exactness and the preservation of the weight grading of the functor $\Ind $ (see \Cref{def::functor induction}).
\end{proof}
We have a protoperadic equivalent of the proposition \cite[Prop. 136]{Val03}.
\begin{prop}
	Let $\calP = \bigoplus_n \calP ^{[n]}$  (resp. $\calC= \bigoplus_n \calC^{[n]}$) be a weight-graded, connected protoperad (resp. coprotoperad). Then the Koszul dual of $\calP $ is a sub weight-graded, connected, coaugmented coprotoperad of $\scrF^c(\Sigma \calP ^{[1]})$ (respectively, the Koszul dual of $\calC$ is a connected, weight-graded, augmented protoperad quotient of $\scrF(\Sigma^{-1}\calC^{[1]})$).
\end{prop}

\subsubsection{Koszul resolution}

\begin{defi}[Koszul protoperad, coprotoperad]
	Let $\calP $ and $\calC$ be respectively a protoperad and a coprotoperad, each weight-graded and connected. The protoperad $\calP $ is \emph{Koszul} if the inclusion $\calP ^\antish\hookrightarrow \Bar \calP $ is a quasi-isomorphism. Dually, the coprotoperad $\calC$ is \emph{Koszul} if the projection  $\Cobar \calC \twoheadrightarrow \calC^\antish$ is a quasi-isomorphism.
\end{defi}

\begin{prop}
	If $\calP $ is a  weight-graded, connected protoperad which is Koszul, then its dual $\calP ^\antish$ is a Koszul coprotoperad, and  $\calP ^{\antish\antish}=\calP $.
\end{prop}
\begin{proof}
	By the properties of the functor $\Ind $  (see \Cref{def::functor induction}) and by \cite[Prop.141]{Val03}.
\end{proof}

\begin{defi}[Koszul complex]
	Let $\calP $ be a weight-graded protoperad. The (right and left) \emph{Koszul complexes} of $\calP $ are the following complexes:
	\begin{enumerate}
	\item the complex $\big( \calP ^\antish\boxtimes_c \calP , \partial=\partial_P+d_{\Delta}^r\big)$, where the differential $d_{\Delta}^r$ is induced by the homogemeous morphism of homological degree $-1$:
	\[
		\xymatrix{
			\calP ^\antish \ar[r]^(0.4)\Delta & \calP ^\antish\boxtimes_c \calP ^\antish \ar@{->>}[r] & (\calP ^\antish\boxtimes_c\big (\Ibox\oplus \overline{\calP }^\antish)^{[1]}\big) \ar[r] & \calP ^\antish\boxtimes_c (\Ibox\oplus \calP ^{[1]})
		} ,
	\]
	where the right morphism is induced by the isomorphism  $(\overline{\calP }^\antish)^{[1]} \cong \overline{\calP }^{[1]}$ ;
	\item the complex $\big( \calP \boxtimes_c \calP ^\antish, \partial=\partial_P+d_{\Delta}^l\big)$, where the differential $d_{\Delta}^l$ is induced by the homogeneous morphism of degree $-1$:
	\[
		\xymatrix{
			\calP ^\antish \ar[r]^(0.4)\Delta & \calP ^\antish\boxtimes_c \calP ^\antish \ar@{->>}[r] & \big( (\Ibox\oplus \overline{\calP }^\antish)^{[1]}\boxtimes_c \calP ^\antish\big) \ar[r] &  (\Ibox\oplus \calP ^{[1]})\boxtimes_c \calP ^\antish
		}.
	\]
	\end{enumerate}
\end{defi}
As in the properadic case, we have the following Koszul criterion:
\begin{thm}[Koszul criterion]\label{thm::critere_Koszul}
	Let $\calP $ be a connected weight-graded protoperad. The following are equivalent:
	\begin{enumerate}
		\item the protoperad $\calP $ is Koszul;
		\item the inclusion $\calP ^\antish \hookrightarrow \Bar \calP $ is a quasi-isomorphism;
		\item the Koszul complex $\big( \calP ^\antish\boxtimes_c \calP , \partial=\partial_P+d_{\Delta}^r\big)$ is acyclic;
		\item the Koszul complex $\big( \calP \boxtimes_c \calP ^\antish, \partial=\partial_P+d_{\Delta}^l\big)$ is acyclic;
		\item the morphism of protoperads $\Cobar \calP ^\antish \rightarrow \calP $ is a quasi-isomorphism.
	\end{enumerate}
\end{thm}
\begin{proof}
By the exactness of the functor $\Ind $  (see \Cref{def::functor induction}) and theorems \cite[Th. 144, Th. 149]{Val03}.
\end{proof}
\begin{rem}
	By \Cref{cor::shuffle}, we have a bar-cobar adjunction and a Koszul duality for shuffle protoperads. Also, a properad $\calP $ is Koszul if and only if $\calP \sh$ is Koszul.
\end{rem}

\subsubsection{The case of quadratic protoperads}\label{sect::proto_quadratiques}
This subsection is strongly inspired by \cite[Sect. 2]{Val08} which described the notion of a quadratic properad. We adapt the notion to the protoperadic framework. Let $V$ be a $\mathfrak{S}$-module and $R\subset \scrF^{(2)}(V)$: such a pair $(V,R)$ is called a \emph{quadratic datum}.

As the underlying $\mathfrak{S}$-modules of the free protoperad $\scrF(V)$ and the cofree coprotoperad $\scrF^c(V)$  are isomorphic, we consider the following morphisms of $\mathfrak{S}$-modules:
\[
	R \hookrightarrow \scrF^{(2)}(V) \hookrightarrow \scrF(V)\overset{(*)}{\cong} \scrF^c(V) \twoheadrightarrow \scrF^{c,(2)}(V) \twoheadrightarrow \scrF^{c,(2)}(V) /R=:\overline{R},
\]
where the isomorphism $(*)$ is an isomorphism of $\mathfrak{S}$-modules. Using this, we naturally define a quotient protoperad of $\scrF(V)$ or a sub-coprotoperad of$\scrF^c(V)$.
\begin{defi}[Quadratic (co)protoperad]
	The \emph{(homogeneous) quadratic protoperad generated by $V$ and $R$} is the quotient protoperad of $\scrF(V)$ by the ideal generated by $R \subset \scrF^{(2)}(V)$. We denote this protoperad by $\mathscr{P}(V,R):=\scrF(V)/\langle R \rangle$. Dually, the \emph{(homogeneous) quadratic  coprotoperad generated by $V$ and $\overline{R}$} is the sub-coprotoperad  of $\scrF^c(V)$ generated by $\scrF^{c,(2)}(V)\twoheadrightarrow \overline{R}$. We denote this coprotoperad by $\mathscr{C}(V,\overline{R})$.
\end{defi}

\begin{rem}
	All quadratic protoperads $\mathscr{P}(V,R)$ and all quadratic coprotoperads  $\mathscr{C}(V,\overline{R})$ have a weight-grading by $V$, as for properads (see \cite[Prop. 55]{Val03}).
\end{rem}
\begin{thm}[Kosul dual $(-)^\antish$]\label{thm::duale_de_koszul}
	Let $(V,R)$ be a quadratic datum. We denote by  $\Sigma^2R$, the image of $R$ in $\scrF^{(2)}(\Sigma V)$ and $\Sigma^{-2}\overline{R}$, the quotient of  $\scrF^{c,(2)}(\Sigma^{-1} V)$ by $\Sigma^{-2}R$. The \emph{Koszul dual} of the protoperad $\mathscr{P}(V,R)$, denoted by $\mathscr{P}(V,R)^\antish$, is the coprotoperad given by
	\[
		\mathscr{P}(V,R)^\antish = \mathscr{C}(\Sigma V,\Sigma^2\overline{R}).
	\]
	Dually, the \emph{Koszul dual} of the coprotoperad $\mathscr{C}(V,\overline{R})$, denoted by $\mathscr{C}(V,\overline{R})^\antish$, is the protoperad given by
	\[
		\mathscr{C}(V,\overline{R})^\antish = \mathscr{P}(\Sigma^{-1}V,\Sigma^{-2}R).
	\]
	Also, we have $\mathscr{P}(V,R)^{\antish\antish} =\mathscr{P}(V,R)$ and $\mathscr{C}(V,\overline{R})^{\antish\antish} =\mathscr{C}(V,\overline{R})$.
\end{thm}
\begin{proof} 
It is a similar proof as \cite[Th. 8]{Val08}.
\end{proof} 
\begin{prop}\label{prop::dual_lineaire_koszul}
	Let $(V,R)$ be a locally finite quadratic datum, i.e. for all finite sets $S$, $V(S)$ has a finite dimension. The linear dual of the coprotoperad $\mathscr{C}(V,\overline{R})$ is the quadratic protoperad
	\[
		\big(\mathscr{C}(V,\overline{R})\big)^*=\scrF(V^*)/\langle R^\bot \rangle
	\]
	with $R^\bot \subset \scrF^{(2)}(V)^* \cong \scrF^{(2)}(V^*)$. In particular, we have
	\[
		\big(\mathscr{P}(V,R)^\antish\big)^* =\mathscr{P}(\Sigma^{-1}V^*,\Sigma^{-2}R^\bot).
	\]
\end{prop}
\begin{proof}
	It is the same proof as \cite[Cor.154]{Val03} or \cite[Prop. 9]{Val08}.
\end{proof}

\section{Simplicial bar construction for protoperads}\label{sect::simplicial}
We construct the simplicial bar complex for protoperads, as in the properadic case (see \cite[Sect 6]{Val07}). Recall that $\Sigma^n$ denotes the homological suspension of degree $n$ (see \Cref{def::suspension_Smod}).

\begin{defi}[(Reduced) Simplicial bar construction]
	Let $(\calP ,\mu,\eta)$ be a protoperad. We denote
	\[
		\mathbf{C}_n(\calP ):= \Sigma^n \Ibox\boxtimes_c\calP ^{\boxtimes_c n}\boxtimes_c \Ibox.
	\]
	The face maps $d_i:\mathbf{C}_n(\calP )\rightarrow 	\mathbf{C}_{n-1}(\calP )$ are induced by:
	\begin{itemize}
		\item the unit $\calP \boxtimes_c \Ibox \rightarrow \calP $ for $i=0$;
		\item the composition $\mu$ of the $i$-th and the $(i+1)$-th row,
		\item the unit $\Ibox\boxtimes_c \calP  \rightarrow \calP $ for $i=n$.
	\end{itemize}
The degeneracy maps $s_i:\mathbf{C}_n(\calP )\rightarrow 	\mathbf{C}_{n+1}(\calP )$ are given by the insertion of the unit $\eta:\Ibox \rightarrow \calP$ of the protoperad $s_i:=\Sigma^n\Ibox\boxtimes_c\calP ^{\boxtimes_c i}\boxtimes_c\eta\boxtimes_c \calP ^{\boxtimes_c n-i}\boxtimes_c \Ibox$. The differential $\partial_{\mathbf{C}}$ is defined  by
\[
\partial_{\mathbf{C}(\calP )}:= \partial_{\calP } + \sum_{i=0}^{n} d_i.
\]
One can check that $\partial_{\mathbf{C}(\calP )}^2=0$. This chain complex is called the (reduced)  \emph{simplicial bar construction} of $\calP $.
\end{defi}
\begin{defi}[Normalized bar construction]
The \emph{normalized bar construction} is given by the quotient of the simplicial bar construction by the image of the degeneracy maps. We denote by $\mathbf{N}(\calP)$ the following graded $\mathfrak{S}$-module, given in grading $n$, by:
\[
	\mathbf{N}_n(\calP):=
	\Sigma^n\mathrm{coker}\left(\sum_{i=0}^n\Ibox\boxtimes_c\calP ^{\boxtimes_c i}\boxtimes_c\eta\boxtimes_c \calP ^{\boxtimes_c n-i}\boxtimes_c \Ibox \right).
\]
\end{defi}
We define the functor $\mathcal{W}^{\mathrm{conn,lev}}_{n\uparrow}:\Fin\op\rightarrow \Fin\op$ of $n$-level connected wall given, for all finite set $S$, by
\[
\mathcal{W}^{\mathrm{conn,lev}}_{n\uparrow}(S)=
\left\{
W=(W^1,\ldots,W^n) \left| 
\begin{array}{l}
\forall i\in [\![1,n]\!], \exists I\in \mathcal{Y}(S) \\
\varnothing \ne W^i =\{W^i_\alpha\}_{\alpha\in A_i}\subseteq I\\
\cup_{i=1}^n \cup_{\alpha\in A_i}  W^i_\alpha =S; \mathcal{K}_S(W) =\{S\}
\end{array} 
\right.\right\}
\]
where $\mathcal{K}_S$ is the natural projection defined in \Cref{def::projection_K} (see also \cite[\Cref{ProtoI-def::projection_K}]{Ler18i}).

We denote the label of the number of levels by ${n\uparrow}$, because $\mathcal{W}^{\mathrm{conn,lev}}_{n\uparrow}(S)$ is also weight-graded by the number of bricks: an element $(W^1,\ldots,W^n)$ lives in $\mathcal{W}^{\mathrm{conn,lev}}_{n\uparrow,b}(S)$ with $b=|W^1|+\ldots +|W^n|$. 
The graded functor of level connected wall is denoted by $\mathcal{W}^{\mathrm{conn,lev}}=\amalg_n \mathcal{W}^{\mathrm{conn,lev}}_{n\uparrow}:\Fin\op \rightarrow \Set\op$. We have also the natural projection of \emph{unlevelization} 
\[
\mathrm{unl}:\mathcal{W}^{\mathrm{conn,lev}} \twoheadrightarrow \Wall 
\]
which sends an element $(\{W^1_\alpha\}_{\alpha \in A_1},\ldots,\{W^n_\alpha\}_{\alpha \in A_n})$ in $\mathcal{W}^{\mathrm{conn,lev}}(S)$  to the connected wall $W$ over $S$ which contains $W^i_\alpha$, for all $i$ in $[\![1,n]\!]$ and all $\alpha$ in $A_i$ and such that, for all $s$ in $S$,  the total order of $\Gamma_s=\{W^i_\alpha|s\in W^i_\alpha\}$ is defined by levels: for $W^i_\alpha$ and $W^j_\beta$ in $\Gamma_s$,  $W^i_\alpha <_s W^j_\beta$ if $i<j$.

Remark that the unlevelization morphism projects the functor of $n$-leveled connected wall with $n$ bricks  to $\Wall_n$. We denote by $\pi_{n\uparrow}$ the restriction of the unlevelization morphism to $\mathcal{W}^{\mathrm{conn,lev}}_{n\uparrow, n}$
\[
	\pi_{n\uparrow} : \mathcal{W}^{\mathrm{conn,lev}}_{n\uparrow, n} \twoheadrightarrow \Wall_n  
\]

\begin{prop}\label{prop::norm_simp_bar}
	Let $\calP$ be an augmented protoperad, and its augmentation ideal $\overline{\calP}$, and $S$ be a finite set. We have the following isomorphism
	\[
	\mathbf{N}_n(\calP)(S)\cong 
	\Sigma^n
	\bigoplus_{{\substack{
				(W^1,\ldots,W^n)\\
				\in \mathcal{W}^{\mathrm{conn,lev}}_{n\uparrow}(S)
			}}}  
	\bigotimes_{i=1}^n 	\bigotimes_{\alpha \in A} \overline{\calP}(W^i_\alpha).
	\]

\end{prop}

\begin{proof}[Proof of \Cref{prop::norm_simp_bar}]
	As in the properadic case (see \cite[Section 6.1.3, the first remark]{Val07}). An element $W=(W^1, \ldots,W^n)$ in $\mathcal{W}^{\mathrm{conn,lev}}_{n\uparrow}(S)$ describes the position of non-trivial elements in each level. In the definition of the normalized bar construction, the cokernel ensures that there is a non-trivial element in each level: this is the condition $W^i\ne \varnothing$. The conditions $	\cup_{i=1}^n \cup_{\alpha\in A_i}  W^i_\alpha =S$ and $\mathcal{K}_S(\mathrm{unl}(W)) =\{S\}$ ensure that we have the connectedness of the product.
\end{proof}
%
\begin{prop}
	The simplicial bar construction and the normalized simplicial bar construction commute with the induction functor $\Ind$:
	\[
	\Ind (\mathbf{C}(-)) =\mathbf{C}\Val(\Ind(-))~\mathrm{and}~\Ind (\mathbf{N}(-)) =\mathbf{N}\Val(\Ind(-)),
	\]
	where the functors $\mathbf{C}\Val$ and $\mathbf{N}\Val$ are respectively, the reduced simplicial bar construction and the normalized simplicial bar construction for properads (see \cite{Val07}).
\end{prop}
\begin{proof}
	The functor $\Ind$ is monoidal and exact (see \Cref{def::functor induction}).
\end{proof}
\begin{prop}
	\begin{enumerate}
	\item The simplicial bar construction and the normalized bar construction preserve quasi-isomorphisms.
	\item Let $\calP$ be a quasi-free protoperad on a weight-graded $\mathfrak{S}$-module $V$, i.e. $\calP$ has underlying $\mathfrak{S}$-module $\scrF(V)$, such that $V^{(0)}=0$ and concentrated in homological degree $0$. The natural projection $\mathbf{N}(\calP)\rightarrow \Sigma V$ is a quasi-isomorphism.
	\end{enumerate}
\end{prop}
\begin{proof}
	\begin{enumerate}
		\item As for the bar construction of protoperads. The functors $\mathrm{Res}$, $\Ind$ and  $\mathbf{C}\Val$, (cf. \cite[Prop 6.1]{Val07}) preserve quasi-isomorphims. Let $\phi$ be a quasi isomorphism of $\mathfrak{S}$-modules, then $\mathrm{Res}(\mathbf{C}\Val(\Ind(\phi)))=\mathrm{Res}\circ\Ind\mathbf{C}(\phi)=\mathbf{C}(\phi)$ is a quasi-isomorphism.
		\item Similar to \cite[Prop. 6.5]{Val07}
	\end{enumerate}
\end{proof}
We define the levelization morphism as in the operadic and the properadic case (see \cite[Section 6.2]{Val07}).
\begin{defiprop}\label{def::levelizationmorphism}
	Let $\calP$ be an augmented protoperad. The \emph{levelization morphism} is the injective morphism of $\mathfrak{S}$-modules 
	\[
		e: \Bar (\calP) \longrightarrow \mathbf{N}(\calP)
	\]
	which, for a finite set $S$, and a wall $W=\{W_\alpha\}_{\alpha \in A}$ in $\Wall(S)$, sends 
	\[
		\bigotimes_{\alpha \in A} \Sigma \overline{\calP}(W_\alpha)
		\overset{e}{\longmapsto}
		\bigoplus_{\widetilde{W}\in \pi_{n\uparrow}^{-1}(W)}
		\bigotimes_{i=1}^{|A|} \Sigma \overline{\calP}(\widetilde{W}_i)
		\subset \mathbf{N}(\calP);
	\]
	the map $e$ sends each element of $\bigotimes_{\alpha \in A} \Sigma \overline{\calP}(W_\alpha)$ to the sum of representatives (with signs induced by the  Koszul sign of the symmetry).
\end{defiprop}

\begin{thm}\label{thm::levelization}
	Let $\calP $ be a weight-graded augmented protoperad. The levelization morphism
	$e: \Bar (\calP ) \rightarrow \mathbf{N}(\calP ) $	is a quasi-isomorphism.
\end{thm}
\begin{proof}
	Let $\calP$ be a weight-graded, augmented protoperad, and consider the levelization morphism $e:\Bar (\calP)\rightarrow\mathbf{N}(\calP)$. The induction functor sends $e$ to $e\Val$, the levelization morphism for properads, defined by Vallette in \cite[Section 6]{Val07}, which is a quasi isomorphism (see \cite[Theorem 6.7]{Val07}):
	\[
		\begin{tikzcd}[column sep=large]
		\Ind(\Bar \calP) \cong \Bar \Val\Ind(\calP) 
		\ar[r,"\Ind(e)=e\Val"',"\sim"]&
		\mathbf{N}\Val\Ind(\calP) \cong \Ind(\mathbf{N}\calP)
		\end{tikzcd}
	\]
	We apply the functor $\Res$ to this map, which is an exact functor, and which satisfies $\Res\circ \Ind=\id$, then the map $e$ is a quasi-isomorphism.
	We just use the same arguments that for the properadic case (see \cite[Sect. 6]{Val07}).
\end{proof}

\section{Studying Koszulness of binary quadratic protoperad}\label{sect::criterion}
In this section, we describe a criterion to study the Koszulness of binary quadratic protoperad, which are protoperads given by a quadratic datum $(V,R)$ such that $V$ is concentrated in arity $2$, $V(S)=0$ for all finite sets $S$ with $|S|\ne 2$. 

\subsection{A useful criterion} We give an algebraic criterion for a binary quadratic protoperads concentrated in homological degree $0$ to be Koszul.

\begin{defi}[Binary protoperad]
	A protoperad $\calP $, given by generators ad relations, i.e. $\calP =\scrF(V)/\langle R \rangle$ is \emph{binary} if the $\mathfrak{S}$-module $V$ is concentrated in arity $2$, i.e. $V(S)=0$ for all finite sets $S$ with $|S|\ne 2$. 
\end{defi}

Let $\calP $ be a binary quadratic protoperad concentrated in homological degree $0$, given by the quadratic datum $(V,R)$, then $V=V_0(2)$. We associate to $\calP $, a family of quadratic algebras $\{\scrA(\calP,n)\}_{n\geqslant 2}$, defined by
\[
\scrA(\calP,n):= \mathbb{S}(\calP \sh)([\![1,n]\!]).
\]
We will see that the algebras $\scrA(\calP,n)$ are quadratic. Fix $n\geqslant 2$, we consider the decomposition of $V$ in irreducible representations:
\[
V=\bigoplus_{\nu=1}^m V_\nu
\] 
where $V_\nu=k\cdot v_\nu$ is the trivial representation or the signature representation of $\mathfrak{S}_2$ (recall that the characteristic of $k$ is different to $2$). To $V$, we associate the set $V(\calP ,n)$ of generators of $\{\scrA(\calP,n)\}_{n\geqslant 2}$:
\[
V(\calP ,n):=\{(v_\nu)_{ij}~|~1\leqslant i<j\leqslant n, 1\leqslant \nu \leqslant m \}
\]
Thus $V(\calP ,n)$ corresponds to the generators of $\mathbb{S}(\calP )([\![1,n]\!])$ as algebra for the product $\mu_\square$ (see \Cref{prop::S_permute_prod_connexe}), i.e.
\begin{equation}\label{eq::correspondance}
(v_\nu)_{ij}\rightsquigarrow 
\begin{tikzpicture}[scale=0.3,baseline=0.1pc]
\draw[densely dotted] (-2,-0.2)  -- (-2,0.75) node[above] {$\scriptscriptstyle{1}$};
\draw (-1.15,0) node {\tiny{$\cdots$}};
\draw[densely dotted] (-0.5,-0.2) node[below] {} -- (-0.5,0.75);
\draw[densely dotted] (0.5,-0.2)  -- (0.5,0.75) node[above] {$\scriptscriptstyle{i}$};
\draw (2,0) node {\tiny{$\cdots$}};
\draw[densely dotted] (3.5,-0.2)  -- (3.5,0.75) node[above] {$\scriptscriptstyle{j}$} ;
\draw[fill=white] (0,0) rectangle (1,0.5);
\draw (0.5,0.25) node {$\scriptscriptstyle \nu$};
\draw[fill=white] (3,0) rectangle (4,0.5);
\draw (3.5,0.25) node {$\scriptscriptstyle \nu$};
\draw[densely dotted] (4.5,-0.2) -- (4.5,0.75);
\draw (5.25,0) node {\tiny{$\cdots$}};
\draw[densely dotted] (6,-0.2)  -- (6,0.75)node[above] {$\scriptscriptstyle{n}$};
\end{tikzpicture}
~.
\end{equation}
As $\calP $ is binary and quadratic, the set of relations $R$ is concentrated in arity $2$ and $3$.
Each relation in $R(2)$ is given by a linear combination of terms as 
\[
\begin{tikzpicture}[scale=0.3,baseline=0.1pc]
\draw[densely dotted] (0.5,-0.2)  -- (0.5,1.45)node[above] {$\scriptscriptstyle{1}$};
\draw[densely dotted] (1.5,-0.2)  -- (1.5,1.45)node[above] {$\scriptscriptstyle{2}$};
\draw[fill=white] (0,0) rectangle (2,0.5);
\draw[fill=white] (0,0.75) rectangle (2,1.25);
\end{tikzpicture}
\]
where each brick is labelled by a generator $v_\nu$. To a such relation $r$ in $R(2)$, we associate a family of quadratic relations $\{r_{ij}\}_{1\leqslant i<j \leqslant n}$ in terms of $V(\calP ,n)$, where $r_{ij}$ is given by replacing a monomial indexed by $v_\alpha$ for the bottom brick and $v_\beta$ for the upper brick, with $v_\alpha$ and $v_\beta$ two generators, by the monomial $(v_\alpha)_{ij}(v_\beta)_{ij}$ in $V(\calP ,n)^{\otimes 2}$, as in \Cref{fig::labelled_procedure2}.
\begin{figure}[!h]
\[
\begin{tikzpicture}[scale=0.3,baseline=0.1pc]
\draw[densely dotted] (0.5,-0.2)  -- (0.5,1.45)node[above] {$\scriptscriptstyle{1}$};
\draw[densely dotted] (1.5,-0.2)  -- (1.5,1.45)node[above] {$\scriptscriptstyle{2}$};
\draw[fill=white] (0,0) rectangle (2,0.5);
\draw (1,0.25) node {$\scriptscriptstyle \alpha$};
\draw[fill=white] (0,0.75) rectangle (2,1.25);
\draw (1,1) node {$\scriptscriptstyle \beta$};
\end{tikzpicture}
\overset{ij}{\rightsquigarrow} (v_\alpha)_{ij}(v_\beta)_{ij}.
\]
\caption{labelled procedure for $R(2)$}
\label{fig::labelled_procedure2}
\end{figure}
We denote by $R(2)_{ij}$, the set of relations in $V(\calP ,n)^{\otimes 2}$ which are obtained by the labelled procedure $\overset{ij}{\rightsquigarrow}$ (see \Cref{fig::labelled_procedure2}). Similarly, by connectivity, each relation in $R(3)$ is given by a linear combination of terms as follow:
\[
\begin{tikzpicture}[scale=0.3,baseline=0.1pc]
\draw[densely dotted] (0.5,-0.2)  -- (0.5,1.45)node[above] {$\scriptscriptstyle{1}$};
\draw[densely dotted] (1.5,-0.2)  -- (1.5,1.45)node[above] {$\scriptscriptstyle{2}$};
\draw[densely dotted] (2.5,-0.2)  -- (2.5,1.45)node[above] {$\scriptscriptstyle{3}$};
\draw[fill=white] (0,0) rectangle (1,0.5);
\draw[fill=white] (2,0) rectangle (3,0.5);
\draw[fill=white] (1,0.75) rectangle (3,1.25);
\end{tikzpicture}
\quad;\quad 
\begin{tikzpicture}[scale=0.3,baseline=0.1pc]
\draw[densely dotted] (0.5,-0.2)  -- (0.5,1.45)node[above] {$\scriptscriptstyle{1}$};
\draw[densely dotted] (1.5,-0.2)  -- (1.5,1.45)node[above] {$\scriptscriptstyle{2}$};
\draw[densely dotted] (2.5,-0.2)  -- (2.5,1.45)node[above] {$\scriptscriptstyle{3}$};
\draw[fill=white] (0,0.75) rectangle (1,1.25);
\draw[fill=white] (2,0.75) rectangle (3,1.25);
\draw[fill=white] (1,0) rectangle (3,0.5);
\end{tikzpicture}
\quad;\quad 
\begin{tikzpicture}[scale=0.3,baseline=0.1pc]
\draw[densely dotted] (0.5,-0.2)  -- (0.5,1.45)node[above] {$\scriptscriptstyle{1}$};
\draw[densely dotted] (1.5,-0.2)  -- (1.5,1.45)node[above] {$\scriptscriptstyle{2}$};
\draw[densely dotted] (2.5,-0.2)  -- (2.5,1.45)node[above] {$\scriptscriptstyle{3}$};
\draw[fill=white] (0,0) rectangle (2,0.5);
\draw[fill=white] (1,0.75) rectangle (3,1.25);
\end{tikzpicture}
\quad;\quad 
\begin{tikzpicture}[scale=0.3,baseline=0.1pc]
\draw[densely dotted] (0.5,-0.2)  -- (0.5,1.45)node[above] {$\scriptscriptstyle{1}$};
\draw[densely dotted] (1.5,-0.2)  -- (1.5,1.45)node[above] {$\scriptscriptstyle{2}$};
\draw[densely dotted] (2.5,-0.2)  -- (2.5,1.45)node[above] {$\scriptscriptstyle{3}$};
\draw[fill=white] (0,0.75) rectangle (2,1.25);
\draw[fill=white] (1,0) rectangle (3,0.5);
\end{tikzpicture}
\quad;\quad 
\begin{tikzpicture}[scale=0.3,baseline=0.1pc]
\draw[densely dotted] (0.5,-0.2)  -- (0.5,1.45)node[above] {$\scriptscriptstyle{1}$};
\draw[densely dotted] (1.5,-0.2)  -- (1.5,1.45)node[above] {$\scriptscriptstyle{2}$};
\draw[densely dotted] (2.5,-0.2)  -- (2.5,1.45)node[above] {$\scriptscriptstyle{3}$};
\draw[fill=white] (0,0.75) rectangle (1,01.25);
\draw[fill=white] (2,0.75) rectangle (3,1.25);
\draw[fill=white] (0,0) rectangle (2,0.5);
\end{tikzpicture}
\quad;\quad 
\begin{tikzpicture}[scale=0.3,baseline=0.1pc]
\draw[densely dotted] (0.5,-0.2)  -- (0.5,1.45)node[above] {$\scriptscriptstyle{1}$};
\draw[densely dotted] (1.5,-0.2)  -- (1.5,1.45)node[above] {$\scriptscriptstyle{2}$};
\draw[densely dotted] (2.5,-0.2)  -- (2.5,1.45)node[above] {$\scriptscriptstyle{3}$};
\draw[fill=white] (0,0) rectangle (1,0.5);
\draw[fill=white] (2,0) rectangle (3,0.5);
\draw[fill=white] (0,0.75) rectangle (2,1.25);
\end{tikzpicture}~~,
\]
where each brick is labelled by a generator $v_\nu$. If $n \geqslant 3$, for all relation $r$ in $R(3)$, we associate a family of quadratic relations $\{r_{ijk}\}_{1\leqslant i<j<k\leqslant n}$ with $r_{ijk}\in V(\calP ,n)^{\otimes 2}$, where $r_{ijk}$ is given by replacing all monomial indexed by $v_\alpha$ for the bottom brick and $v_\beta$ for the upper brick, with $v_\alpha$ and $v_\beta$ two generators, as in \Cref{fig::labelled_procedure3}.
\begin{figure}[!h]
\begin{minipage}{0.45\linewidth}
\begin{align*}
\begin{tikzpicture}[scale=0.3,baseline=0.1pc]
\draw[densely dotted] (0.5,-0.2)  -- (0.5,1.45)node[above] {$\scriptscriptstyle{1}$};
\draw[densely dotted] (1.5,-0.2)  -- (1.5,1.45)node[above] {$\scriptscriptstyle{2}$};
\draw[densely dotted] (2.5,-0.2)  -- (2.5,1.45)node[above] {$\scriptscriptstyle{3}$};
\draw[fill=white] (0,0) rectangle (1,0.5);
\draw[fill=white] (2,0) rectangle (3,0.5);
\draw[fill=white] (1,0.75) rectangle (3,1.25);
\end{tikzpicture}
~~ \overset{ijk}{\rightsquigarrow} ~~& (v_\alpha)_{ik} (v_\beta)_{jk} ~~;\\
\begin{tikzpicture}[scale=0.3,baseline=0.1pc]
\draw[densely dotted] (0.5,-0.2)  -- (0.5,1.45)node[above] {$\scriptscriptstyle{1}$};
\draw[densely dotted] (1.5,-0.2)  -- (1.5,1.45)node[above] {$\scriptscriptstyle{2}$};
\draw[densely dotted] (2.5,-0.2)  -- (2.5,1.45)node[above] {$\scriptscriptstyle{3}$};
\draw[fill=white] (0,0.75) rectangle (1,1.25);
\draw[fill=white] (2,0.75) rectangle (3,1.25);
\draw[fill=white] (1,0) rectangle (3,0.5);
\end{tikzpicture}
~~ \overset{ijk}{\rightsquigarrow} ~~& (v_\alpha)_{jk} (v_\beta)_{ik} ~~;\\
\begin{tikzpicture}[scale=0.3,baseline=0.1pc]
\draw[densely dotted] (0.5,-0.2)  -- (0.5,1.45)node[above] {$\scriptscriptstyle{1}$};
\draw[densely dotted] (1.5,-0.2)  -- (1.5,1.45)node[above] {$\scriptscriptstyle{2}$};
\draw[densely dotted] (2.5,-0.2)  -- (2.5,1.45)node[above] {$\scriptscriptstyle{3}$};
\draw[fill=white] (0,0) rectangle (2,0.5);
\draw[fill=white] (1,0.75) rectangle (3,1.25);
\end{tikzpicture}
~~ \overset{ijk}{\rightsquigarrow} ~~& (v_\alpha)_{ij} (v_\beta)_{jk} ~~;
\end{align*}
\end{minipage}
\begin{minipage}{0.45\linewidth}
\begin{align*} 
\begin{tikzpicture}[scale=0.3,baseline=0.1pc]
\draw[densely dotted] (0.5,-0.2)  -- (0.5,1.45)node[above] {$\scriptscriptstyle{1}$};
\draw[densely dotted] (1.5,-0.2)  -- (1.5,1.45)node[above] {$\scriptscriptstyle{2}$};
\draw[densely dotted] (2.5,-0.2)  -- (2.5,1.45)node[above] {$\scriptscriptstyle{3}$};
\draw[fill=white] (0,0.75) rectangle (2,1.25);
\draw[fill=white] (1,0) rectangle (3,0.5);
\end{tikzpicture}
~~ \overset{ijk}{\rightsquigarrow} ~~& (v_\alpha)_{jk} (v_\beta)_{ij}~~; \\
\begin{tikzpicture}[scale=0.3,baseline=0.1pc]
\draw[densely dotted] (0.5,-0.2)  -- (0.5,1.45)node[above] {$\scriptscriptstyle{1}$};
\draw[densely dotted] (1.5,-0.2)  -- (1.5,1.45)node[above] {$\scriptscriptstyle{2}$};
\draw[densely dotted] (2.5,-0.2)  -- (2.5,1.45)node[above] {$\scriptscriptstyle{3}$};
\draw[fill=white] (0,0.75) rectangle (1,01.25);
\draw[fill=white] (2,0.75) rectangle (3,1.25);
\draw[fill=white] (0,0) rectangle (2,0.5);
\end{tikzpicture}
~~ \overset{ijk}{\rightsquigarrow} ~~& (v_\alpha)_{ij} (v_\beta)_{ik} ~~;\\
\begin{tikzpicture}[scale=0.3,baseline=0.1pc]
\draw[densely dotted] (0.5,-0.2)  -- (0.5,1.45)node[above] {$\scriptscriptstyle{1}$};
\draw[densely dotted] (1.5,-0.2)  -- (1.5,1.45)node[above] {$\scriptscriptstyle{2}$};
\draw[densely dotted] (2.5,-0.2)  -- (2.5,1.45)node[above] {$\scriptscriptstyle{3}$};
\draw[fill=white] (0,0) rectangle (1,0.5);
\draw[fill=white] (2,0) rectangle (3,0.5);
\draw[fill=white] (0,0.75) rectangle (2,1.25);
\end{tikzpicture}
~~ \overset{ijk}{\rightsquigarrow} ~~& (v_\alpha)_{ik} (v_\beta)_{ij}~~.
\end{align*}
\end{minipage}
\caption{labelled procedure for $R(3)$}
\label{fig::labelled_procedure3}
\end{figure}

We denote by $R(3)_{ijk}$, the set of relations in $V(\calP ,n)^{\otimes 2}$ which are obtained by the labelled procedure $\overset{ijk}{\rightsquigarrow}$. 
We consider the quadratic algebra 
\[
\dfrac{\rmT\left(V(\calP ,n)\right)}{
	\left\langle \left. R(2)_{ij}, R(3)_{ijk}, \big[(v_\alpha)_{ij},(v_\beta)_{ab}\big]~\right|~
	{\substack{ 
	\forall 1\leqslant i<j<k \leqslant n, \\
	\forall 1\leqslant a<b \leqslant n ,\{i,j\}\cap\{a,b\}=\varnothing
	}}
	\right\rangle 
}\ .
\]
The new relations given by the commutator $[(v_\alpha)_{ij},(v_\beta)_{ab}]$ correspond to the "parallelism commutativity" which is present in the protoperadic structure:
\[
\begin{tikzpicture}[scale=0.3,baseline=0.1pc]
\draw[densely dotted] (0.5,-0.2)  -- (0.5,1.45)node[above] {$\scriptscriptstyle{1}$};
\draw[densely dotted] (1.5,-0.2)  -- (1.5,1.45)node[above] {$\scriptscriptstyle{2}$};
\draw[densely dotted] (2.75,-0.2)  -- (2.75,1.45)node[above] {$\scriptscriptstyle{3}$};
\draw[densely dotted] (3.75,-0.2)  -- (3.75,1.45)node[above] {$\scriptscriptstyle{4}$};
\draw[fill=white] (0,0) rectangle (2,0.5);
\draw[fill=white] (2.25,0.75) rectangle (4.25,1.25);
\draw (1,0.25) node {$\scriptscriptstyle{\alpha}$};
\draw (3.25,1) node {$\scriptscriptstyle{\beta}$};
\end{tikzpicture}
~=~
\begin{tikzpicture}[scale=0.3,baseline=0.1pc]
\draw[densely dotted] (0.5,-0.2)  -- (0.5,1.45)node[above] {$\scriptscriptstyle{1}$};
\draw[densely dotted] (1.5,-0.2)  -- (1.5,1.45)node[above] {$\scriptscriptstyle{2}$};
\draw[densely dotted] (2.75,-0.2)  -- (2.75,1.45)node[above] {$\scriptscriptstyle{3}$};
\draw[densely dotted] (3.75,-0.2)  -- (3.75,1.45)node[above] {$\scriptscriptstyle{4}$};
\draw[fill=white] (0,0.75) rectangle (2,1.25);
\draw[fill=white] (2.25,0) rectangle (4.25,0.5);
\draw (1,1) node {$\scriptscriptstyle{\alpha}$};
\draw (3.25,0.25) node {$\scriptscriptstyle{\beta}$};
\end{tikzpicture}~,
\]
(see \cite{Val07} for the properadic case).
\begin{lem}\label{lem::iso_algebra_proto}
	Let $\calP $ be a binary quadratic protoperad. For all integer $n\geqslant 2$, we have the isomorphism of algebras
	\begin{equation}\label{eq::APn_iso_SPn}
		\scrA(\calP,n)\cong 
		\dfrac{\rmT\left(V(\calP ,n)\right)}{
			\left\langle \left. R(2)_{ij}, R(3)_{ijk}, \big[(v_\alpha)_{ij},(v_\beta)_{ab}\big]~\right|~
			{\substack{ 
					\forall 1\leqslant i<j<k \leqslant n, \\
					\forall 1\leqslant a<b \leqslant n ,\{i,j\}\cap\{a,b\}=\varnothing
			}}
			\right\rangle 
		}\ .
	\end{equation}
\end{lem}
\begin{proof}
	We recall that, for a protoperad $(\calP , \mu)$, the product on $\mathbb{S}(\calP )([\![1,n]\!])$ is given by
	\[
	\mathbb{S}(\calP )([\![1,n]\!])\:\square\: \mathbb{S}(\calP )([\![1,n]\!]) 
	\cong \mathbb{S}(\calP \boxtimes_c\calP )([\![1,n]\!])
	\overset{\mathbb{S}(\mu)}{\longrightarrow}
	\mathbb{S}(\calP )([\![1,n]\!]).
	\]
	As $V(\calP ,n)$ is, by construction, a set of generators of the algebra $\mathbb{S}(\scrF\sh(V))([\![1,n]\!])$, we have the following morphism of algebras $\rmT\left(V(\calP ,n)\right) \rightarrow \mathbb{S}(\scrF\sh(V))([\![1,n]\!])$, which factorizes as follows:
	\[
	\xymatrix{
	\rmT\left(V(\calP ,n)\right) \ar[d] \ar[r] & \mathbb{S}(\scrF\sh(V))([\![1,n]\!]) \\
	\dfrac{	\rmT\left(V(\calP ,n)\right)}{\left\langle \big[(v_\alpha)_{ij},(v_\beta)_{ab}\big]~|~{\substack{\forall v_\alpha, v_\beta \\ \forall 1\leqslant a<b \leqslant n,  \forall 1\leqslant i<j \leqslant n \\ \{i,j\}\cap\{a,b\}=\varnothing }}\right\rangle} \ar[ru]_\cong^\phi &
	}~;
	\] 
	the isomorphism $\phi$ induces the isomorphism \eqref{eq::APn_iso_SPn}.
\end{proof}
\begin{thm}[Criterion of Koszulness]\label{thm::critere_koszulite}
	Let $\calP $ be a binary quadratic protoperad. If, for all integer $n\geqslant 2$, the quadratic algebra $\scrA(\calP,n)$ is Koszul, then the protoperad $\calP $ is Koszul.
\end{thm}
\begin{proof}
	Fix $n$ an integer such that $n \geqslant 2$. By \Cref{lem::iso_algebra_proto}, the bar constructions of the algebras $\scrA(\calP,n)$ are isomorphic, so we have the isomorphism of chain complexes
	\begin{equation}\label{eq::iso_bar}
		\Bar ^\Alg(\scrA(\calP,n)) \cong \Bar ^\Alg\big( \mathbb{S}(\calP \sh)\big)([\![1,n]\!]),
	\end{equation}
	where $\Bar ^\Alg$ is the bar construction for algebras (see \cite[Section 2.2]{LV12}). To a monomial $m$ of $\Bar ^\Alg(\scrA(\calP,n))$, we associate the partition which is induced by the set of pairs  $(i,j)$ of generator indices which appear in $m$, as explain below. We have the surjection $\Bar ^\Alg(\rmT(V(\calP ,n))) \rightarrow\Bar ^\Alg(\scrA(\calP,n))$, so choose a representative $\bar{m}$ of $m$ in $\Bar ^\Alg(\rmT(V(\calP ,n)))$ and consider the set of  pairs  $(i,j)$ of generator indices which appear in $\bar{m}$, completed by singletons $\{k\}$ if $k$ in $[\![1,n]\!]$ does not appear in any of the pairs. Such sets can be viewed as elements of $\Wall([\![1,n]\!])$, with the partial order induced by the lexicographic order. Then, by the natural transformation $\mathcal{K}:\mathcal{W} \rightarrow \mathcal{Y}$  (see \Cref{def::projection_K}), we associate $\bar{m}$, a partition $\mathfrak{p}$ of $[\![1,n]\!]$. 
	
	All relations in $\scrA(\calP,n)$ are given by $r_{ab}$ and $r'_{ijk}$ for $1\leqslant i<j<k \leqslant n$, $1\leqslant a<b \leqslant n$, $r$ in $R(2)$ and $r'$ in $R(3)$. So, as we see in \Cref{fig::labelled_procedure3}, any choice of representative $\bar{m}$ for $m$ gives us the same partition, then the partition $\mathfrak{p}$ does not depend of the choice of the representative $\bar{m}$. By the same argument, as the differential of $\Bar ^\Alg(\scrA(\calP,n))$ is induced by the product of $\scrA(\calP,n)$, the bar complex splits:
	\begin{equation}\label{eq::B_split}
	\Bar ^\Alg(\scrA(\calP,n)) \cong \bigoplus_{\mathfrak{p}\in \mathcal{Y}([\![1,n]\!])} \Bar ^\Alg_{\mathfrak{p}}(\scrA(\calP,n)).
	\end{equation}
	For convenience, we denote by $\mathfrak{p}_0$, the trivial partition with one element of $[\![1,n]\!]$. Through the isomorphism in \Cref{eq::iso_bar}, we identify the complex $\Bar ^\Alg_{\mathfrak{p}_0}(\scrA(\calP,n))$ with the normalized simplicial bar construction  $\mathbf{N}(\calP \sh)([\![1,n]\!])$.
	
	Let $p\geqslant 1$, a monomial $w$ in $\mathbf{N}^{(p)}(\calP \sh)([\![1,n]\!])$ is given by a leveled connected wall where bricks are labelled by monomials of $\calP \sh$.  To such a monomial $w
	$, we associate directly an element of $\Bar _{\mathfrak{p}_0}^{\Alg,(p)}(\scrA(\calP,n))$ where each level $w_i$ of $w$ is sent to a monomial $m_i$ in $\scrA(\calP,n)$, as in  \Cref{fig::example}.
	\begin{figure}[!h]
	\[
	\mathbf{N}^{(2)}(\calP \sh)([\![1,5]\!]) \ni 
	\begin{tikzpicture}[scale=0.3,baseline=1pc]
	\draw[densely dotted] (0.5,-0.2)  -- (0.5,3.25) node[above] {$\scriptscriptstyle{1}$};
	\draw[densely dotted] (1.5,-0.2)  -- (1.5,3.25) node[above] {$\scriptscriptstyle{2}$};
	\draw[densely dotted] (2.5,-0.2)  -- (2.5,3.25) node[above] {$\scriptscriptstyle{3}$};
	\draw[densely dotted] (3.5,-0.2)  -- (3.5,3.25) node[above] {$\scriptscriptstyle{4}$};
	\draw[densely dotted] (4.5,-0.2)  -- (4.5,3.25) node[above] {$\scriptscriptstyle{5}$};
	\draw[fill=white] (0,0) rectangle (1,0.5);
	\draw (0.5,0.25) node {$\scriptscriptstyle{\alpha}$} ;
	\draw[fill=white] (2,0) rectangle (3,0.5);
	\draw (2.5,0.25) node {$\scriptscriptstyle{\alpha}$} ;
	\draw[fill=white] (1,0.75) rectangle (3,1.25);
	\draw (2,1) node {$\scriptscriptstyle{\beta}$} ;
	\draw[dashed] (-0.5,1.5) -- (5.5,1.5);
	\draw[fill=white] (3,1.75) rectangle (5,2.25);
	\draw (4,2) node {$\scriptscriptstyle{\gamma}$} ;
	\draw[fill=white] (2,2.5) rectangle (4,3);
	\draw (3,2.75) node {$\scriptscriptstyle{\alpha}$} ;
	\end{tikzpicture}
	\longleftrightarrow
	(v_\alpha)_{13}(v_\beta)_{23}\otimes  (v_\gamma)_{45}(v_\alpha)_{34} \in \Bar _{\mathfrak{p}_0}^{\Alg,(2)}(A(\calP ,5))
	\]
	\caption{Example}
	\label{fig::example}
	\end{figure}
	It is clear that this application is an isomorphism of chain complexes:
	\begin{equation}\label{eq::N_iso_B}
		\mathbf{N}^{(p)}(\calP \sh)([\![1,n]\!])\cong \Bar _{\mathfrak{p}_0}^{\Alg,(p)}(\scrA(\calP,n)).
	\end{equation} 
	As the algebras $\scrA(\calP,n)$ are Koszul by hypothesis, for all $n\geqslant 2$, then the homology of $\Bar ^{\Alg,(p)}(\scrA(\calP,n))$ is concentrated in degree $p$. As this complex splits (see \Cref{eq::B_split}), then the homology of $\Bar _{\mathfrak{p}_0}^{\Alg,(p)}(\scrA(\calP,n))$ is also concentrated in degree $p$. Then, by the isomorphism in \Cref{eq::N_iso_B}, the homology of $\mathbf{N}^{(p)}(\calP \sh)([\![1,n]\!])$ is also concentrated in degree $p$. So, by  \Cref{thm::levelization}, the shuffle protoperad $\calP \sh$ is Koszul, then $\calP $ too, because $\Bar \calP$ and $\Bar _\shuffle \calP\sh$ are isomorphic as chain complexes, by \Cref{cor::shuffle}.
\end{proof}

\subsection{The main example: the protoperad \texorpdfstring{$\mathcal{DL}ie$}{DLie}}
In this section, we define the protoperad  $\mathcal{DL}ie$ and we show that it is Koszul by  \Cref{thm::critere_koszulite}.
\begin{defi}[The protoperad $\mathcal{DL}ie$]
	The protoperad $\mathcal{DL}ie$ is the quadratic protoperad 
	\[
	\mathcal{DL}ie:= \scrF\left(
	\begin{tikzpicture}[scale=0.3,baseline=-0.5ex]
	\draw[densely dotted] (0.2,0.75) node[above] {$\scriptscriptstyle{1}$} to (0.2,-0.75);
	\draw[densely dotted]  (1.8,0.75) node[above] {$\scriptscriptstyle{2}$} to (1.8,-0.75);
	\draw[fill=white] (0,-0.2) rectangle (2,0.3);
	\end{tikzpicture}
	= -
	\begin{tikzpicture}[scale=0.3,baseline=-0.5ex]
	\draw[densely dotted]  (0.2,0.75) node[above] {$\scriptscriptstyle{2}$} to (0.2,-0.75);
	\draw[densely dotted]  (1.8,0.75) node[above] {$\scriptscriptstyle{1}$} to (1.8,-0.75);
	\draw[fill=white] (0,-0.2) rectangle (2,0.3);
	\end{tikzpicture}
	\right) \Bigg/ 
	\Bigg\langle~
	\begin{tikzpicture}[scale=0.3,baseline=0.1pc]
	\draw[densely dotted] (0.5,-0.2)  -- (0.5,1.45)node[above] {$\scriptscriptstyle{1}$};
	\draw[densely dotted] (1.5,-0.2)  -- (1.5,1.45)node[above] {$\scriptscriptstyle{2}$};
	\draw[densely dotted] (2.5,-0.2)  -- (2.5,1.45)node[above] {$\scriptscriptstyle{3}$};
	\draw[fill=white] (0,0) rectangle (2,0.5);
	\draw[fill=white] (1,0.75) rectangle (3,1.25);
	\end{tikzpicture}
	+
	\begin{tikzpicture}[scale=0.3,baseline=0.1pc]
	\draw[densely dotted] (0.5,-0.2)  -- (0.5,1.45)node[above] {$\scriptscriptstyle{2}$};
	\draw[densely dotted] (1.5,-0.2)  -- (1.5,1.45)node[above] {$\scriptscriptstyle{3}$};
	\draw[densely dotted] (2.5,-0.2)  -- (2.5,1.45)node[above] {$\scriptscriptstyle{1}$};
	\draw[fill=white] (0,0) rectangle (2,0.5);
	\draw[fill=white] (1,0.75) rectangle (3,1.25);
	\end{tikzpicture}
	+
	\begin{tikzpicture}[scale=0.3,baseline=0.1pc]
	\draw[densely dotted] (0.5,-0.2)  -- (0.5,1.45)node[above] {$\scriptscriptstyle{3}$};
	\draw[densely dotted] (1.5,-0.2)  -- (1.5,1.45)node[above] {$\scriptscriptstyle{1}$};
	\draw[densely dotted] (2.5,-0.2)  -- (2.5,1.45)node[above] {$\scriptscriptstyle{2}$};
	\draw[fill=white] (0,0) rectangle (2,0.5);
	\draw[fill=white] (1,0.75) rectangle (3,1.25);
	\end{tikzpicture}~
	\Bigg\rangle~.
	\]
\end{defi}
\begin{rem}
	We associate to the protoperad $\DLie$, the shuffle protoperad 

\[
\DLie\sh=
\dfrac{\scrF\sh\left(
\begin{tikzpicture}[scale=0.3,baseline=-0.5ex]
\draw[densely dotted] (0.2,0.75) node[above] {$\scriptscriptstyle{1}$} to (0.2,-0.75);
\draw[densely dotted]  (1.8,0.75) node[above] {$\scriptscriptstyle{2}$} to (1.8,-0.75);
\draw[fill=white] (0,-0.2) rectangle (2,0.3);
\end{tikzpicture}
\right)
}{ 
\left\langle 
\begin{tikzpicture}[scale=0.3,baseline=0.1pc]
	\draw[densely dotted] (0.5,-0.2)  -- (0.5,1.45)node[above] {$\scriptscriptstyle{1}$};
\draw[densely dotted] (1.5,-0.2)  -- (1.5,1.45)node[above] {$\scriptscriptstyle{2}$};
\draw[densely dotted] (2.5,-0.2)  -- (2.5,1.45)node[above] {$\scriptscriptstyle{3}$};
\draw[fill=white](0,0) rectangle (1,0.5);
\draw[fill=white](2,0) rectangle (3,0.5);
\draw[fill=white](1,0.75) rectangle (3,1.25);
\end{tikzpicture}
\ - \
\begin{tikzpicture}[scale=0.3,baseline=0.1pc]
	\draw[densely dotted] (0.5,-0.2)  -- (0.5,1.45)node[above] {$\scriptscriptstyle{1}$};
\draw[densely dotted] (1.5,-0.2)  -- (1.5,1.45)node[above] {$\scriptscriptstyle{2}$};
\draw[densely dotted] (2.5,-0.2)  -- (2.5,1.45)node[above] {$\scriptscriptstyle{3}$};
\draw[fill=white](0,0.75) rectangle (2,1.25);
\draw[fill=white](1,0) rectangle (3,0.5);
\end{tikzpicture}
\ - \
\begin{tikzpicture}[scale=0.3,baseline=0.1pc]
	\draw[densely dotted] (0.5,-0.2)  -- (0.5,1.45)node[above] {$\scriptscriptstyle{1}$};
\draw[densely dotted] (1.5,-0.2)  -- (1.5,1.45)node[above] {$\scriptscriptstyle{2}$};
\draw[densely dotted] (2.5,-0.2)  -- (2.5,1.45)node[above] {$\scriptscriptstyle{3}$};
\draw[fill=white](0,0.75) rectangle (1,01.25);
\draw[fill=white](2,0.75) rectangle (3,1.25);
\draw[fill=white](0,0) rectangle (2,0.5);
\end{tikzpicture}
\ ; \
\begin{tikzpicture}[scale=0.3,baseline=0.1pc]
	\draw[densely dotted] (0.5,-0.2)  -- (0.5,1.45)node[above] {$\scriptscriptstyle{1}$};
\draw[densely dotted] (1.5,-0.2)  -- (1.5,1.45)node[above] {$\scriptscriptstyle{2}$};
\draw[densely dotted] (2.5,-0.2)  -- (2.5,1.45)node[above] {$\scriptscriptstyle{3}$};
\draw[fill=white](0,0) rectangle (2,0.5);
\draw[fill=white](1,0.75) rectangle (3,1.25);
\end{tikzpicture}
\ - \
\begin{tikzpicture}[scale=0.3,baseline=0.1pc]
\draw[densely dotted] (0.5,-0.2)  -- (0.5,1.45)node[above] {$\scriptscriptstyle{1}$};
\draw[densely dotted] (1.5,-0.2)  -- (1.5,1.45)node[above] {$\scriptscriptstyle{2}$};
\draw[densely dotted] (2.5,-0.2)  -- (2.5,1.45)node[above] {$\scriptscriptstyle{3}$};
\draw[fill=white](0,0.75) rectangle (1,1.25);
\draw[fill=white](2,0.75) rectangle (3,1.25);
\draw[fill=white](1,0) rectangle (3,0.5);
\end{tikzpicture}
\ - \ 
\begin{tikzpicture}[scale=0.3,baseline=0.1pc]
	\draw[densely dotted] (0.5,-0.2)  -- (0.5,1.45)node[above] {$\scriptscriptstyle{1}$};
\draw[densely dotted] (1.5,-0.2)  -- (1.5,1.45)node[above] {$\scriptscriptstyle{2}$};
\draw[densely dotted] (2.5,-0.2)  -- (2.5,1.45)node[above] {$\scriptscriptstyle{3}$};
\draw[fill=white](0,0) rectangle (1,0.5);
\draw[fill=white](2,0) rectangle (3,0.5);
\draw[fill=white](0,0.75) rectangle (2,1.25);
\end{tikzpicture} 
\ \right\rangle
}
\ ,
\]
by \Cref{cor::shuffle}. 
\end{rem}
To the protoperad $\DLie$, we associate the family of quadratic algebras, denoted by $\scrA(\DLie,n)$ for $n\geqslant 2$, given by the quadratic datum $\left(V(\DLie,n),R(\DLie,n)\right)$, with generators
\[
V(\DLie,n)=\{x_{ij}~|~1\leqslant i<j\leqslant n\} 
\]
and relations
\[
R(\DLie,2)=0
~;
~
R(\DLie,3)=\left\{
\begin{array}{c}
x_{12}x_{23}- x_{23}x_{13} -x_{13}x_{12} \\
x_{23}x_{12}- x_{13}x_{23} -x_{12}x_{13} \\
\end{array}
\right\}
~;
\]
and, for $n\geqslant 4$,
\[
R(\DLie,n)=\left\{
\begin{array}{c}
\left.
\begin{array}{c}
x_{ij}x_{jk}- x_{jk}x_{ik} -x_{ik}x_{ij} \\
x_{jk}x_{ij}- x_{ik}x_{jk} -x_{ij}x_{ik} \\
x_{ab}x_{uv}-x_{uv}x_{ab}
\end{array}
\right|
\begin{array}{c}
1\leqslant i<j<k \leqslant n\\
1\leqslant u<v \leqslant n \\
1\leqslant a<b \leqslant n\\
\{a,b\}\cap\{u,v\}=\varnothing
\end{array}
\end{array}
\right\}
~.
\]

\begin{prop}\label{prop::ADLie_Koszul}
	For all $n\geqslant 2$, the quadratic algebra $\scrA(\DLie,n)$ is Koszul.
\end{prop}
\begin{proof} 
See \Cref{proof::ADLie_Koszul} for the proof.
\end{proof}
\begin{thm}\label{thm::proto_DLie_Koszul}
	The protoperad $\DLie$ is Koszul.
\end{thm}
\begin{proof}
	By \Cref{prop::ADLie_Koszul} and \Cref{thm::critere_koszulite}
\end{proof}
\begin{cor}\label{thm::prop_DLie_Koszul}
 The properad $\Ind(\DLie)$ is Koszul.
\end{cor}
\begin{proof}
	The monoidal functor $\Ind$ is exact by \Cref{def::functor induction}.
\end{proof}
This corollary is very important: it is the first example of a Koszul properad with a generator not in arity $(1,2)$ or $(2,1)$.

\section{\texorpdfstring{$\DPois$}{DPois} is Koszul}\label{sect::DPoisKoszul}
In this section, we study the Koszul dual of the protoperad $\DLie$, which is called $\DCom$, by analogy of the case of operads $\Lie$ and $\Com$.

\subsection{The Koszul dual of $\DLie$} \label{subsect::DCom}
To the protoperad $\DLie$, we associate its Koszul dual, which we will called $\DCom$:
\[
	\DCom :=\scrF(V_{\DLie}^*)\big/\big\langle R_{\mathcal{DJ}}^\bot\big\rangle,
\]
where $V^*$ is the linear dual of $V$ and, for all $R \subset \scrF^{(2)}(V)$, $R^\bot$ is the orthogonal of $R$ in  $\scrF^{(2)}(V^*)$. The $\mathfrak{S}$-module $V_{\DLie}^*$ is identified to 
\[
	V_{\DLie}^*([\![1,m]\!])=
		\left\{\begin{array}{cc}
		\mathrm{sgn}(\mathfrak{S}_2) &~\mbox{ if }m=2 \\
		0 & \mbox{ otherwise}.
		\end{array}\right\}
\]
Then, as in the case of the protoperad $\DLie$, we can diagrammatically interpret $V_{\DLie}^*([\![1,2]\!])$ as follow 

\[
	V_{\DLie}^*([\![1,2]\!]) ~~\cong~~
	\left\langle \
	\begin{tikzpicture}[baseline=0.4ex,scale=0.2]
	\draw[thin] (2,1.5) node[above] {$\scriptscriptstyle{1}$} -- (2,0);
	\draw[thin] (4,1.5) node[above] {$\scriptscriptstyle{2}$} -- (4,0);
	\draw[fill=white] (1.7,0.5) rectangle (4.3,1);
	\end{tikzpicture}
	\ = \ -
	\begin{tikzpicture}[baseline=0.4ex,scale=0.2]
	\draw[thin] (2,1.5) node[above] {$\scriptscriptstyle{2}$} -- (2,0);
	\draw[thin] (4,1.5) node[above] {$\scriptscriptstyle{1}$} -- (4,0);
	\draw[fill=white] (1.7,0.5) rectangle (4.3,1);
	\end{tikzpicture}
	\ \right\rangle .
\]
We also have the following relations:
\[
	R_{\mathcal{DJ}}^\bot : ~~ 
	\begin{tikzpicture}[scale=0.2,baseline=0]
		\draw (0,3) node[below] {$\scriptscriptstyle{1}$};
		\draw (2,3) node[below] {$\scriptscriptstyle{2}$};
		\draw (4,3) node[below] {$\scriptscriptstyle{3}$};
		\draw[thin] (2,1.5) -- (2,1);
		\draw[thin] (4,1.5) -- (4,1);
		\draw (1.7,0.5) rectangle (4.3,1);
		\draw[thin] (2,0.5) -- (2,0);
		\draw[thin] (4,0.5) -- (4,-1);
		\draw[thin] (0,1.5) -- (0,0);
		\draw[thin] (2,1.5) -- (2,1);
		\draw (-0.3,-0.5) rectangle (2.3,0);
		\draw[thin] (0,-1) -- (0,-0.5);
		\draw[thin] (2,-1) -- (2,-0.5);
	\end{tikzpicture}
	~~-~~
	\begin{tikzpicture}[scale=0.2,baseline=0]
		\draw (0,3) node[below] {$\scriptscriptstyle{2}$};
		\draw (2,3) node[below] {$\scriptscriptstyle{3}$};
		\draw (4,3) node[below] {$\scriptscriptstyle{1}$};
		\draw[thin] (2,1.5) -- (2,1);
		\draw[thin] (4,1.5) -- (4,1);
		\draw (1.7,0.5) rectangle (4.3,1);
		\draw[thin] (2,0.5) -- (2,0);
		\draw[thin] (4,0.5) -- (4,-1);
		\draw[thin] (0,1.5) -- (0,0);
		\draw[thin] (2,1.5) -- (2,1);
		\draw (-0.3,-0.5) rectangle (2.3,0);
		\draw[thin] (0,-1) -- (0,-0.5);
		\draw[thin] (2,-1) -- (2,-0.5);
	\end{tikzpicture}
	~~;~~
	\begin{tikzpicture}[scale=0.2,baseline=4]
		\draw[thin] (2,2) -- (2,2.5);
		\draw[thin] (4,2) -- (4,2.5);
		\draw (1.7,1.5) rectangle (4.3,2);
		\draw[thin] (2,1.5) -- (2,1);
		\draw[thin] (4,1.5) -- (4,1);
		\draw (1.7,0.5) rectangle (4.3,1);
		\draw[thin] (2,0.5) -- (2,0);
		\draw[thin] (4,0.5) -- (4,0);
	\end{tikzpicture} \ .
\]
By the second relation in  $\mathcal{R}^\bot_{\DLie}$, we directly have that 
\[
	\DCom([\![1,2]\!])=\DCom_{(1)}([\![1,2]\!])=V_{\DLie}^*.
\]
For the $\mathfrak{S}$-module $\DCom([\![1,3]\!])$, we have:
\[
	\begin{tikzpicture}[scale=0.2,baseline=3]
		\draw (2,2.5) node[above] {$\scriptscriptstyle{1}$};
		\draw (4,2.5) node[above] {$\scriptscriptstyle{2}$};
		\draw (6,2.5) node[above] {$\scriptscriptstyle{3}$};
		\draw[thin] (4,2) -- (4,2.5);
		\draw[thin] (6,2) -- (6,2.5);
		\draw (3.7,1.5) rectangle (6.3,2);
		\draw[thin] (2,2.5) -- (2,1);
		\draw[thin] (4,1.5) -- (4,1);
		\draw (1.7,0.5) rectangle (4.3,1);
		\draw[thin] (2,0.5) -- (2,0);
		\draw[thin] (4,0.5) -- (4,0);
		\draw[thin] (6,1.5) -- (6,0);
	\end{tikzpicture}
	~~=~~
	\begin{tikzpicture}[scale=0.2,baseline=3]
		\draw (2,2.5) node[above] {$\scriptscriptstyle{3}$};
		\draw (4,2.5) node[above] {$\scriptscriptstyle{2}$};
		\draw (6,2.5) node[above] {$\scriptscriptstyle{1}$};
		\draw[thin] (2,0) -- (2,2.5);
		\draw[thin] (4,0) -- (4,2.5);
		\draw[thin] (6,0) -- (6,2.5);
		\draw[fill=white] (1.7,1.5) rectangle (4.3,2);
		\draw[fill=white] (3.7,0.5) rectangle (6.3,1);
	\end{tikzpicture}.
\]
If we consider the elements of weight $3$ in $\DCom([\![1,3]\!])$, we have, by the first relation in $R_{\DLie}^\bot$, the two following equality:
\[
	\begin{tikzpicture}[scale=0.2,baseline=6]
		\draw (2,3.5) node[above] {$\scriptscriptstyle{1}$};
		\draw (4,3.5) node[above] {$\scriptscriptstyle{2}$};
		\draw (6,3.5) node[above] {$\scriptscriptstyle{3}$};
		\draw[thin] (2,3.5) -- (2,0);
		\draw[thin] (4,3.5) -- (4,0);
		\draw[thin] (6,3.5) -- (6,0);
		\draw[fill=white] (1.7,0.5) rectangle (4.3,1);
		\draw[fill=white] (3.7,1.5) rectangle (6.3,2);
		\draw[fill=white] (1.7,2.5) rectangle (4.3,3);
	\end{tikzpicture}
	~~=~~
	\begin{tikzpicture}[scale=0.2,baseline=6]
		\draw (2,3.5) node[above] {$\scriptscriptstyle{3}$};
		\draw (4,3.5) node[above] {$\scriptscriptstyle{1}$};
		\draw (6,3.5) node[above] {$\scriptscriptstyle{2}$};
		\draw[thin] (2,3.5) -- (2,0);
		\draw[thin] (4,3.5) -- (4,0);
		\draw[thin] (6,3.5) -- (6,0);
		\draw[fill=white] (1.7,0.5) rectangle (4.3,1);
		\draw[fill=white] (3.7,1.5) rectangle (6.3,2);
		\draw[fill=white] (3.7,2.5) rectangle (6.3,3);
	\end{tikzpicture}
	~~=0~~;~~
	\begin{tikzpicture}[scale=0.2,baseline=10]
		\draw (2,4.5) node[above] {$\scriptscriptstyle{1}$};
		\draw (4,4.5) node[above] {$\scriptscriptstyle{3}$};
		\draw (6,4.5) node[above] {$\scriptscriptstyle{2}$};
		\draw[thin] (2,4.5) -- (2,0);
		\draw[thin] (4,2) -- (4,0);
		\draw[thin] (6,2) -- (6,0);
		\draw[thin] (4,2) to[out=90,in=270] (6,3.5);
		\draw[draw=white,double=black,double distance=\pgflinewidth,thick] (6,2) to[out=90,in=270] (4,3.5);
		\draw[thin] (4,3.5) -- (4,4.5);
		\draw[thin] (6,3.5) -- (6,4.5);
		\draw[fill=white] (1.7,0.5) rectangle (4.3,1);
		\draw[fill=white] (3.7,1.5) rectangle (6.3,2);
		\draw[fill=white] (1.7,3.5) rectangle (4.3,4);
	\end{tikzpicture}
	~~=~~-
	\begin{tikzpicture}[scale=0.2,baseline=6]
		\draw (2,3.5) node[above] {$\scriptscriptstyle{3}$};
		\draw (4,3.5) node[above] {$\scriptscriptstyle{2}$};
		\draw (6,3.5) node[above] {$\scriptscriptstyle{1}$};
		\draw[thin] (2,3.5) -- (2,0);
		\draw[thin] (4,3.5) -- (4,0);
		\draw[thin] (6,3.5) -- (6,0);
		\draw[fill=white] (1.7,0.5) rectangle (4.3,1);
		\draw[fill=white] (3.7,1.5) rectangle (6.3,2);
		\draw[fill=white] (3.7,2.5) rectangle (6.3,3);
	\end{tikzpicture}
	~~=0.
\]
That implies that the  $\mathfrak{S}$-module $\DCom([\![1,3]\!])$ is reduced to its component of weight $2$, i.e. $\DCom([\![1,3]\!])=\DCom_{(2)}([\![1,3]\!])$. This equality is a more general thing, as we will see, i.e. we will prove that, for all  $n\geqslant 2$, we have $\DCom([\![1,n]\!])=\DCom_{(n-1)}([\![1,n]\!])$.

\begin{lem}\label{invariance_cyclique}
	Every stairway of arity $n$ is invariant up to the sign by the diagonal action of  $\Z/n\Z$, that is, for all $n\geqslant 2$
	\[
		\begin{tikzpicture}[scale=0.2,baseline=1.2pc]
		\draw (0,0.5) node[above] {$\scriptscriptstyle{1}$};
		\draw (2,1.5) node[above] {$\scriptscriptstyle{2}$};
		\draw[thin] (2,0) -- (2,1.5);
		\draw[thin] (4,1.5) -- (4,2);
		\draw[fill=white] (-0.3,0) rectangle (2.3,0.5);
		\draw[fill=white] (1.7,1) rectangle (4.3,1.5);
		\draw [dotted] (4,2) -- (8,3.5);
		\draw (8,4.5) node[above] {$\scriptscriptstyle{n-2}$};
		\draw (10,5.5) node[above] {$\scriptscriptstyle{n-1}$};
		\draw (12,5.5) node[above] {$\scriptscriptstyle{n}$};
		\draw (10,4.5) -- (10,5);
		\draw (8,3.5) -- (8,4);
		\draw[fill=white] (7.7,4) rectangle (10.3,4.5);
		\draw[fill=white] (9.7,5) rectangle (12.3,5.5);
		\end{tikzpicture}
		~~=~~(-1)^{n-1}~~
		\begin{tikzpicture}[scale=0.2,baseline=1.2pc]
		\draw (0,0.5) node[above] {$\scriptscriptstyle{2}$};
		\draw (2,1.5) node[above] {$\scriptscriptstyle{3}$};
		\draw[thin] (2,0) -- (2,1.5);
		\draw[thin] (4,1.5) -- (4,2);
		\draw[fill=white] (-0.3,0) rectangle (2.3,0.5);
		\draw[fill=white] (1.7,1) rectangle (4.3,1.5);
		\draw [dotted] (4,2) -- (8,3.5);
		\draw (8,4.5) node[above] {$\scriptscriptstyle{n-1}$};
		\draw (10,5.5) node[above] {$\scriptscriptstyle{n}$};
		\draw (12,5.5) node[above] {$\scriptscriptstyle{1}$};
		\draw (10,4.5) -- (10,5);
		\draw (8,3.5) -- (8,4);
		\draw[fill=white] (7.7,4) rectangle (10.3,4.5);
		\draw[fill=white] (9.7,5) rectangle (12.3,5.5);
		\end{tikzpicture} ~.
	\]
\end{lem}

\begin{proof}
	We prove this result by induction on the arity $n$.  By the definition of the protoperad $\DCom$, we have:
	\[
		\begin{tikzpicture}[scale=0.2,baseline=2]
		\draw (2,1.5) node[above] {$\scriptscriptstyle{1}$};
		\draw (4,1.5) node[above] {$\scriptscriptstyle{2}$};
		\draw[thin] (2,1.5) -- (2,0);
		\draw[thin] (4,1.5) -- (4,0);
		\draw[fill=white] (1.7,0.5) rectangle (4.3,1);
		\end{tikzpicture}
		~~=~-~~
		\begin{tikzpicture}[scale=0.2,baseline=2]
		\draw (2,1.5) node[above] {$\scriptscriptstyle{2}$};
		\draw (4,1.5) node[above] {$\scriptscriptstyle{1}$};
		\draw[thin] (2,1.5) -- (2,0);
		\draw[thin] (4,1.5) -- (4,0);
		\draw[fill=white] (1.7,0.5) rectangle (4.3,1);
		\end{tikzpicture}
		~~\mbox{ and }~~
		\begin{tikzpicture}[scale=0.2,baseline=2]
		\draw (2,2.5) node[above] {$\scriptscriptstyle{1}$};
		\draw (4,2.5) node[above] {$\scriptscriptstyle{2}$};
		\draw (6,2.5) node[above] {$\scriptscriptstyle{3}$};
		\draw[thin] (4,2) -- (4,2.5);
		\draw[thin] (6,2) -- (6,2.5);
		\draw (3.7,1.5) rectangle (6.3,2);
		\draw[thin] (2,2.5) -- (2,1);
		\draw[thin] (4,1.5) -- (4,1);
		\draw (1.7,0.5) rectangle (4.3,1);
		\draw[thin] (2,0.5) -- (2,0);
		\draw[thin] (4,0.5) -- (4,0);
		\draw[thin] (6,1.5) -- (6,0);
		\end{tikzpicture}
		~~=~~
		\begin{tikzpicture}[scale=0.2,baseline=2]
		\draw (2,2.5) node[above] {$\scriptscriptstyle{2}$};
		\draw (4,2.5) node[above] {$\scriptscriptstyle{3}$};
		\draw (6,2.5) node[above] {$\scriptscriptstyle{1}$};
		\draw[thin] (2,0) -- (2,2.5);
		\draw[thin] (4,0) -- (4,2.5);
		\draw[thin] (6,0) -- (6,2.5);
		\draw[fill=white] (1.7,0.5) rectangle (4.3,1);
		\draw[fill=white] (3.7,1.5) rectangle (6.3,2);
		\end{tikzpicture}~.
	\]
	Suppose that, for a fixed integer $n$, we have the following equality:
	\[
		\begin{tikzpicture}[scale=0.2,baseline=1.2pc]
		\draw (0,0.5) node[above] {$\scriptscriptstyle{1}$};
		\draw (2,1.5) node[above] {$\scriptscriptstyle{2}$};
		\draw[thin] (2,0) -- (2,1.5);
		\draw[thin] (4,1.5) -- (4,2);
		\draw[fill=white] (-0.3,0) rectangle (2.3,0.5);
		\draw[fill=white] (1.7,1) rectangle (4.3,1.5);
		\draw [dotted] (4,2) -- (8,3.5);
		\draw (8,4.5) node[above] {$\scriptscriptstyle{n-2}$};
		\draw (10,5.5) node[above] {$\scriptscriptstyle{n-1}$};
		\draw (12,5.5) node[above] {$\scriptscriptstyle{n}$};
		\draw (10,4.5) -- (10,5);
		\draw (8,3.5) -- (8,4);
		\draw[fill=white] (7.7,4) rectangle (10.3,4.5);
		\draw[fill=white] (9.7,5) rectangle (12.3,5.5);
		\end{tikzpicture}
		~~=~~(-1)^{n-1}
		\begin{tikzpicture}[scale=0.2,baseline=1.2pc]
		\draw (0,0.5) node[above] {$\scriptscriptstyle{2}$};
		\draw (2,1.5) node[above] {$\scriptscriptstyle{3}$};
		\draw[thin] (2,0) -- (2,1.5);
		\draw[thin] (4,1.5) -- (4,2);
		\draw[fill=white] (-0.3,0) rectangle (2.3,0.5);
		\draw[fill=white] (1.7,1) rectangle (4.3,1.5);
		\draw [dotted] (4,2) -- (8,3.5);
		\draw (8,4.5) node[above] {$\scriptscriptstyle{n-1}$};
		\draw (10,5.5) node[above] {$\scriptscriptstyle{n}$};
		\draw (12,5.5) node[above] {$\scriptscriptstyle{1}$};
		\draw (10,4.5) -- (10,5);
		\draw (8,3.5) -- (8,4);
		\draw[fill=white] (7.7,4) rectangle (10.3,4.5);
		\draw[fill=white] (9.7,5) rectangle (12.3,5.5);
		\end{tikzpicture}~.
	\]
	Then, we have
	\[
		\begin{tikzpicture}[scale=0.2,baseline=1.2pc]
		\draw (0,0.5) node[above] {$\scriptscriptstyle{1}$};
		\draw (2,1.5) node[above] {$\scriptscriptstyle{2}$};
		\draw[thin] (2,0) -- (2,1.5);
		\draw[thin] (4,1.5) -- (4,2);
		\draw[fill=white] (-0.3,0) rectangle (2.3,0.5);
		\draw[fill=white] (1.7,1) rectangle (4.3,1.5);
		\draw [dotted] (4,2) -- (8,3.5);
		\draw (8,4.5) node[above] {$\scriptscriptstyle{n-1}$};
		\draw (10,5.5) node[above] {$\scriptscriptstyle{n}$};
		\draw (12,5.5) node[above] {$\scriptscriptstyle{n+1}$};
		\draw (10,4.5) -- (10,5);
		\draw (8,3.5) -- (8,4);
		\draw[fill=white] (7.7,4) rectangle (10.3,4.5);
		\draw[fill=white] (9.7,5) rectangle (12.3,5.5);
		\end{tikzpicture}
		~~=~~(-1)^{n-1}~
		\begin{tikzpicture}[scale=0.2,baseline=1.6pc]
		\draw (0,0.5) node[above] {$\scriptscriptstyle{2}$};
		\draw (2,1.5) node[above] {$\scriptscriptstyle{3}$};
		\draw[thin] (2,0) -- (2,1.5);
		\draw[thin] (4,1.5) -- (4,2);
		\draw[fill=white] (-0.3,0) rectangle (2.3,0.5);
		\draw[fill=white] (1.7,1) rectangle (4.3,1.5);
		\draw [dotted] (4,2) -- (8,3.5);
		\draw (8,8.5) node[above] {$\scriptscriptstyle{n}$};
		\draw (10,4.5) node[above] {$\scriptscriptstyle{1}$};
		\draw (10,8.5) node[above] {$\scriptscriptstyle{n+1}$};
		\draw (8,3.5) -- (8,8.5);
		\draw[fill=white] (7.7,4) rectangle (10.3,4.5);
		\draw[fill=white] (7.7,8) rectangle (10.3,8.5);
		\end{tikzpicture}
		~~=~~(-1)^{n}~
		\begin{tikzpicture}[scale=0.2,baseline=1.6pc]
		\draw (0,0.5) node[above] {$\scriptscriptstyle{2}$};
		\draw (2,1.5) node[above] {$\scriptscriptstyle{3}$};
		\draw[thin] (2,0) -- (2,1.5);
		\draw[thin] (4,1.5) -- (4,2);
		\draw[fill=white] (-0.3,0) rectangle (2.3,0.5);
		\draw[fill=white] (1.7,1) rectangle (4.3,1.5);
		\draw [dotted] (4,2) -- (6,3);
		\draw (8,8.5) node[above] {$\scriptscriptstyle{n}$};
		\draw (6,7.5) node[above] {$\scriptscriptstyle{1}$};
		\draw (10,8.5) node[above] {$\scriptscriptstyle{n+1}$};
		\draw (8,3.5) -- (8,8.5);
		\draw (6,3) -- (6,3.5);
		\draw (6,4) node[above] {$\scriptscriptstyle{n-1}$};
		\draw[fill=white] (5.7,3.5) rectangle (8.3,4);
		\draw[fill=white] (5.7,7) rectangle (8.3,7.5);
		\draw[fill=white] (7.7,8) rectangle (10.3,8.5);
		\end{tikzpicture} \ .
	\]
	Then, we have
	\[	
		\begin{tikzpicture}[scale=0.2,baseline=1.2pc]
			\draw (0,0.5) node[above] {$\scriptscriptstyle{1}$};
			\draw (2,1.5) node[above] {$\scriptscriptstyle{2}$};
			\draw[thin] (2,0) -- (2,1.5);
			\draw[thin] (4,1.5) -- (4,2);
			\draw[fill=white] (-0.3,0) rectangle (2.3,0.5);
			\draw[fill=white] (1.7,1) rectangle (4.3,1.5);
			\draw [dotted] (4,2) -- (8,3.5);
			\draw (8,4.5) node[above] {$\scriptscriptstyle{n-1}$};
			\draw (10,5.5) node[above] {$\scriptscriptstyle{n}$};
			\draw (12,5.5) node[above] {$\scriptscriptstyle{n+1}$};
			\draw (10,4.5) -- (10,5);
			\draw (8,3.5) -- (8,4);
			\draw[fill=white] (7.7,4) rectangle (10.3,4.5);
			\draw[fill=white] (9.7,5) rectangle (12.3,5.5);
		\end{tikzpicture}
		~~=~~(-1)^{n}~
		\begin{tikzpicture}[scale=0.2,baseline=1.6pc]
			\draw (0,0.5) node[above] {$\scriptscriptstyle{2}$};
			\draw (2,1.5) node[above] {$\scriptscriptstyle{3}$};
			\draw[thin] (2,0) -- (2,1.5);
			\draw[thin] (4,1.5) -- (4,2);
			\draw[fill=white] (-0.3,0) rectangle (2.3,0.5);
			\draw[fill=white] (1.7,1) rectangle (4.3,1.5);
			\draw [dotted] (4,2) -- (6,3);
			\draw (10,8.5) node[above] {$\scriptscriptstyle{n+1}$};
			\draw (8,7.5) node[above] {$\scriptscriptstyle{n}$};
			\draw (12,8.5) node[above] {$\scriptscriptstyle{1}$};
			\draw (8,3.5) -- (8,7);
			\draw (6,3) -- (6,3.5);
			\draw (6,4) node[above] {$\scriptscriptstyle{n-1}$};
			\draw[fill=white] (5.7,3.5) rectangle (8.3,4);
			\draw[fill=white] (7.7,7) rectangle (10.3,7.5);
			\draw (10,7.5) -- (10,8);
			\draw[fill=white] (9.7,8) rectangle (12.3,8.5);
		\end{tikzpicture}~.
	\]
\end{proof}

\begin{lem}\label{egal_symetrique}
	For all integer  $n$ greater than $2$, we have the following equality
	\[
	\begin{tikzpicture}[scale=0.2,baseline=1.2pc]
		\draw (0,0.5) node[above] {$\scriptscriptstyle{1}$};
		\draw (2,1.5) node[above] {$\scriptscriptstyle{2}$};
		\draw[thin] (2,0) -- (2,1.5);
		\draw[thin] (4,1.5) -- (4,2);
		\draw[fill=white] (-0.3,0) rectangle (2.3,0.5);
		\draw[fill=white] (1.7,1) rectangle (4.3,1.5);
		\draw [dotted] (4,2) -- (8,3.5);
		\draw (8,4.5) node[above] {$\scriptscriptstyle{n-2}$};
		\draw (10,5.5) node[above] {$\scriptscriptstyle{n-1}$};
		\draw (12,5.5) node[above] {$\scriptscriptstyle{n}$};
		\draw (10,4.5) -- (10,5);
		\draw (8,3.5) -- (8,4);
		\draw[fill=white] (7.7,4) rectangle (10.3,4.5);
		\draw[fill=white] (9.7,5) rectangle (12.3,5.5);
	\end{tikzpicture}
	~~=~~(-1)^{n+1}
	\begin{tikzpicture}[scale=0.2,baseline=1.2pc]
		\draw (0,5.5) node[above] {$\scriptscriptstyle{n}$};
		\draw (2,5.5) node[above] {$\scriptscriptstyle{n-1}$};
		\draw[thin] (2,4.5) -- (2,5);
		\draw[thin] (4,3.5) -- (4,4);
		\draw[fill=white] (-0.3,5) rectangle (2.3,5.5);
		\draw[fill=white] (1.7,4) rectangle (4.3,4.5);
		\draw (4,4.5) node[above] {$\scriptscriptstyle{n-2}$};
		\draw [dotted] (4,3.5) -- (8,2);
		\draw (10,1.5) node[above] {$\scriptscriptstyle{2}$};
		\draw (12,0.5) node[above] {$\scriptscriptstyle{1}$};
		\draw (10,1) -- (10,0.5);
		\draw (8,1.5) -- (8,2);
		\draw[fill=white] (7.7,1) rectangle (10.3,1.5);
		\draw[fill=white] (9.7,0) rectangle (12.3,0.5);
	\end{tikzpicture}~.
	\]
\end{lem}
\begin{proof}
	We prove this result by induction on the arity. We also have
	\[
	\begin{tikzpicture}[scale=0.2,baseline=2]
		\draw (2,2.5) node[above] {$\scriptscriptstyle{1}$};
		\draw (4,2.5) node[above] {$\scriptscriptstyle{2}$};
		\draw (6,2.5) node[above] {$\scriptscriptstyle{3}$};
		\draw[thin] (2,0) -- (2,2.5);
		\draw[thin] (4,0) -- (4,2.5);
		\draw[thin] (6,0) -- (6,2.5);
		\draw[fill=white] (3.7,1.5) rectangle (6.3,2);
		\draw[fill=white] (1.7,0.5) rectangle (4.3,1);
	\end{tikzpicture}
	~~=~~
	\begin{tikzpicture}[scale=0.2,baseline=2]
		\draw (2,2.5) node[above] {$\scriptscriptstyle{3}$};
		\draw (4,2.5) node[above] {$\scriptscriptstyle{2}$};
		\draw (6,2.5) node[above] {$\scriptscriptstyle{1}$};
		\draw[thin] (2,0) -- (2,2.5);
		\draw[thin] (4,0) -- (4,2.5);
		\draw[thin] (6,0) -- (6,2.5);
		\draw[fill=white] (1.7,1.5) rectangle (4.3,2);
		\draw[fill=white] (3.7,0.5) rectangle (6.3,1);
	\end{tikzpicture}.
	\]
	Suppose that, for a fixed integer $n\geqslant 2$, we have
	\[
	\begin{tikzpicture}[scale=0.2,baseline=1.2pc]
		\draw (0,0.5) node[above] {$\scriptscriptstyle{1}$};
		\draw (2,1.5) node[above] {$\scriptscriptstyle{2}$};
		\draw[thin] (2,0) -- (2,1.5);
		\draw[thin] (4,1.5) -- (4,2);
		\draw[fill=white] (-0.3,0) rectangle (2.3,0.5);
		\draw[fill=white] (1.7,1) rectangle (4.3,1.5);
		\draw [dotted] (4,2) -- (8,3.5);
		\draw (8,4.5) node[above] {$\scriptscriptstyle{n-3}$};
		\draw (10,5.7) node[above] {$\scriptscriptstyle{n-2}$};
		\draw (12,5) node[above] {$\scriptscriptstyle{n-1}$};
		\draw (10,4.5) -- (10,5);
		\draw (8,3.5) -- (8,4);
		\draw[fill=white] (7.7,4) rectangle (10.3,4.5);
		\draw[fill=white] (9.7,5) rectangle (12.3,5.5);
	\end{tikzpicture}
	~~=~~(-1)^n~
	\begin{tikzpicture}[scale=0.2,baseline=1.2pc]
		\draw (0,5) node[above] {$\scriptscriptstyle{n-1}$};
		\draw (2,5.7) node[above] {$\scriptscriptstyle{n-2}$};
		\draw[thin] (2,4.5) -- (2,5);
		\draw[thin] (4,3.5) -- (4,4);
		\draw[fill=white] (-0.3,5) rectangle (2.3,5.5);
		\draw[fill=white] (1.7,4) rectangle (4.3,4.5);
		\draw (4,4.5) node[above] {$\scriptscriptstyle{n-3}$};
		\draw [dotted] (4,3.5) -- (8,2);
		\draw (10,1.5) node[above] {$\scriptscriptstyle{2}$};
		\draw (12,0.5) node[above] {$\scriptscriptstyle{1}$};
		\draw (10,1) -- (10,0.5);
		\draw (8,1.5) -- (8,2);
		\draw[fill=white] (7.7,1) rectangle (10.3,1.5);
		\draw[fill=white] (9.7,0) rectangle (12.3,0.5);
	\end{tikzpicture}.
	\]
	Then, we have
	\[
	\begin{tikzpicture}[scale=0.2,baseline=1.2pc]
		\draw (0,0.5) node[above] {$\scriptscriptstyle{1}$};
		\draw (2,1.5) node[above] {$\scriptscriptstyle{2}$};
		\draw[thin] (2,0) -- (2,1.5);
		\draw[thin] (4,1.5) -- (4,2);
		\draw[fill=white] (-0.3,0) rectangle (2.3,0.5);
		\draw[fill=white] (1.7,1) rectangle (4.3,1.5);
		\draw [dotted] (4,2) -- (8,3.5);
		\draw (8,4.5) node[above] {$\scriptscriptstyle{n-2}$};
		\draw (10,5.5) node[above] {$\scriptscriptstyle{n-1}$};
		\draw (12,5.5) node[above] {$\scriptscriptstyle{n}$};
		\draw (10,4.5) -- (10,5);
		\draw (8,3.5) -- (8,4);
		\draw[fill=white] (7.7,4) rectangle (10.3,4.5);
		\draw[fill=white] (9.7,5) rectangle (12.3,5.5);
	\end{tikzpicture}
	~~=~~-~
	\begin{tikzpicture}[scale=0.2,baseline=1.2pc]
		\draw (0,0.5) node[above] {$\scriptscriptstyle{1}$};
		\draw (2,1.5) node[above] {$\scriptscriptstyle{2}$};
		\draw[thin] (2,0) -- (2,1.5);
		\draw[thin] (4,1.5) -- (4,2);
		\draw[fill=white] (-0.3,0) rectangle (2.3,0.5);
		\draw[fill=white] (1.7,1) rectangle (4.3,1.5);
		\draw [dotted] (4,2) -- (8,3.5);
		\draw (8,4.5) node[above] {$\scriptscriptstyle{n-2}$};
		\draw (10,7.5) node[above] {$\scriptscriptstyle{n-1}$};
		\draw (8,7.5) node[above] {$\scriptscriptstyle{n}$};
		\draw (10,4.5) -- (10,7);
		\draw (8,3.5) -- (8,4);
		\draw[fill=white] (7.7,4) rectangle (10.3,4.5);
		\draw[fill=white] (7.7,7) rectangle (10.3,7.5);
	\end{tikzpicture}
	~~=~~-(-1)^n
	\begin{tikzpicture}[scale=0.2,baseline=1.2pc]
		\draw (0,5.5) node[above] {$\scriptscriptstyle{n}$};
		\draw (2,5.5) node[above] {$\scriptscriptstyle{n-1}$};
		\draw[thin] (2,4.5) -- (2,5);
		\draw[thin] (4,3.5) -- (4,4);
		\draw[fill=white] (-0.3,5) rectangle (2.3,5.5);
		\draw[fill=white] (1.7,4) rectangle (4.3,4.5);
		\draw (4,4.5) node[above] {$\scriptscriptstyle{n-2}$};
		\draw [dotted] (4,3.5) -- (8,2);
		\draw (10,1.5) node[above] {$\scriptscriptstyle{2}$};
		\draw (12,0.5) node[above] {$\scriptscriptstyle{1}$};
		\draw (10,1) -- (10,0.5);
		\draw (8,1.5) -- (8,2);
		\draw[fill=white] (7.7,1) rectangle (10.3,1.5);
		\draw[fill=white] (9.7,0) rectangle (12.3,0.5);
	\end{tikzpicture}.
	\]
\end{proof}

\begin{lem}\label{ecriture_escalier}
	Every monomial of $\DCom$ such that the underlying non-oriented graph does not have cycles can be rewriten as a stairway.
\end{lem} 
\begin{proof}
	We prove this result by induction on the weight of monomials, i.e. the number of vertices of the underlying graph. By \Cref{egal_symetrique}, we have that this lemma holds for a monomial of weight $2$. Let $n$ be an integer strictly greater than $2$. Suppose the lemma holds for every monomial of weight $w<n$. We consider $\Phi$, a monomial of weight $n$ and we denote by $\bar{\Phi}$, its underlying non-oriented graph. As $\DCom$ is a properad, the graph $\bar{\Phi}$ is connected:  we label its $n$ vertices by $v_1, \ldots, v_n$. There exists $\alpha$ in $[\![1,n]\!]$ such that the subgraph $\bar{\Phi}^*:=\bar{\Phi} \backslash v_\alpha$ is connected. By the induction hypothesis,  we can rewrite $\Phi^*$ as a stairway, then we can rewrite $\Phi$ as a one of these two following monomials:
	\[
		\begin{tikzpicture}[scale=0.2,baseline=1.2pc]
			\draw[thin] (2,0) -- (2,1.5);
			\draw[thin] (4,1.5) -- (4,2);
			\draw[fill=white] (-0.3,0) rectangle (2.3,0.5);
			\draw[fill=white] (1.7,1) rectangle (4.3,1.5);
			\draw (12,5.5) -- (12,6);
			\draw (10,4.5) -- (10,5);
			\draw[fill=white] (9.7,5) rectangle (12.3,5.5);
			\draw[fill=white] (11.7,6) rectangle (14.3,6.5);
			\draw (8,3.5) -- (8,4);
			\draw (6,2.5) -- (6,3);
			\draw[dotted] (4,2) -- (6,2.5);
			\draw[dotted] (8,4) -- (10,4.5);
			\draw[fill=white] (5.7,3) rectangle (8.3,3.5); 
			\draw (6,3.5) to[out=90,in=270] (5,6);
			\draw[fill=white] (4.7,6) rectangle (7.3,6.5);
		\end{tikzpicture}
		~~\quad\mbox{ or }~~\quad
		\begin{tikzpicture}[scale=0.2,baseline=1.2pc]
			\draw[thin] (2,0) -- (2,1.5);
			\draw[thin] (4,1.5) -- (4,2);
			\draw[fill=white] (-0.3,0) rectangle (2.3,0.5);
			\draw[fill=white] (1.7,1) rectangle (4.3,1.5);
			\draw (12,5.5) -- (12,6);
			\draw (10,4.5) -- (10,5);
			\draw[fill=white] (9.7,5) rectangle (12.3,5.5);
			\draw[fill=white] (11.7,6) rectangle (14.3,6.5);
			\draw (8,3.5) -- (8,4);
			\draw (6,2.5) -- (6,3);
			\draw[dotted] (4,2) -- (6,2.5);
			\draw[dotted] (8,4) -- (10,4.5);
			\draw[fill=white] (5.7,3) rectangle (8.3,3.5); 
			\draw (13,0.5) to[out=90,in=270] (8,3);
			\draw[fill=white] (10.7,0) rectangle (13.3,0.5);
		\end{tikzpicture}
	\]
	and, by invariance of stairways under the cyclic group action, we have our result.
\end{proof}

\begin{lem}\label{cycle_nul}
	Every monomial of $\DCom$ such that the underlying non-oriented graph has a cycle is null. 
\end{lem}
\begin{proof}
	We prove the result by induction on the weight of monomials. We limit ourselves to considering only monomials whose underlying non-oriented graph is a cycle, i.e. monomials whose each elementary block is linked by two edges to another block. We have the relation
	\[
		\begin{tikzpicture}[scale=0.2,baseline=2]
			\draw[thin] (2,2) -- (2,2.5);
			\draw[thin] (4,2) -- (4,2.5);
			\draw (1.7,1.5) rectangle (4.3,2);
			\draw[thin] (2,1.5) -- (2,1);
			\draw[thin] (4,1.5) -- (4,1);
			\draw (1.7,0.5) rectangle (4.3,1);
			\draw[thin] (2,0.5) -- (2,0);
			\draw[thin] (4,0.5) -- (4,0);
		\end{tikzpicture}
		~~=~~0
	\]
	which initialize our induction. Suppose that every cycle of weight $\leqslant n-1$ is null. We consider a cycle $\Phi$ of weight $n$ and, we isole one of the blocs $v_\alpha$ in the cycle (i.e. one of the vertex of the underlying graph) such that its two outputs are linked with an other bloc. In a cycle, such a bloc already exists. We denote by $\Phi^*$, the monomial obtained by the forgetfulness of the bloc $v_\alpha$ in the initial cycle. The monomial $\Phi^*$ does not contain a cycle, then, by \Cref{ecriture_escalier}, $\Phi^*$ can be rewriting in a stairway. Finally, the monomial $\Phi$ can be rewrite as one of the two following monomials in \Cref{fig::monome_1} and \Cref{fig::monome_2}.
	\begin{figure}[!ht]
		\begin{minipage}[b]{.48\linewidth}
			\centering
			\begin{tikzpicture}[scale=0.2]
			\draw[thin] (2,0.5) -- (2,1);
			\draw[fill=white] (-0.3,0) rectangle (2.3,0.5);
			\draw (12,5.5) -- (12,6);
			\draw (10,4.5) -- (10,5);
			\draw[fill=white] (7.7,4) rectangle (10.3,4.5);
			\draw[fill=white] (11.7,6) rectangle (14.3,6.5);
			\draw (8,3.5) -- (8,4);
			\draw (4,1.5) -- (4,2);
			\draw[dotted] (2,1) -- (4,1.5);
			\draw[dotted] (6,3) -- (8,3.5);
			\draw[dotted] (10,5) -- (12,5.5);
			\draw (6,2.5) -- (6,3);
			\draw[fill=white] (3.7,2) rectangle (6.3,2.5);
			\draw (4,2.5) to[out=90,in=270] (3,8);
			\draw (4,2.5) node[above,left] {$\scriptscriptstyle{i}$};
			\draw (8,4.5) to[out=90,in=270] (5,8);
			\draw (8,4.5) node[left] {$\scriptscriptstyle{j}$};
			\draw[fill=white] (2.7,8) rectangle (5.3,8.5);
			\end{tikzpicture}
			\caption{Monomial of form 1 \label{fig::monome_1}}
		\end{minipage} \hfill
		\begin{minipage}[b]{.48\linewidth}
			\centering
			\begin{tikzpicture}[scale=0.2]
			\draw[thin] (2,0.5) -- (2,1);
			\draw[fill=white] (-0.3,0) rectangle (2.3,0.5);
			\draw (10,4.5) -- (10,5);
			\draw[fill=white] (7.7,4) rectangle (10.3,4.5);
			\draw (8,3.5) -- (8,4);
			\draw (4,1.5) -- (4,2);
			\draw[dotted] (2,1) -- (4,1.5);
			\draw[dotted] (6,3) -- (8,3.5);
			\draw (6,2.5) -- (6,3);
			\draw[fill=white] (3.7,2) rectangle (6.3,2.5);
			\draw (0,0.5) to[out=90,in=270] (3,8);
			\draw (10,5) to[out=90,in=270] (5,8);
			\draw[fill=white] (2.7,8) rectangle (5.3,8.5);
			\end{tikzpicture}
			\caption{Monomial of form 2 \label{fig::monome_2}}
		\end{minipage}
	\end{figure}
	By the invariance of staiways under the diagonal action of the cyclic group (cf. \Cref{invariance_cyclique}), a monomial with the form $2$ (see \Cref{fig::monome_2}) can be rewrite as a monomial with the form $1$ (see \Cref{fig::monome_1}). Then, $\Phi$ can be rewrite as a monomial which contains a smaller cycle, then $\Phi$ is null.
\end{proof}
\begin{prop}\label{prop::generateurs_DCom}
	For all $n\geqslant 2$, we have
	\[
		\DCom([\![1,n]\!])=\DCom_{(n-1)}([\![1,n]\!]).
	\]
	with $\DCom_{(n-1)}([\![1,n]\!])$ generated by $\phi_n$, the stairway with $n$ inputs, which is stable under the diagonal action of the cyclic group. In terms of group representations, $\DCom(n,n)$ is given by
	\[
	\DCom([\![1,n]\!])=
	\left\{
	\begin{array}{cc}
		\mathrm{triv}(\Z/n\Z) \uparrow_{\Z/n\Z}^{\mathfrak{S}_n } & \mbox{if} \ n \ \mbox{is even,} \smallskip \\ 
		\mathrm{sgn}(\Z/n\Z) \uparrow_{\Z/n\Z}^{\mathfrak{S}_n} & \mbox{if} \ n \ \mbox{is odd.} 
	\end{array}
	\right. 
	\]
\end{prop}
\begin{proof}
	We already  have that $\DCom_{(k)}([\![1,n]\!])=0$ for all $k$ in $[\![0,n-2]\!]$. We also have that the $\mathfrak{S}$-module $\DCom_{(n-1)}([\![1,n]\!])$ is generated by the stairway with $n$ inputs. Finally, monomials in  $\DLie^!_{(k)}([\![1,n]\!])$ for $k\geqslant n$ have a cycle, thus they are null, by \Cref{cycle_nul}.
\end{proof}
\begin{nota}
	For all integers $n\geqslant 2$, we denote:
	\[
	\sgn(\Z/n\Z)\overset{\mathrm{not.}}{:=} \mathrm{sgn}(\Z/n\Z)^{\otimes n} =
	\left\{
	\begin{array}{cc}
		\mathrm{triv}(\Z/n\Z)  & \mbox{if} \ n \ \mbox{is even;} \smallskip \\ 
		\mathrm{sgn}(\Z/n\Z) & \mbox{if} \ n \ \mbox{is odd.} 
	\end{array}
	\right. 
	\]
\end{nota}
By \Cref{thm::duale_de_koszul}, the dual coprotoperad of $\DLie$ is given by
\[
	\DLie^\antish= \big(\mathscr{P}(V_{\DLie},R_{DJ})\big)^\antish =\mathscr{C}(\Sigma V_{\DLie},\Sigma^2\overline{R_{DJ}}).
\]
We have seen, in \Cref{prop::generateurs_DCom}, that the protoperad $\DCom=\mathscr{P}(V_{\DLie}^*,R_{DJ}^\bot)$ satisfies
\[
	\DCom([\![1,n]\!]) = \DCom_{(n-1)}([\![1,n]\!]) = \sgn(\Z/n\Z)\uparrow_{\Z/n\Z}^{\mathfrak{S}_n\op \times \mathfrak{S}_n}
\]
By \Cref{prop::dual_lineaire_koszul}, we have the following isomorphism $\mathscr{C}(\Sigma V_{\DLie},\Sigma^2\overline{R_{DJ}})^*\cong \mathscr{P}(\Sigma^{-1}V_{\DLie}^*,\Sigma^{-2}R_{DJ}^\bot)$, so, for all integer $n>0$, we have
\[
	\DLie^\antish([\![1,n]\!])=\Sigma^{n-1} \sgn(\Z/n\Z)\big\uparrow_{\Z/n\Z}^{\mathfrak{S}_n}.
\]
So we have the following proposition.

\begin{prop}
	The properad $\Ind(\DLie)_\infty:=\Cobar (\Ind(\DLie)^\antish)$, which is a cofibrant resolution of the properad $\Ind(\DLie)$, is the free properad $(\scrF\Val(\Sigma^{-1}W),\partial_\Delta)$ with $\Sigma^{-1}W$, the $\mathfrak{S}$-bimodule defined by $\Sigma^{-1}W([\![1,m]\!],[\![1,n]\!])=0$ for $m\ne n$ in $\N$ and
	\renewcommand{\arraystretch}{2.5}
	\begin{align*}
	\Sigma^{-1}W([\![1,n]\!]&,[\![1,n]\!]):=~ (\Sigma^{-1}W)_{n-2}([\![1,n]\!],[\![1,n]\!])  =\Sigma^{n-2}\Ind \big( \sgn(\Z/n\Z)\uparrow_{\Z/n\Z}^{\mathfrak{S}_n} \big) \\
	=~& \left\langle \Sigma^{-1}\phi_n:=~\Sigma^{-1}
	\begin{tikzpicture}[scale=0.2,baseline=1]
	\draw[thin] (0,0) -- (0,1.5);
	\draw (0,0) node[below] {\tiny{$1$}};
	\draw (0,1.5) node[above] {\tiny{$1$}};
	\draw[thin] (2,0) -- (2,1.5);
	\draw (2,0) node[below] {\tiny{$2$}};
	\draw (2,1.5) node[above] {\tiny{$2$}};
	\draw (5,0) node[below] {\tiny{$\cdots$}};
	\draw (5,1.5) node[above] {\tiny{$\cdots$}};
	\draw[thin]	(8,0) -- (8,1.5);
	\draw (8,0) node[below] {\tiny{$n$}};
	\draw (8,1.5) node[above] {\tiny{$n$}};
	\draw[fill=white] (-0.3,0.5) rectangle (8.3,1);
	\end{tikzpicture}
	~~=~(-1)^{n+1}\Sigma^{-1}~
	\begin{tikzpicture}[scale=0.2,baseline=1]
	\draw[thin] (0,0) -- (0,1.5);
	\draw (0,0) node[below] {\tiny{$2$}};
	\draw (0,1.5) node[above] {\tiny{$2$}};
	\draw[thin] (2,0) -- (2,1.5);
	\draw (2,0) node[below] {\tiny{$3$}};
	\draw (2,1.5) node[above] {\tiny{$3$}};
	\draw (5,0) node[below] {\tiny{$\cdots$}};
	\draw (5,1.5) node[above] {\tiny{$\cdots$}};
	\draw[thin]	(8,0) -- (8,1.5);
	\draw (8,0) node[below] {\tiny{$1$}};
	\draw (8,1.5) node[above] {\tiny{$1$}};
	\draw[fill=white] (-0.3,0.5) rectangle (8.3,1);
	\end{tikzpicture} 
	\right\rangle
	\end{align*}
	\renewcommand{\arraystretch}{1}
	with $\Sigma^{-1}W([\![1,n]\!],[\![1,n]\!])$ concentrated in homological degree $n-2$ and with the differential $\partial_\Delta$ induced by the coproduct $\Delta$ of the coproperad $\Ind(\DLie)^\antish$, which sends
	\[
	\begin{tikzpicture}[scale=0.2,baseline=1]
	\draw[thin] (0,0) -- (0,1.5);
	\draw (0,0) node[below] {\tiny{$1$}};
	\draw (0,1.5) node[above] {\tiny{$1$}};
	\draw[thin] (1.5,0) node[below] {\tiny{$2$}} -- (1.5,1.5) node[above] {\tiny{$2$}};
	\draw (5,0) node[below] {\tiny{$\cdots$}};
	\draw (5,1.5) node[above] {\tiny{$\cdots$}};
	\draw[thin]	(8,0) -- (8,1.5);
	\draw (8,0) node[below] {\tiny{$n$}};
	\draw (8,1.5) node[above] {\tiny{$n$}};
	\draw[fill=white] (-0.3,0.5) rectangle (8.3,1);
	\end{tikzpicture}
	\ \overset{\Delta}{\longmapsto} \
	\sum_{{\substack{2\leqslant i \leqslant n-1 \\ \sigma \in \Z/n\Z}}}
	\begin{tikzpicture}[scale=0.2,baseline=(n.base)]
	\node (n) at (0,1) {};
	\draw[thin] (0,0) node[below] {\tiny{$\sigma(1)$}} -- (0,2) node[above] {\tiny{$\sigma(1)$}};
	\draw[thin] (1.5,-0.8) node[below] {\tiny{$\sigma(2)$}} -- (1.5,3) node[above] {\tiny{$\sigma(2)$}};
	\draw (2.75,0.5) node[below] {\tiny{$\cdots$}};
	\draw (2.75,1.5) node[above] {\tiny{$\cdots$}};
	\draw[thin]	(4,0) node[below] {\tiny{$\sigma(i)$}} -- (4,2.5) node[above] {\tiny{$\sigma(i)$}};
	\draw (6,0.5) node[below] {\tiny{$\cdots$}};
	\draw (6,1.5) node[above] {\tiny{$\cdots$}};
	\draw[thin]	(8,0) node[below] {\tiny{$\sigma(n)$}} -- (8,2.5) node[above] {\tiny{$\sigma(n)$}};
	\draw[fill=white] (-0.3,0.5) rectangle (4.3,1);
	\draw[fill=white] (3.7,1.5) rectangle (8.3,2);
	\end{tikzpicture}
	\]
\end{prop}

We exhibit the action of the differential on generators of degree $2$, $3$ and $4$.
\begin{itemize}
	\item The element $\phi_2$ in $ \Ind(\DLie)^\antish$ is primitive, i.e. $\Delta(\phi_2)=(1\otimes 1)\phi_2+\phi_2(1\otimes 1)$ then 
	\[
	\partial_\Delta(s^{-1}\phi_2)=-\big( (1\otimes 1)s^{-1}\phi_2+(-1)^{|\phi_2|}s^{-1}\phi_2(1\otimes 1)\big)=0.
	\]
	\item We have seen that the $\mathfrak{S}$-bimodule $\Ind(\DLie)^\antish([\![1,3]\!],[\![1,3]\!])$ is generated by $\phi_3$, the stairway of arity $3$ which is stable under the diagonal action of the cyclic group, so
	\[
		\Delta(\phi_3) = \sum_{i=0}^{2}\sigma^{-i}_{(123)}(\phi_2,1)(1,\phi_2)\sigma_{(123)}^i~;
	\]
	then 
	\begin{align*}
		\partial_\Delta(s^{-1}\phi_3)=&-\big(  \sum_{i=0}^{2}(-1)^{|\phi_2|}\sigma^{-i}_{(123)}(s^{-1}\phi_2,1)(1,s^{-1}\phi_2)\sigma_{(123)}^i \big) \\
		=& \sum_{i=0}^{2}\sigma^{-i}_{(123)}(s^{-1}\phi_2,1)(1,s^{-1}\phi_2)\sigma_{(123)}^i,
	\end{align*}
	which is exactly the double Jacobi relation. 
	\item $\Ind(\DLie)^\antish([\![1,4]\!],[\![1,4]\!])$ is generated by $\phi_4$, with
	\[
		\Delta(\phi_4)=\sum_{i=0}^{3}\sigma^{-i}_{(1234)}\Big( (\phi_2,1)(1,\phi_3)+(\phi_3,1)(1,\phi_2)\Big)\sigma_{(1234)}^i~;
	\]
	then 
	\begin{align*}
	\partial_\Delta(s^{-1}\phi_4) = &~ -\sum_{i=0}^{3}\sigma^{-i}_{(1234)}\Big( (-1)^{|\phi_2|}(s^{-1}\phi_2,1)(1,s^{-1}\phi_3) \\
	& \qquad \qquad +(-1)^{|\phi_3|}(s^{-1}\phi_3,1)(1,s^{-1}\phi_2)\Big)\sigma_{(1234)}^i \\
	= & ~ \sum_{i=0}^{3}\sigma^{-i}_{(1234)}\Big( (s^{-1}\phi_2,1)(1,s^{-1}\phi_3)-(s^{-1}\phi_3,1)(1,s^{-1}\phi_2)\Big)\sigma_{(1234)}^i.
	\end{align*}
\end{itemize}

\subsection{The properad $\DPois$}
We define the properad $\DPois$ which encodes the structure of double Poisson algebra.  $\DPois$ is the quadratic properad gives as follows:
\[
	\DPois:=\scrF\Val\big(V\oplus W\big)\big/ \big\langle R_{As} \oplus D_\lambda \oplus R_{DJ} \big\rangle
\]
with generators concentrated in homological degree $0$: 
\[
	V:=V_{As}= \mu.k\otimes k[\mathfrak{S}_2] =
	\begin{tikzpicture}[baseline=1.8ex,scale=0.15]
		\draw (0,4) -- (2,2);
		\draw (4,4) -- (2,2);
		\draw[fill=black] (2,2) circle (6pt);
		\draw (2,2) -- (2,0);
		\draw (0,4) node[above] {$\scriptscriptstyle{1}$};
		\draw (4,4) node[above] {$\scriptscriptstyle{2}$};
		\draw (2,0) node[below] {$\scriptscriptstyle{1}$};
		\draw (2,2) node[below left] {$\scriptscriptstyle{\mu}$};
	\end{tikzpicture}\otimes k[\mathfrak{S}_2]
\]
and
\[
	W:=V_{\DLie} = 
	\begin{tikzpicture}[baseline=0ex,scale=0.2]
		\draw (0,0) node[below] {$\scriptscriptstyle{1}$};
		\draw (2,0) node[below] {$\scriptscriptstyle{2}$};
		\draw (0,1.5) node[above] {$\scriptscriptstyle{1}$};
		\draw (2,1.5) node[above] {$\scriptscriptstyle{2}$};
		\draw (0,0) -- (0,1.5);
		\draw (2,0) -- (2,1.5);
		\draw[fill=black] (-0.3,0.5) rectangle (2.3,1);
	\end{tikzpicture}\otimes \  \mathrm{sgn}(\mathfrak{S}_2)\uparrow_{\mathfrak{S}_2}^{\mathfrak{S}_2\times\mathfrak{S}_2\op }
\]
and the relations
\begin{itemize}
	\item of associativity for the  product $\mu$~:
	\[
		R_{As}:=\quad
		\begin{tikzpicture}[baseline=0.5ex,scale=0.15]
			\draw (0,4) --(4,0);
			\draw (4,4) -- (2,2);
			\draw (4,0) -- (4,-1);
			\draw (8,4) -- (4,0);
			\draw (0,4) node[above] {$\scriptscriptstyle{1}$};
			\draw (4,4) node[above] {$\scriptscriptstyle{2}$};
			\draw (8,4) node[above] {$\scriptscriptstyle{3}$};
			\draw (4,-1) node[below] {$\scriptscriptstyle{1}$};
			\draw[fill=black] (2,2) circle (6pt);
			\draw[fill=black] (4,0) circle (6pt);
		\end{tikzpicture}
		\quad - \quad
		\begin{tikzpicture}[baseline=0.5ex,scale=0.15]
			\draw (0,4) --(4,0);
			\draw (4,4) -- (6,2);
			\draw (4,0) -- (4,-1);
			\draw (8,4) -- (4,0);
			\draw (0,4) node[above] {$\scriptscriptstyle{1}$};
			\draw (4,4) node[above] {$\scriptscriptstyle{2}$};
			\draw (8,4) node[above] {$\scriptscriptstyle{3}$};
			\draw (4,-1) node[below] {$\scriptscriptstyle{1}$};
			\draw[fill=black] (6,2) circle (6pt);
			\draw[fill=black] (4,0) circle (6pt);
		\end{tikzpicture}~~;
	\]
	\item double Jacobi for the double bracket:
	\[
		R_{DJ}:=\quad
		\begin{tikzpicture}[scale=0.2,baseline=-1]
			\draw (0,3) node[below] {$\scriptscriptstyle{1}$};
			\draw (2,3) node[below] {$\scriptscriptstyle{2}$};
			\draw (4,3) node[below] {$\scriptscriptstyle{3}$};
			\draw[thin] (0,-1) -- (0,1.5);
			\draw[thin] (2,-1) -- (2,1.5);
			\draw[thin] (4,-1) -- (4,1.5);
			\draw (0,-1) node[below] {$\scriptscriptstyle{1}$};
			\draw (2,-1) node[below] {$\scriptscriptstyle{2}$};
			\draw (4,-1) node[below] {$\scriptscriptstyle{3}$};
			\draw[fill=black] (1.7,0.5) rectangle (4.3,1);
			\draw[fill=black] (-0.3,-0.5) rectangle (2.3,0);
		\end{tikzpicture}
		~~+~~
		\begin{tikzpicture}[scale=0.2,baseline=-1]
			\draw (0,3) node[below] {$\scriptscriptstyle{2}$};
			\draw (2,3) node[below] {$\scriptscriptstyle{3}$};
			\draw (4,3) node[below] {$\scriptscriptstyle{1}$};
			\draw[thin] (0,-1) -- (0,1.5);
			\draw[thin] (2,-1) -- (2,1.5);
			\draw[thin] (4,-1) -- (4,1.5);
			\draw (0,-1) node[below] {$\scriptscriptstyle{2}$};
			\draw (2,-1) node[below] {$\scriptscriptstyle{3}$};
			\draw (4,-1) node[below] {$\scriptscriptstyle{1}$};
			\draw[fill=black] (1.7,0.5) rectangle (4.3,1);
			\draw[fill=black] (-0.3,-0.5) rectangle (2.3,0);
		\end{tikzpicture}
		~~+~~
		\begin{tikzpicture}[scale=0.2,baseline=-1]
			\draw (0,3) node[below] {$\scriptscriptstyle{3}$};
			\draw (2,3) node[below] {$\scriptscriptstyle{1}$};
			\draw (4,3) node[below] {$\scriptscriptstyle{2}$};
			\draw[thin] (0,-1) -- (0,1.5);
			\draw[thin] (2,-1) -- (2,1.5);
			\draw[thin] (4,-1) -- (4,1.5);
			\draw (0,-1) node[below] {$\scriptscriptstyle{3}$};
			\draw (2,-1) node[below] {$\scriptscriptstyle{1}$};
			\draw (4,-1) node[below] {$\scriptscriptstyle{2}$};
			\draw[fill=black] (1.7,0.5) rectangle (4.3,1);
			\draw[fill=black] (-0.3,-0.5) rectangle (2.3,0);
		\end{tikzpicture}
		~~;
	\]
	\item of derivation:
	\[
		D_\lambda := \quad
		\begin{tikzpicture}[scale=0.2,baseline=-1]
			\draw (0,0) -- (0,1);
			\draw (2,0) -- (2,1.5);
			\draw (2,1.5) -- (1,2.5);
			\draw (2,1.5) -- (3,2.5);
			\draw (0,1) to[out=90,in=270] (-1,2.5);
			\draw (-1,2.5) node[above] {$\scriptscriptstyle{1}$};
			\draw (1,2.5) node[above] {$\scriptscriptstyle{2}$};
			\draw (3,2.5) node[above] {$\scriptscriptstyle{3}$};
			\draw (0,0) node[below] {$\scriptscriptstyle{1}$};
			\draw (2,0) node[below] {$\scriptscriptstyle{2}$};
			\draw[fill=black] (2,1.5) circle (4.5pt);
			\draw[fill=black] (-0.3,0.5) rectangle (2.3,1);
		\end{tikzpicture}
		\quad - \quad
		\begin{tikzpicture}[scale=0.2,baseline=-1]
			\draw (1,0.5) -- (2,1.5);
			\draw (1,0.5) -- (0,1.5);
			\draw (1,0.5) -- (1,-0.5);
			\draw (0,2.5) -- (0,1.5);
			\draw (2,2.5) -- (2,1.5);
			\draw (4,2) -- (4,2.5);
			\draw (3,-0.5) to[out=90,in=270] (4,1.5) ;
			\draw (0,2.5) node[above] {$\scriptscriptstyle{2}$};
			\draw (2,2.5) node[above] {$\scriptscriptstyle{1}$};
			\draw (4,2.5) node[above] {$\scriptscriptstyle{3}$};
			\draw (1,-0.5) node[below] {$\scriptscriptstyle{1}$};
			\draw (3,-0.5) node[below] {$\scriptscriptstyle{2}$};
			\draw[fill=black] (1,0.5) circle (4.5pt);
			\draw[fill=black] (1.7,1.5) rectangle (4.3,2);
		\end{tikzpicture}
		\quad - \quad
		\begin{tikzpicture}[scale=0.2,baseline=-1]
			\draw (1,0.5) -- (2,1.5);
			\draw (1,0.5) -- (0,1.5);
			\draw (1,0.5) -- (1,-0.5);
			\draw (0,2.5) -- (0,1.5);
			\draw (2,2.5) -- (2,1.5);
			\draw (-2,2) -- (-2,2.5);
			\draw (-1,-0.5) to[out=90,in=270] (-2,1.5);
			\draw (-2,2.5) node[above] {$\scriptscriptstyle{1}$};
			\draw (0,2.5) node[above] {$\scriptscriptstyle{2}$};
			\draw (2,2.5) node[above] {$\scriptscriptstyle{3}$};
			\draw (-1,-0.5) node[below] {$\scriptscriptstyle{1}$};
			\draw (1,-0.5) node[below] {$\scriptscriptstyle{2}$};
			\draw[fill=black] (1,0.5) circle (4.5pt);
			\draw[fill=black] (-2.3,1.5) rectangle (0.3,2);
		\end{tikzpicture}~~.
	\]
\end{itemize}
We recall the following result of Vallette

\begin{prop}[see {\cite[lem. 155 prop. 156 and 158]{Val03}}]\label{prop::loi_remplacement}
	Let $\calP $ be a properad of the form $\calP  := \mathscr{P}(V\oplus W, R\oplus D_\lambda \oplus S)$, with $\lambda$, a compatible distributive law. Then we have the following isomorphism of $\mathfrak{S}$-bimodules
	\[
		\calP \cong \mathcal{A}\boxtimes_c\Val \mathcal{B}
	\]
	with $\mathcal{A}:=\mathscr{P}(V,R)$ and $\mathcal{B}:=\mathscr{P}(W,S)$. Also, if the sum 
	\[
		\sum_{m,n}\mathrm{dim}_k\big((V\oplus W)([\![1,m]\!],[\![1,n]\!])\big)
	\]
	is finite and $W$ is concentrated in homological degree $0$, then we have the isomorphism of  $\mathfrak{S}$-bimodules $\calP ^\antish\cong \mathcal{B}^\antish\boxtimes_c\Val \mathcal{A}^\antish$ with $\mathcal{A}:=\mathscr{P}(V,R)$ and $\mathcal{B}:=\mathscr{P}(W,S)$. Moreover, if the properads $\mathcal{A}$ and $\mathcal{B}$ are Koszul, then the properad $\calP $ is also a Koszul properad. 
\end{prop}

For $\DPois$, the relation of derivation $D_\lambda$ is given by a compatible replacement law (see \cite{Val03,Val07}), with $\lambda$ the following morphism of $\mathfrak{S}$-bimodules:
\[
	\lambda : (\Ibox\oplus W)\boxtimes_c\Val (\Ibox\oplus V)^{(1)_W,(1)_V} \longrightarrow (\Ibox\oplus V)\boxtimes_c\Val (\Ibox\oplus W)^{(1)_V,(1)_W} 
\]
given by 
\[
	\lambda : \quad
	\begin{tikzpicture}[scale=0.2,baseline=1]
		\draw (0,0) -- (0,1);
		\draw (2,0) -- (2,1.5);
		\draw (2,1.5) -- (1,2.5);
		\draw (2,1.5) -- (3,2.5);
		\draw (0,1) to[out=90,in=270] (-1,2.5);
		\draw (-1,2.5) node[above] {$\scriptscriptstyle{1}$};
		\draw (1,2.5) node[above] {$\scriptscriptstyle{2}$};
		\draw (3,2.5) node[above] {$\scriptscriptstyle{3}$};
		\draw (0,0) node[below] {$\scriptscriptstyle{1}$};
		\draw (2,0) node[below] {$\scriptscriptstyle{2}$};
		\draw[fill=black] (2,1.5) circle (4.5pt);
		\draw[fill=white] (-0.3,0.5) rectangle (2.3,1);
	\end{tikzpicture}
	\quad \longmapsto \quad
	\begin{tikzpicture}[scale=0.2,baseline=0]
		\draw (1,0.5) -- (2,1.5);
		\draw (1,0.5) -- (0,1.5);
		\draw (1,0.5) -- (1,-0.5);
		\draw (0,2.5) -- (0,1.5);
		\draw (2,2.5) -- (2,1.5);
		\draw (4,2) -- (4,2.5);
		\draw (3,-0.5) to[out=90,in=270] (4,1.5) ;
		\draw (0,2.5) node[above] {$\scriptscriptstyle{2}$};
		\draw (2,2.5) node[above] {$\scriptscriptstyle{1}$};
		\draw (4,2.5) node[above] {$\scriptscriptstyle{3}$};
		\draw (1,-0.5) node[below] {$\scriptscriptstyle{1}$};
		\draw (3,-0.5) node[below] {$\scriptscriptstyle{2}$};
		\draw[fill=black] (1,0.5) circle (4.5pt);
		\draw[fill=white] (1.7,1.5) rectangle (4.3,2);
	\end{tikzpicture}
	\quad + \quad
	\begin{tikzpicture}[scale=0.2,baseline=0]
		\draw (1,0.5) -- (2,1.5);
		\draw (1,0.5) -- (0,1.5);
		\draw (1,0.5) -- (1,-0.5);
		\draw (0,2.5) -- (0,1.5);
		\draw (2,2.5) -- (2,1.5);
		\draw (-2,2) -- (-2,2.5);
		\draw (-1,-0.5) to[out=90,in=270] (-2,1.5);
		\draw (-2,2.5) node[above] {$\scriptscriptstyle{1}$};
		\draw (0,2.5) node[above] {$\scriptscriptstyle{2}$};
		\draw (2,2.5) node[above] {$\scriptscriptstyle{3}$};
		\draw (-1,-0.5) node[below] {$\scriptscriptstyle{1}$};
		\draw (1,-0.5) node[below] {$\scriptscriptstyle{2}$};
		\draw[fill=black] (1,0.5) circle (4.5pt);
		\draw[fill=white] (-2.3,1.5) rectangle (0.3,2);
	\end{tikzpicture}~~.
\]

\begin{lem}\label{lem::loi_remplacement_compatible}
	The morphisms of $\mathfrak{S}$-bimodules
	\[
	\As \boxtimes_c\Val \Ind(\DLie) ^{(1),(2)} \longrightarrow \DPois
	\quad \mbox{ and } \quad
	\As \boxtimes_c\Val  \Ind(\DLie) ^{(2),(1)} \longrightarrow \DPois
	\]
	are injectives.
\end{lem}
\begin{proof}
	We  start by considering the  morphism $\As\boxtimes_c\Val \Ind(\DLie) ^{(2),(1)} \rightarrow \DPois $: in $\DPois$, we consider the terms 
	\[
		\begin{tikzpicture}[scale=0.2,baseline=(n.base)]
			\node (n) at (2,1.5) {};
			\draw (0,0) -- (0,1) to[out=90,in=270] (-2,3) -- (-2,3.5);
			\draw (2,0) -- (2,1.5) -- (0,3.5);
			\draw (2,1.5) -- (4,3.5) -- (3,2.5) -- (2,3.5);
			\draw (-2,3.5) node[above] {$\scriptscriptstyle{1}$};
			\draw (0,3.5) node[above] {$\scriptscriptstyle{2}$};
			\draw (2,3.5) node[above] {$\scriptscriptstyle{3}$};
			\draw (4,3.5) node[above] {$\scriptscriptstyle{4}$};
			\draw (0,0) node[below] {$\scriptscriptstyle{1}$};
			\draw (2,0) node[below] {$\scriptscriptstyle{2}$};
			\draw[fill=black] (2,1.5) circle (4.5pt);
			\draw[fill=black] (3,2.5) circle (4.5pt);
			\draw[fill=white] (-0.3,0.5) rectangle (2.3,1);
		\end{tikzpicture}
		\qquad \mbox{ and } \qquad
		\begin{tikzpicture}[scale=0.2,baseline=(n.base)]
			\node (n) at (2,1.5) {};
			\draw (0,0) -- (0,1) to[out=90,in=270] (-2,3) -- (-2,3.5);
			\draw (2,0) -- (2,1.5) -- (0,3.5);
			\draw (1,2.5) -- (2,3.5);
			\draw (2,1.5) -- (4,3.5);
			\draw (-2,3.5) node[above] {$\scriptscriptstyle{1}$};
			\draw (0,3.5) node[above] {$\scriptscriptstyle{2}$};
			\draw (2,3.5) node[above] {$\scriptscriptstyle{3}$};
			\draw (4,3.5) node[above] {$\scriptscriptstyle{4}$};
			\draw (0,0) node[below] {$\scriptscriptstyle{1}$};
			\draw (2,0) node[below] {$\scriptscriptstyle{2}$};
			\draw[fill=black] (2,1.5) circle (4.5pt);
			\draw[fill=black] (1,2.5) circle (4.5pt);
			\draw[fill=white] (-0.3,0.5) rectangle (2.3,1);
		\end{tikzpicture}.
	\]
	In the properad $\DPois$, by the relation $D_\lambda$, we have the following equalities
	\[
		\begin{tikzpicture}[scale=0.2,baseline=(n.base)]
			\node (n) at (2,1.5) {};
			\draw (0,0) node[below] {$\scriptscriptstyle{1}$} -- (0,1) to[out=90,in=270] (-2,3) -- (-2,3.5) node[above] {$\scriptscriptstyle{1}$};
			\draw (2,0) node[below] {$\scriptscriptstyle{2}$} -- (2,1.5) -- (0,3.5) node[above] {$\scriptscriptstyle{2}$};
			\draw (2,1.5) -- (4,3.5) node[above] {$\scriptscriptstyle{4}$} -- (3,2.5) -- (2,3.5) node[above] {$\scriptscriptstyle{3}$};
			\draw[fill=black] (2,1.5) circle (4.5pt);
			\draw[fill=black] (3,2.5) circle (4.5pt);
			\draw[fill=white] (-0.3,0.5) rectangle (2.3,1);
		\end{tikzpicture}
		\quad = \quad
		\begin{tikzpicture}[scale=0.2,baseline=(n.base)]
			\node (n) at (2,1.5) {};
			\draw (1,-0.5) node[below] {$\scriptscriptstyle{1}$} -- (1,0.5) -- (2,1.5)-- (2,3) node[above] {$\scriptscriptstyle{1}$};
			\draw (-1.5,3) node[above] {$\scriptscriptstyle{2}$} -- (1,0.5)  -- (-0.5,2) -- (0.5,3) node[above] {$\scriptscriptstyle{3}$};
			\draw (3,-0.5) node[below] {$\scriptscriptstyle{2}$} to[out=90,in=270] (4,1.5) -- (4,2.5) -- (4,3) node[above] {$\scriptscriptstyle{4}$};
			\draw[fill=black] (1,0.5) circle (4.5pt);
			\draw[fill=white] (1.7,1.5) rectangle (4.3,2);
			\draw[fill=black] (-0.5,2) circle (4.5pt);
		\end{tikzpicture}
		\quad + \quad
		\begin{tikzpicture}[scale=0.2,baseline=(n.base)]
			\node (n) at (2,1.5) {};
			\draw (1,-0.5) node[below] {$\scriptscriptstyle{1}$} -- (1,0.5) -- (2,1.5)-- (2,2.5) node[above] {$\scriptscriptstyle{1}$};
			\draw (-0.5,2) node[above] {$\scriptscriptstyle{2}$}-- (1,0.5); 
			\draw (5,-0.5) node[below] {$\scriptscriptstyle{2}$} -- (5,0.5) -- (4,1.5) -- (4,2.5) node[above] {$\scriptscriptstyle{3}$} -- (4,1.5) -- (5,0.5) -- (6.5,2) node[above] {$\scriptscriptstyle{4}$};
			\draw[fill=black] (1,0.5) circle (4.5pt);
			\draw[fill=black] (5,0.5) circle (4.5pt);
			\draw[fill=white] (1.7,1.5) rectangle (4.3,2);
		\end{tikzpicture}
		\quad + \quad 
		\begin{tikzpicture}[scale=0.2,baseline=(n.base)]
			\node (n) at (2,1.5) {};
			\draw (0,3) node[above] {$\scriptscriptstyle{1}$} -- (0,1.5) to[out=270,in=90] (1,-0.5) node[below] {$\scriptscriptstyle{1}$} ;
			\draw (3,-0.5) node[below] {$\scriptscriptstyle{2}$} -- (3,0.5) -- (5.5,3)  node[above] {$\scriptscriptstyle{4}$} -- (3,0.5)  -- (2,1.5) -- (2,3) node[above] {$\scriptscriptstyle{2}$} ;
			\draw (4.5,2) -- (3.5,3) node[above] {$\scriptscriptstyle{3}$} ; 
			\draw[fill=white] (-0.3,1.5) rectangle (2.3,2);
			\draw[fill=black] (3,0.5) circle (4.5pt);
			\draw[fill=black] (4.5,2) circle (4.5pt);
		\end{tikzpicture}
		\quad = \quad 
		\begin{tikzpicture}[scale=0.2,baseline=(n.base)]
			\node (n) at (2,1.5) {};
			\draw (0,0) -- (0,1) to[out=90,in=270] (-2,3) -- (-2,3.5);
			\draw (2,0) -- (2,1.5) -- (0,3.5);
			\draw (1,2.5) -- (2,3.5);
			\draw (2,1.5) -- (4,3.5);
			\draw (-2,3.5) node[above] {$\scriptscriptstyle{1}$};
			\draw (0,3.5) node[above] {$\scriptscriptstyle{2}$};
			\draw (2,3.5) node[above] {$\scriptscriptstyle{3}$};
			\draw (4,3.5) node[above] {$\scriptscriptstyle{4}$};
			\draw (0,0) node[below] {$\scriptscriptstyle{1}$};
			\draw (2,0) node[below] {$\scriptscriptstyle{2}$};
			\draw[fill=black] (2,1.5) circle (4.5pt);
			\draw[fill=black] (1,2.5) circle (4.5pt);
			\draw[fill=white] (-0.3,0.5) rectangle (2.3,1);
		\end{tikzpicture},
	\]
	then $\As\boxtimes_c\Val \Ind(\DLie) ^{(2),(1)} \rightarrow \DPois$ is injective.
	As the double jacobiator is a multiderivation (see \cite{VdB08}), then the morphism $\As\boxtimes_c\Val \Ind(\DLie) ^{(1),(2)} \rightarrow \DPois$ is injective
\end{proof}
\begin{cor}
	We have the following isomorphism of properads:
	\[
		\DPois\cong \As \boxtimes_c\Val \Ind(\DLie).	
	\]
\end{cor}
As the properads $\As$ and $\Ind(\DLie)$ are Koszul (see \cite[Chap. 9]{LV12} for the case of $\As$), we obtain the main theorem of this paper.
\begin{thm}\label{thm::DPois_Koszul}
	The properad $\DPois$ is Koszul.
\end{thm}
\begin{proof}
	Directly by \Cref{prop::loi_remplacement}.
\end{proof}

\appendix

\section{The algebras \texorpdfstring{$\scrA(\DLie,n)$}{A(DLie,n)} are Koszul}
In this section, $\DLie$ is the protoperad of double Lie algebras.\\
We consider the family of quadratic algebras $\scrA(\DLie,n)$, for $n\geqslant 2$, given by the quadratic datum $\left(V(\DLie,n),R(\DLie,n)\right)$, with
\[
	V(\DLie,n)=\{x_{ij}~|~1\leqslant i<j\leqslant n\}
\]
and, for $n$ in $\N$,
\[
R(\DLie,n)=\left\{
\begin{array}{c}
\left.
\begin{array}{c}
x_{ij}x_{jk}- x_{jk}x_{ik} -x_{ik}x_{ij} \\
x_{jk}x_{ij}- x_{ik}x_{jk} -x_{ij}x_{ik} \\
x_{ab}x_{uv}-x_{uv}x_{ab}
\end{array}
\right|
\begin{array}{c}
1\leqslant i<j<k \leqslant n\\
1\leqslant u<v \leqslant n \\
1\leqslant a<b \leqslant n\\
\{a,b\}\cap\{u,v\}=\varnothing
\end{array}
\end{array}
\right\}
~.
\]
\begin{prop}\label{proof::ADLie_Koszul}
	For all $n\geqslant 2$, the algebra $\scrA(\DLie,n)$ is Koszul.
\end{prop}
\begin{proof} 
The algebra $\scrA(\DLie,2)$ is isomorphic to $k[x]$, which is Koszul. We denote by $W^n$, the Koszul dual of  $\scrA(\DLie,n)$; this quadratic algebra is given by the quadratic datum $(V(\DLie,n)^\vee, R(\DLie,n)^\bot)$ :
\[
\begin{aligned}
V(\DLie,n)^\vee&=\{x_{ij}~|~1\leqslant i<j\leqslant n\} \\
R(\DLie,n)^\bot&=\left\{
\begin{array}{c}
\left.
\begin{array}{c}
x_{ij}^2\\
x_{ij}x_{jk} + x_{jk}x_{ik} \\
x_{ij}x_{jk} + x_{ik}x_{ij} \\
x_{ij}x_{ik} + x_{jk}x_{ij} \\
x_{ij}x_{ik} - x_{ik}x_{jk}\\
x_{ab}x_{uv} + x_{uv}x_{ab}
\end{array}
\right|
\begin{array}{c}
1\leqslant i<j<k \leqslant n\\
1\leqslant u<v \leqslant n \\
1\leqslant a<b \leqslant n\\
\{a,b\}\cap\{u,v\}=\varnothing
\end{array}
\end{array}
\right\}
\end{aligned}.
\]
We prove that the algebra $W_n$ is Koszul by the rewriting method; we will follow \cite[Chap. 4, Sect 4.1]{LV12}. 

\textbf{Step 1:} We totally order the set of generators of $W^n$by the right lexicographisc order on indices: 
\[
x_{ij}<x_{kl} \mbox{ if } j<l \mbox{ or } j=l \mbox{ and } i<k.
\]

\textbf{Step 2:} We extend this order to the set of monomials by the left lexicographic order.

\textbf{Step 3:} We obtain the following rewriting rules:
\[
\xymatrix@R=0.15cm{
x_{ij}^2 \ar@{~>}[r]^{\mbox{\textcircled{{\footnotesize $1$}}}} & 0 & 	x_{jk}x_{ik} \ar@{~>}[r]^{\mbox{\textcircled{{\footnotesize $2$}}}} & -x_{ij}x_{jk} \\
x_{ik}x_{ij} \ar@{~>}[r]^{\mbox{\textcircled{{\footnotesize $3$}}}} & -x_{ij}x_{jk} & 	x_{jk}x_{ij} \ar@{~>}[r]^{\mbox{\textcircled{{\footnotesize $4$}}}} & -x_{ij}x_{ik} \\
x_{ik}x_{jk} \ar@{~>}[r]^{\mbox{\textcircled{{\footnotesize $5$}}}} & x_{ij}x_{ik} & 	x_{uv}x_{ij} \ar@{~>}[r]^{\mbox{\textcircled{{\footnotesize $6$}}}} & - x_{ij}x_{uv}. \\
}
\]
Observe that the rewriting rule \textcircled{\footnotesize $5$} is the only one which does not change the sign.

\textbf{Step 4:} We test the confluence of rewriting rules for all critical monomials. Recall that a \emph{critical monomial} is a monomial $x_{ij}x_{kl}x_{uv}$ such that monomials $x_{ij}x_{kl}$ and $x_{kl}x_{uv}$ can be rewrite by rewriting rules. Any critical monomial gives an oriented graph under the rewriting rules which is \emph{confluent} if it has only one terminal vertex.

We denote by \textcircled{\footnotesize  $\alpha$}-\textcircled{\footnotesize $\beta$} the confluence diagram associated to the monomial $x_{ij}x_{kl}x_{uv}$ where $x_{ij}x_{kl}$ is the leading term (the term of the left side) of the rewriting rule \textcircled{\footnotesize  $\alpha$} and $x_{kl}x_{uv}$, the leading term of the rewriting rule \textcircled{\footnotesize  $\beta$}. We adopt the following notation:  for a monomial $x_{ij}x_{kl}x_{uv}$, when we use the rewriting rule \textcircled{\footnotesize  $\alpha$} on $x_{ij}x_{kl}$, we denote that by 
\[
\xymatrix{
	x_{ij}x_{kl}x_{uv} 
	\ar@{~>}[d]_{\mbox{\textcircled{{\footnotesize  $\alpha$}}}} \\
	\widetilde{x_{ab}x_{cd}}{x}_{uv}
}
\]
and when we use the rewriting rule \textcircled{\footnotesize  $\alpha$} on $x_{kl}x_{uv}$, we denote that by 
\[
	\xymatrix{
		x_{ij}x_{kl}x_{uv} 
		\ar@{~>}[d]^{\mbox{\textcircled{{\footnotesize  $\alpha$}}}} \\
		{x}_{ij}\widetilde{x_{ab}x_{cd}}
	}
\]
We start with the case of $1 \leqslant i<j<k \leqslant n$ and $x_{uv}<x_{ij}$ to study diagrams of the form \textcircled{\footnotesize  $1$}-\textcircled{\footnotesize $\beta$}

\noindent 
\begin{minipage}[b]{0.3\linewidth}
$\xymatrix@C=0.2cm@R=0.2cm{
& x_{jk}^2x_{ik} \ar@{~>}[dr]^{\mbox{\textcircled{{\footnotesize $2$}}}} \ar@{~>}[dl]_{\mbox{\mbox{\textcircled{{\footnotesize $1$}}}}} & \\
0 && -x_{jk}x_{ij}x_{jk} \ar@{~>}[dl]_{\mbox{\textcircled{{\footnotesize $4$}}}}\\
& x_{ij}x_{jk}^2 \ar@{~>}[ul]_{\mbox{\mbox{\textcircled{{\footnotesize $1$}}}}} &
}$
\end{minipage}\quad
\begin{minipage}[b]{0.3\linewidth}
$\xymatrix@C=0.2cm@R=0.2cm{
& x_{ik}^2x_{ij} \ar@{~>}[dr]^{\mbox{\textcircled{{\footnotesize $3$}}}} \ar@{~>}[dl]_{\mbox{\mbox{\textcircled{{\footnotesize $1$}}}}} & \\
0 && -x_{ik}x_{ij}x_{jk} \ar@{~>}[dl]_{\mbox{\textcircled{{\footnotesize $3$}}}}\\
& x_{ij}x_{jk}^2 \ar@{~>}[ul]_{\mbox{\mbox{\textcircled{{\footnotesize $1$}}}}} &
}$
\end{minipage}\quad
\begin{minipage}[b]{0.3\linewidth}
$\xymatrix@C=0.2cm@R=0.2cm{
	& x_{jk}^2x_{ij} \ar@{~>}[dr]^{\mbox{\textcircled{{\footnotesize $4$}}}} \ar@{~>}[dl]_{\mbox{\mbox{\textcircled{{\footnotesize $1$}}}}} & \\
	0 && -x_{jk}x_{ij}x_{ik} \ar@{~>}[dl]_{\mbox{\textcircled{{\footnotesize $4$}}}}\\
	& x_{ij}x_{jk}^2 \ar@{~>}[ul]_{\mbox{\mbox{\textcircled{{\footnotesize $1$}}}}} &
}$
\end{minipage}

\noindent 
\begin{minipage}[b]{0.45\linewidth}
	$\xymatrix@C=0.4cm@R=0.3cm{
		&x_{ik}x_{ij}x_{ik} \ar@{~>}[rd]_{\mbox{\textcircled{{\footnotesize $3$}}}} &  & \\
		x_{ik}^2x_{jk} \ar@{~>}[ur]^{\mbox{\textcircled{{\footnotesize $5$}}}} \ar@{~>}[d]_{\mbox{\mbox{\textcircled{{\footnotesize $1$}}}}} & &
		-x_{ij}x_{jk}x_{ik} \ar@{~>}[dl]^{\mbox{\textcircled{{\footnotesize $2$}}}} &  \\
		0& x_{ij}^2x_{jk}\ar@{~>}[l]^{\mbox{\mbox{\textcircled{{\footnotesize $1$}}}}}&&
	}$
\end{minipage}\quad
\begin{minipage}[b]{0.3\linewidth} 
	$\xymatrix@C=0.2cm@R=0.2cm{
		& x_{ij}^2x_{uv} \ar@{~>}[dr]^{\mbox{\textcircled{{\footnotesize $6$}}}} \ar@{~>}[dl]_{\mbox{\mbox{\textcircled{{\footnotesize $1$}}}}} & \\
		0 && -x_{ij}x_{uv}x_{ij} \ar@{~>}[dl]_{\mbox{\textcircled{{\footnotesize $6$}}}}\\
		& x_{uv}x_{ij}^2 \ar@{~>}[ul]_{\mbox{\mbox{\textcircled{{\footnotesize $1$}}}}} &
	}$
\end{minipage}

then, all diagrams for a critical monomial with the leading term of \mbox{\mbox{\textcircled{{\footnotesize $1$}}}} on the left are confluent. Similarly, all diagrams \textcircled{\footnotesize  $\alpha$}-\textcircled{\footnotesize $1$} are confluent. 

Now, we study the diagrams for a critical monomial with the leading term of \mbox{\textcircled{{\footnotesize $2$}}} on the left. We start with \textcircled{\footnotesize  $2$}-\textcircled{\footnotesize $2$}: let $u<i<j<k$:
\[
\xymatrix@C=0.5cm@R=0.3cm{
& -x_{jk}x_{ui}x_{ik}\ar@{~>}[r]_{\mbox{\textcircled{{\footnotesize $6$}}}}  & x_{ui}x_{jk}x_{ik} \ar@{~>}[dr]^{\mbox{\textcircled{{\footnotesize $2$}}}} & \\
x_{jk}x_{ik}x_{uk} \ar@{~>}[ur]^{\mbox{\textcircled{{\footnotesize $2$}}}} \ar@{~>}[dr]_{\mbox{\textcircled{{\footnotesize $2$}}}}&&& -x_{ui}x_{ij}x_{jk} .	\\
&-x_{ij}x_{jk}x_{uk}\ar@{~>}[r]^{\mbox{\textcircled{{\footnotesize $2$}}}} & x_{ij}x_{uj}x_{jk}\ar@{~>}[ur]_{\mbox{\textcircled{{\footnotesize $2$}}}} &
}
\]
For \textcircled{\footnotesize  $2$}-\textcircled{\footnotesize $3$}, there are three cases: we begin with $i<j<u<k$:
\[
\xymatrix@C=0.5cm@R=0.3cm{
	& -x_{jk}x_{iu}x_{uk}\ar@{~>}[r]_{\mbox{\textcircled{{\footnotesize $6$}}}}  & x_{iu}x_{jk}x_{uk} \ar@{~>}[dr]^{\mbox{\textcircled{{\footnotesize $5$}}}} &  \\
	x_{jk}x_{ik}x_{iu} \ar@{~>}[ur]^{\mbox{\textcircled{{\footnotesize $3$}}}} \ar@{~>}[dr]_{\mbox{\textcircled{{\footnotesize $2$}}}}&&& x_{iu}x_{ju}x_{jk}\ar@{~>}[dl]_{\mbox{\textcircled{{\footnotesize $5$}}}} &	;\\
	&-x_{ij}x_{jk}x_{iu}\ar@{~>}[r]^{\mbox{\textcircled{{\footnotesize $6$}}}} & x_{ij}x_{iu}x_{jk} &&
}
\]
for the case $i<j=u<k$, we have:
\[
\xymatrix@C=0.5cm@R=0.3cm{
	& -x_{jk}x_{ij}x_{jk}\ar@{~>}[r]_{\mbox{\textcircled{{\footnotesize $4$}}}}  & x_{ij}x_{ik}x_{jk} \ar@{~>}[dr]^{\mbox{\mbox{\textcircled{{\footnotesize $1$}}}}} &  \\
	x_{jk}x_{ik}x_{ij} \ar@{~>}[ur]^{\mbox{\textcircled{{\footnotesize $3$}}}} \ar@{~>}[dr]_{\mbox{\textcircled{{\footnotesize $2$}}}}&&& 0 &	;\\
	&-x_{ij}x_{jk}x_{ij}\ar@{~>}[r]^{\mbox{\textcircled{{\footnotesize $4$}}}} & x_{ij}x_{ij}x_{ik} \ar@{~>}[ur]_{\mbox{\mbox{\textcircled{{\footnotesize $1$}}}}} &&
}
\]
and, for $i<u<j<k$, we have:
\[
\xymatrix@C=0.5cm@R=0.3cm{
	& -x_{jk}x_{iu}x_{uk}\ar@{~>}[r]_{\mbox{\textcircled{{\footnotesize $6$}}}}  & x_{iu}x_{jk}x_{uk} \ar@{~>}[dr]^{\mbox{\textcircled{{\footnotesize $2$}}}} &  \\
	x_{jk}x_{ik}x_{iu} \ar@{~>}[ur]^{\mbox{\textcircled{{\footnotesize $3$}}}} \ar@{~>}[dr]_{\mbox{\textcircled{{\footnotesize $2$}}}}&&& -x_{iu}x_{uj}x_{jk} &	\\
	&-x_{ij}x_{jk}x_{iu}\ar@{~>}[r]^{\mbox{\textcircled{{\footnotesize $6$}}}} & x_{ij}x_{iu}x_{jk} \ar@{~>}[ur]^{\mbox{\textcircled{{\footnotesize $3$}}}} &&
} \ .
\]
For  \textcircled{\footnotesize  $2$}-\textcircled{\footnotesize $4$}, there are only one case: let $u<i<j<k$
\[
\xymatrix@C=0.5cm@R=0.3cm{
	& -x_{jk}x_{ui}x_{uk}\ar@{~>}[r]_{\mbox{\textcircled{{\footnotesize $6$}}}}  & x_{ui}x_{jk}x_{uk} \ar@{~>}[dr]^{\mbox{\textcircled{{\footnotesize $4$}}}} &  \\
	x_{jk}x_{ik}x_{ui} \ar@{~>}[ur]^{\mbox{\textcircled{{\footnotesize $4$}}}} \ar@{~>}[dr]_{\mbox{\textcircled{{\footnotesize $2$}}}}&&& -x_{ui}x_{uj}x_{uk} &	;\\
	&-x_{ij}x_{jk}x_{ui}\ar@{~>}[r]^{\mbox{\textcircled{{\footnotesize $4$}}}} & x_{ij}x_{ui}x_{uk} \ar@{~>}[ur]^{\mbox{\textcircled{{\footnotesize $4$}}}} &&
}
\]
For \textcircled{\footnotesize  $2$}-\textcircled{\footnotesize $5$}, there is three cases: we begin with $i<u<j<k$:
\[
\xymatrix@C=0.5cm@R=0.3cm{
	& x_{jk}x_{iu}x_{ik}\ar@{~>}[r]_{\mbox{\textcircled{{\footnotesize $6$}}}}  & -x_{iu}x_{jk}x_{ik} \ar@{~>}[dr]^{\mbox{\textcircled{{\footnotesize $2$}}}} &  \\
	x_{jk}x_{ik}x_{uk} \ar@{~>}[ur]^{\mbox{\textcircled{{\footnotesize $5$}}}} \ar@{~>}[dr]_{\mbox{\textcircled{{\footnotesize $2$}}}}&&& -x_{iu}x_{ij}x_{jk} &	;\\
	&-x_{ij}x_{jk}x_{uk}\ar@{~>}[r]^{\mbox{\textcircled{{\footnotesize $2$}}}} & x_{ij}x_{uj}x_{jk} \ar@{~>}[ur]_{\mbox{\textcircled{{\footnotesize $5$}}}}&&
}
\]
for the case $i<j=u<k$, we have:
\[
\xymatrix@C=0.5cm@R=0.3cm{
	& x_{jk}x_{ij}x_{ik}\ar@{~>}[r]_{\mbox{\textcircled{{\footnotesize $4$}}}}  & -x_{ij}x_{ik}x_{ik} \ar@{~>}[dd]^{\mbox{\mbox{\textcircled{{\footnotesize $1$}}}}} &  \\
	x_{jk}x_{ik}x_{jk} \ar@{~>}[ur]^{\mbox{\textcircled{{\footnotesize $5$}}}} \ar@{~>}[dr]_{\mbox{\textcircled{{\footnotesize $2$}}}}&&&  &	;\\
	&-x_{ij}x_{jk}x_{jk}\ar@{~>}[r]^{\mbox{\mbox{\textcircled{{\footnotesize $1$}}}}} & 0
	 &&
}
\]
and, for $i<u<j<k$, we have:
\[
\xymatrix@C=0.5cm@R=0.3cm{
	& x_{jk}x_{iu}x_{ik}\ar@{~>}[r]_{\mbox{\textcircled{{\footnotesize $6$}}}}  & x_{iu}x_{jk}x_{ik} \ar@{~>}[dr]^{\mbox{\textcircled{{\footnotesize $2$}}}} &  \\
	x_{jk}x_{ik}x_{uk} \ar@{~>}[ur]^{\mbox{\textcircled{{\footnotesize $5$}}}} \ar@{~>}[dr]_{\mbox{\textcircled{{\footnotesize $2$}}}}&&& 
	-x_{iu}x_{ij}x_{jk} \ar@{~>}[dl]_{\mbox{\textcircled{{\footnotesize $3$}}}} &	\\
	&-x_{ij}x_{jk}x_{uk}\ar@{~>}[r]^{\mbox{\textcircled{{\footnotesize $5$}}}} & x_{ij}x_{iu}x_{jk}  &&
} \ .
\]
For \textcircled{\footnotesize  $2$}-\textcircled{\footnotesize $6$}, there are three cases: we begin with $u<i<j<k$, $v<k$ and $v\ne j$:
\[
\xymatrix@C=0.5cm@R=0.3cm{
	& -x_{jk}x_{uv}x_{ik}\ar@{~>}[r]_{\mbox{\textcircled{{\footnotesize $6$}}}}  & x_{uv}x_{jk}x_{ik} \ar@{~>}[dr]^{\mbox{\textcircled{{\footnotesize $2$}}}} &  \\
	x_{jk}x_{ik}x_{uv} \ar@{~>}[ur]^{\mbox{\textcircled{{\footnotesize $6$}}}} \ar@{~>}[dr]_{\mbox{\textcircled{{\footnotesize $2$}}}}&&& -x_{uv}x_{ij}x_{jk} &	;\\
	&-x_{ij}x_{jk}x_{uv}\ar@{~>}[r]^{\mbox{\textcircled{{\footnotesize $6$}}}} & x_{ij}x_{uv}x_{jk} \ar@{~>}[ur]_{\mbox{\textcircled{{\footnotesize $6$}}}}&&
}
\]
and for $u<i<j<k$ and $v= j$:
\[
\xymatrix@C=0.5cm@R=0.3cm{
	& -x_{jk}x_{uj}x_{ik}\ar@{~>}[r]_{\mbox{\textcircled{{\footnotesize $4$}}}}  & x_{uj}x_{uk}x_{ik} \ar@{~>}[r]^{\mbox{\textcircled{{\footnotesize $5$}}}} &  x_{uj}x_{ui}x_{uk} \ar@{~>}[dr]_{\mbox{\textcircled{{\footnotesize $3$}}}} \\
	x_{jk}x_{ik}x_{uj} \ar@{~>}[ur]^{\mbox{\textcircled{{\footnotesize $6$}}}} \ar@{~>}[dr]_{\mbox{\textcircled{{\footnotesize $2$}}}}&&& &-x_{ui}x_{ij}x_{uk} \ar@{~>}[dl]^{\mbox{\textcircled{{\footnotesize $6$}}}}	\\
	&-x_{ij}x_{jk}x_{uj}\ar@{~>}[r]^{\mbox{\textcircled{{\footnotesize $4$}}}} & x_{ij}x_{uj}x_{uk} \ar@{~>}[r]_{\mbox{\textcircled{{\footnotesize $2$}}}}& - x_{ui}x_{uk}x_{ij}&
}\ .
\]
For \textcircled{\footnotesize  $3$}-\textcircled{\footnotesize $2$}, let $u<i<j<k$:
\[
\xymatrix@C=0.5cm@R=0.3cm{
	& -x_{ik}x_{ui}x_{ij}\ar@{~>}[r]_{\mbox{\textcircled{{\footnotesize $2$}}}}  & x_{ui}x_{uk}x_{ij}  & \\
	x_{ik}x_{ij}x_{uj} \ar@{~>}[ur]^{\mbox{\textcircled{{\footnotesize $2$}}}} \ar@{~>}[dr]_{\mbox{\textcircled{{\footnotesize $3$}}}}&&& -x_{ui}x_{ij}x_{uk}\ar@{~>}[ul]_{\mbox{\textcircled{{\footnotesize $6$}}}} 	\\
	&-x_{ij}x_{jk}x_{uj}\ar@{~>}[r]^{\mbox{\textcircled{{\footnotesize $3$}}}} & x_{ij}x_{uj}x_{uk}\ar@{~>}[ur]_{\mbox{\textcircled{{\footnotesize $2$}}}} &
}\ .
\]
Consider the case \textcircled{\footnotesize  $3$}-\textcircled{\footnotesize $3$}, let $i<v<j<k$:
\[
\xymatrix@C=0.5cm@R=0.3cm{
	& -x_{ik}x_{iv}x_{vj}\ar@{~>}[r]_{\mbox{\textcircled{{\footnotesize $3$}}}}  & x_{iv}x_{vk}x_{vj}\ar@{~>}[dr]_{\mbox{\textcircled{{\footnotesize $3$}}}}  & \\
	x_{ik}x_{ij}x_{iv} \ar@{~>}[ur]^{\mbox{\textcircled{{\footnotesize $3$}}}} \ar@{~>}[dr]_{\mbox{\textcircled{{\footnotesize $3$}}}}&&& -x_{iv}x_{vj}x_{jk} .	\\
	&-x_{ij}x_{jk}x_{iv}\ar@{~>}[r]^{\mbox{\textcircled{{\footnotesize $6$}}}} & x_{ij}x_{iv}x_{jk}\ar@{~>}[ur]_{\mbox{\textcircled{{\footnotesize $3$}}}} &
}
\]
For the case \textcircled{\footnotesize  $3$}-\textcircled{\footnotesize $4$}, let $u<i<j<k$:
\[
\xymatrix@C=0.5cm@R=0.3cm{
	& -x_{ik}x_{ui}x_{uj}\ar@{~>}[r]_{\mbox{\textcircled{{\footnotesize $4$}}}}  & x_{ui}x_{uk}x_{uj} \ar@{~>}[dr]_{\mbox{\textcircled{{\footnotesize $3$}}}} & \\
	x_{ik}x_{ij}x_{ui} \ar@{~>}[ur]^{\mbox{\textcircled{{\footnotesize $4$}}}} \ar@{~>}[dr]_{\mbox{\textcircled{{\footnotesize $3$}}}}&&& -x_{ui}x_{uj}x_{jk} .	\\
	&-x_{ij}x_{jk}x_{ui}\ar@{~>}[r]^{\mbox{\textcircled{{\footnotesize $6$}}}} & x_{ij}x_{ui}x_{jk}\ar@{~>}[ur]_{\mbox{\textcircled{{\footnotesize $4$}}}} &
}
\]
Consider the case \textcircled{\footnotesize  $3$}-\textcircled{\footnotesize $5$}, let $i<u<j<k$:
\[
\xymatrix@C=0.5cm@R=0.3cm{
	& x_{ik}x_{iu}x_{ij}\ar@{~>}[r]_{\mbox{\textcircled{{\footnotesize $3$}}}}  & -x_{iu}x_{uk}x_{ij} \ar@{~>}[dr]_{\mbox{\textcircled{{\footnotesize $6$}}}} & \\
	x_{ik}x_{ij}x_{uj} \ar@{~>}[ur]^{\mbox{\textcircled{{\footnotesize $5$}}}} \ar@{~>}[dr]_{\mbox{\textcircled{{\footnotesize $3$}}}}&&& x_{iu}x_{ij}x_{uk} .	\\
	&-x_{ij}x_{jk}x_{uj}\ar@{~>}[r]^{\mbox{\textcircled{{\footnotesize $4$}}}} & x_{ij}x_{uj}x_{uk}\ar@{~>}[ur]_{\mbox{\textcircled{{\footnotesize $5$}}}} &
}
\]
For \textcircled{\footnotesize  $3$}-\textcircled{\footnotesize $6$}, there is three cases: we begin with $u<i<j<k$, $v\ne j$ and $v\ne k$:
\[
\xymatrix@C=0.5cm@R=0.3cm{
	& -x_{ik}x_{uv}x_{ij}\ar@{~>}[r]_{\mbox{\textcircled{{\footnotesize $6$}}}}  & x_{uv}x_{ik}x_{ij} \ar@{~>}[dr]^{\mbox{\textcircled{{\footnotesize $3$}}}} &  \\
	x_{ik}x_{ij}x_{uv} \ar@{~>}[ur]^{\mbox{\textcircled{{\footnotesize $6$}}}} \ar@{~>}[dr]_{\mbox{\textcircled{{\footnotesize $3$}}}}&&& -x_{uv}x_{ij}x_{jk} &	;\\
	&-x_{ij}x_{jk}x_{uv}\ar@{~>}[r]^{\mbox{\textcircled{{\footnotesize $6$}}}} & x_{ij}x_{uv}x_{jk} \ar@{~>}[ur]_{\mbox{\textcircled{{\footnotesize $6$}}}}&&
}
\]
and for $u<i<j<k$, $v\ne j$ and $v = k$:
\[
\xymatrix@C=0.5cm@R=0.3cm{
	& -x_{ik}x_{uk}x_{ij}\ar@{~>}[r]_{\mbox{\textcircled{{\footnotesize $2$}}}}  & x_{ui}x_{ik}x_{ij} \ar@{~>}[dr]^{\mbox{\textcircled{{\footnotesize $3$}}}} &  \\
	x_{ik}x_{ij}x_{uk} \ar@{~>}[ur]^{\mbox{\textcircled{{\footnotesize $6$}}}} \ar@{~>}[dr]_{\mbox{\textcircled{{\footnotesize $3$}}}}&&& -x_{ui}x_{ij}x_{jk} &	\\
	&-x_{ij}x_{jk}x_{uv}\ar@{~>}[r]^{\mbox{\textcircled{{\footnotesize $2$}}}} & x_{ij}x_{uj}x_{jk} \ar@{~>}[ur]_{\mbox{\textcircled{{\footnotesize $2$}}}}&&
}\ .
\]
For \textcircled{\footnotesize  $4$}-\textcircled{\footnotesize $2$}, let $u<i<j<k$:
\[
\xymatrix@C=0.5cm@R=0.3cm{
	& -x_{jk}x_{ui}x_{ij}\ar@{~>}[r]_{\mbox{\textcircled{{\footnotesize $6$}}}}  & x_{ui}x_{jk}x_{ij} \ar@{~>}[dr]^{\mbox{\textcircled{{\footnotesize $4$}}}} & \\
	x_{jk}x_{ij}x_{uj} \ar@{~>}[ur]^{\mbox{\textcircled{{\footnotesize $2$}}}} \ar@{~>}[dr]_{\mbox{\textcircled{{\footnotesize $4$}}}}&&& -x_{ui}x_{ij}x_{ik} .	\\
	&-x_{ij}x_{ik}x_{uj}\ar@{~>}[r]^{\mbox{\textcircled{{\footnotesize $6$}}}} & x_{ij}x_{uj}x_{ik}\ar@{~>}[ur]_{\mbox{\textcircled{{\footnotesize $2$}}}} &
}
\]
Consider the case \textcircled{\footnotesize  $4$}-\textcircled{\footnotesize $3$}, let $i<u<j<k$:
\[
\xymatrix@C=0.5cm@R=0.3cm{
	& -x_{jk}x_{iu}x_{uj}\ar@{~>}[r]_{\mbox{\textcircled{{\footnotesize $6$}}}}  & x_{iu}x_{jk}x_{uj}\ar@{~>}[dr]^{\mbox{\textcircled{{\footnotesize $4$}}}}  & \\
	x_{jk}x_{ij}x_{iu} \ar@{~>}[ur]^{\mbox{\textcircled{{\footnotesize $3$}}}} \ar@{~>}[dr]_{\mbox{\textcircled{{\footnotesize $4$}}}}&&& -x_{iu}x_{uj}x_{uk} .	\\
	&-x_{ij}x_{jk}x_{iu}\ar@{~>}[r]^{\mbox{\textcircled{{\footnotesize $3$}}}} & x_{ij}x_{iu}x_{uk}\ar@{~>}[ur]_{\mbox{\textcircled{{\footnotesize $3$}}}} &
}
\]
Consider the case \textcircled{\footnotesize  $4$}-\textcircled{\footnotesize $4$}, let $u<i<j<k$:
\[
\xymatrix@C=0.5cm@R=0.3cm{
	& -x_{jk}x_{ui}x_{uj}\ar@{~>}[r]_{\mbox{\textcircled{{\footnotesize $6$}}}}  & x_{ui}x_{jk}x_{uj}\ar@{~>}[dr]^{\mbox{\textcircled{{\footnotesize $4$}}}}  & \\
	x_{jk}x_{ij}x_{ui} \ar@{~>}[ur]^{\mbox{\textcircled{{\footnotesize $4$}}}} \ar@{~>}[dr]_{\mbox{\textcircled{{\footnotesize $4$}}}}&&& -x_{ui}x_{uj}x_{uk} .	\\
	&-x_{ij}x_{ik}x_{ui}\ar@{~>}[r]^{\mbox{\textcircled{{\footnotesize $4$}}}} & x_{ij}x_{ui}x_{uk}\ar@{~>}[ur]_{\mbox{\textcircled{{\footnotesize $4$}}}} &
}
\]
Consider the case \textcircled{\footnotesize  $4$}-\textcircled{\footnotesize $5$}, let $i<u<j<k$:
\[
\xymatrix@C=0.5cm@R=0.3cm{
	& -x_{jk}x_{iu}x_{ij}\ar@{~>}[r]_{\mbox{\textcircled{{\footnotesize $6$}}}}  & x_{iu}x_{jk}x_{ij}\ar@{~>}[dr]^{\mbox{\textcircled{{\footnotesize $5$}}}}  & \\
	x_{jk}x_{ij}x_{uj} \ar@{~>}[ur]^{\mbox{\textcircled{{\footnotesize $5$}}}} \ar@{~>}[dr]_{\mbox{\textcircled{{\footnotesize $4$}}}}&&& -x_{iu}x_{ij}x_{ik} .	\\
	&-x_{ij}x_{ik}x_{uj}\ar@{~>}[r]^{\mbox{\textcircled{{\footnotesize $6$}}}} & x_{ij}x_{uj}x_{ik}\ar@{~>}[ur]_{\mbox{\textcircled{{\footnotesize $5$}}}} &
}
\]
For \textcircled{\footnotesize  $4$}-\textcircled{\footnotesize $6$}, there are three cases: we begin with $i<u<j<k$, $v\ne j$ and $v\ne k$:
\[
\xymatrix@C=0.5cm@R=0.3cm{
	& -x_{jk}x_{uv}x_{ij}\ar@{~>}[r]_{\mbox{\textcircled{{\footnotesize $6$}}}}  & x_{uv}x_{jk}x_{ij} \ar@{~>}[dr]^{\mbox{\textcircled{{\footnotesize $3$}}}} &  \\
	x_{jk}x_{ij}x_{uv} \ar@{~>}[ur]^{\mbox{\textcircled{{\footnotesize $6$}}}} \ar@{~>}[dr]_{\mbox{\textcircled{{\footnotesize $4$}}}}&
	&& -x_{uv}x_{ij}x_{ik} &	;\\
	&-x_{ij}x_{ik}x_{uv}\ar@{~>}[r]^{\mbox{\textcircled{{\footnotesize $6$}}}} & x_{ij}x_{uv}x_{ik} \ar@{~>}[ur]_{\mbox{\textcircled{{\footnotesize $6$}}}}&&
}
\]
and for $i<u<j<k$, $v\ne j$ and $v = k$:
\[
\xymatrix@C=0.5cm@R=0.3cm{
	& -x_{jk}x_{uk}x_{ij}\ar@{~>}[r]_{\mbox{\textcircled{{\footnotesize $2$}}}}  & x_{uj}x_{jk}x_{ij} \ar@{~>}[r]^{\mbox{\textcircled{{\footnotesize $4$}}}} & -x_{uj}x_{ij}x_{ik} \ar@{~>}[dd]_{\mbox{\textcircled{{\footnotesize $5$}}}}\\
	x_{jk}x_{ij}x_{uk} \ar@{~>}[ur]^{\mbox{\textcircled{{\footnotesize $6$}}}} \ar@{~>}[dr]_{\mbox{\textcircled{{\footnotesize $4$}}}}&&&  &	\\
	&-x_{ij}x_{ik}x_{uv}\ar@{~>}[r]^{\mbox{\textcircled{{\footnotesize $2$}}}} & x_{ij}x_{ui}x_{ik} \ar@{~>}[r]_{\mbox{\textcircled{{\footnotesize $4$}}}}& -x_{ui}x_{uj}x_{ik} &
} \ .
\]
For \textcircled{\footnotesize  $5$}-\textcircled{\footnotesize $2$}, we have three cases. We begin with the case where $u<i<j<k$:
\[
\xymatrix@C=0.5cm@R=0.3cm{
	& -x_{ik}x_{uj}x_{jk}\ar@{~>}[r]_{\mbox{\textcircled{{\footnotesize $6$}}}}  & x_{uj}x_{ik}x_{jk} \ar@{~>}[r]^{\mbox{\textcircled{{\footnotesize $5$}}}} & x_{uj}x_{ij}x_{ik} \ar@{~>}[dd]_{\mbox{\textcircled{{\footnotesize $5$}}}} \\
	x_{ik}x_{jk}x_{uk} \ar@{~>}[ur]^{\mbox{\textcircled{{\footnotesize $2$}}}} \ar@{~>}[dr]_{\mbox{\textcircled{{\footnotesize $5$}}}}&&&  ;	\\
	&x_{ij}x_{ik}x_{uk}\ar@{~>}[r]^{\mbox{\textcircled{{\footnotesize $2$}}}} & -x_{ij}x_{ui}x_{ik}\ar@{~>}[r]_{\mbox{\textcircled{{\footnotesize $4$}}}} & x_{ui}x_{ij}x_{ik}
}
\]
we continue with  $u=i<j<k$:
\[
\xymatrix@C=0.5cm@R=0.3cm{
	& -x_{ik}x_{ij}x_{jk}\ar@{~>}[r]_{\mbox{\textcircled{{\footnotesize $3$}}}}  & x_{ij}x_{jk}x_{jk} \ar@{~>}[dd]^{\mbox{\mbox{\textcircled{{\footnotesize $1$}}}}} &  \\
	x_{ik}x_{jk}x_{ik} \ar@{~>}[ur]^{\mbox{\textcircled{{\footnotesize $2$}}}} \ar@{~>}[dr]_{\mbox{\textcircled{{\footnotesize $5$}}}}&&&  ;	\\
	&x_{ij}x_{ik}x_{ik}\ar@{~>}[r]^{\mbox{\mbox{\textcircled{{\footnotesize $1$}}}}} & 0 & 
}
\]
and we finish by 
$i<u<j<k$:
\[
\xymatrix@C=0.5cm@R=0.3cm{
	& -x_{ik}x_{uj}x_{jk}\ar@{~>}[r]_{\mbox{\textcircled{{\footnotesize $6$}}}}  & x_{uj}x_{ik}x_{jk} \ar@{~>}[r]^{\mbox{\textcircled{{\footnotesize $5$}}}} & x_{uj}x_{ij}x_{ik} \ar@{~>}[dd]_{\mbox{\textcircled{{\footnotesize $2$}}}} \\
	x_{ik}x_{jk}x_{uk} \ar@{~>}[ur]^{\mbox{\textcircled{{\footnotesize $2$}}}} \ar@{~>}[dr]_{\mbox{\textcircled{{\footnotesize $5$}}}}&&&  	\\
	&x_{ij}x_{ik}x_{uk}\ar@{~>}[r]^{\mbox{\textcircled{{\footnotesize $5$}}}} & x_{ij}x_{iu}x_{ik}\ar@{~>}[r]_{\mbox{\textcircled{{\footnotesize $3$}}}} & -x_{iu}x_{uj}x_{ik}
}\ .
\]
For \textcircled{\footnotesize  $5$}-\textcircled{\footnotesize $3$}, let $i<j<u<k$:
\[
\xymatrix@C=0.5cm@R=0.3cm{
	& -x_{ik}x_{ju}x_{uk}\ar@{~>}[r]_{\mbox{\textcircled{{\footnotesize $6$}}}}  & x_{ju}x_{ik}x_{uk} \ar@{~>}[rd]^{\mbox{\textcircled{{\footnotesize $5$}}}} &  \\
	x_{ik}x_{jk}x_{ju} \ar@{~>}[ur]^{\mbox{\textcircled{{\footnotesize $3$}}}} \ar@{~>}[dr]_{\mbox{\textcircled{{\footnotesize $5$}}}}&&& x_{ju}x_{iu}x_{ik} \ar@{~>}[dl]_{\mbox{\textcircled{{\footnotesize $2$}}}} 	\\
	&x_{ij}x_{ik}x_{ju}\ar@{~>}[r]^{\mbox{\textcircled{{\footnotesize $6$}}}} & -x_{ij}x_{ju}x_{ik}&
}\ .
\]
For \textcircled{\footnotesize  $5$}-\textcircled{\footnotesize $4$}, we have three cases. We begin with the case where $u<i<j<k$:
\[
\xymatrix@C=0.5cm@R=0.3cm{
	& -x_{ik}x_{uj}x_{uk}\ar@{~>}[r]_{\mbox{\textcircled{{\footnotesize $6$}}}}  & x_{uj}x_{ik}x_{uk} \ar@{~>}[r]^{\mbox{\textcircled{{\footnotesize $2$}}}} & -x_{uj}x_{uj}x_{ik} \ar@{~>}[dd]_{\mbox{\textcircled{{\footnotesize $3$}}}} \\
	x_{ik}x_{jk}x_{uj} \ar@{~>}[ur]^{\mbox{\textcircled{{\footnotesize $4$}}}} \ar@{~>}[dr]_{\mbox{\textcircled{{\footnotesize $5$}}}}&&&  ;	\\
	&x_{ij}x_{ik}x_{uj}\ar@{~>}[r]^{\mbox{\textcircled{{\footnotesize $6$}}}} & -x_{ij}x_{uj}x_{ik}\ar@{~>}[r]_{\mbox{\textcircled{{\footnotesize $2$}}}} & x_{ui}x_{ij}x_{ik}
}\ ;
\]
we continue with  $u=i<j<k$:
\[
\xymatrix@C=0.5cm@R=0.3cm{
	& -x_{ik}x_{ij}x_{ik}\ar@{~>}[r]_{\mbox{\textcircled{{\footnotesize $3$}}}}  & x_{ij}x_{jk}x_{ik} \ar@{~>}[r]^{\mbox{\textcircled{{\footnotesize $2$}}}} & -x_{ij}x_{ij}x_{jk} \ar@{~>}[dd]_{\mbox{\mbox{\textcircled{{\footnotesize $1$}}}}}  \\
	x_{ik}x_{jk}x_{ij} \ar@{~>}[ru]^{\mbox{\textcircled{{\footnotesize $4$}}}} \ar@{~>}[dr]_{\mbox{\textcircled{{\footnotesize $5$}}}}&&&  ;	\\
	&x_{ij}x_{ik}x_{ij}\ar@{~>}[r]^{\mbox{\textcircled{{\footnotesize $3$}}}} & x_{ij}x_{ij}x_{jk}\ar@{~>}[r]_{\mbox{\mbox{\textcircled{{\footnotesize $1$}}}}}  & 0
} \ ;
\]
and we finish by 
$i<u<j<k$:
\[
\xymatrix@C=0.5cm@R=0.3cm{
	& -x_{ik}x_{uj}x_{uk}\ar@{~>}[r]_{\mbox{\textcircled{{\footnotesize $6$}}}}  & x_{uj}x_{ik}x_{uk} \ar@{~>}[r]^{\mbox{\textcircled{{\footnotesize $5$}}}} & x_{uj}x_{iu}x_{ik} \ar@{~>}[dd]_{\mbox{\textcircled{{\footnotesize $4$}}}} \\
	x_{ik}x_{jk}x_{uj} \ar@{~>}[ur]^{\mbox{\textcircled{{\footnotesize $4$}}}} \ar@{~>}[dr]_{\mbox{\textcircled{{\footnotesize $5$}}}}&&&  	\\
	&x_{ij}x_{ik}x_{uj}\ar@{~>}[r]^{\mbox{\textcircled{{\footnotesize $6$}}}} & -x_{ij}x_{uj}x_{ik}\ar@{~>}[r]_{\mbox{\textcircled{{\footnotesize $5$}}}} & -x_{iu}x_{ij}x_{ik}
}\ .
\]
For \textcircled{\footnotesize  $5$}-\textcircled{\footnotesize $5$}, we have three cases. We begin with the case where $i<j<u<k$:
\[
\xymatrix@C=0.5cm@R=0.3cm{
	& -x_{ik}x_{ju}x_{uk}\ar@{~>}[r]_{\mbox{\textcircled{{\footnotesize $6$}}}}  & x_{uj}x_{ik}x_{uk} \ar@{~>}[rd]^{\mbox{\textcircled{{\footnotesize $5$}}}} &  \\
	x_{ik}x_{jk}x_{uk} \ar@{~>}[ur]^{\mbox{\textcircled{{\footnotesize $5$}}}} \ar@{~>}[dr]_{\mbox{\textcircled{{\footnotesize $5$}}}}&&& x_{ju}x_{ij}x_{ik} \ar@{~>}[dl]_{\mbox{\textcircled{{\footnotesize $4$}}}} 	\\
	&x_{ij}x_{ik}x_{uk}\ar@{~>}[r]^{\mbox{\textcircled{{\footnotesize $5$}}}} & x_{ij}x_{iu}x_{ik}&
}\ .
\]
For \textcircled{\footnotesize  $5$}-\textcircled{\footnotesize $6$}, there are many  cases: we begin with $i<j<k$, $u\ne i$ and $j<v<k$:
\[
\xymatrix@C=0.5cm@R=0.3cm{
	& -x_{ik}x_{uv}x_{jk}\ar@{~>}[r]_{\mbox{\textcircled{{\footnotesize $6$}}}}  & x_{uv}x_{ik}x_{jk} \ar@{~>}[dr]^{\mbox{\textcircled{{\footnotesize $5$}}}} &  \\
	x_{ik}x_{jk}x_{uv} \ar@{~>}[ur]^{\mbox{\textcircled{{\footnotesize $6$}}}} \ar@{~>}[dr]_{\mbox{\textcircled{{\footnotesize $5$}}}}&
	&& -x_{uv}x_{ij}x_{ik} \ar@{~>}[dl]_{\mbox{\textcircled{{\footnotesize $6$}}}} &	\\
	&x_{ij}x_{ik}x_{uv}\ar@{~>}[r]^{\mbox{\textcircled{{\footnotesize $6$}}}} & -x_{ij}x_{uv}x_{ik} &&
}\ ;
\]
for $i<j<k$, $u,v\ne i$ and $v< j$:
\[
\xymatrix@C=0.5cm@R=0.3cm{
	& -x_{ik}x_{uv}x_{jk}\ar@{~>}[r]_{\mbox{\textcircled{{\footnotesize $6$}}}}  & x_{uv}x_{ik}x_{jk} \ar@{~>}[dr]^{\mbox{\textcircled{{\footnotesize $5$}}}} &  \\
	x_{ik}x_{jk}x_{uv} \ar@{~>}[ur]^{\mbox{\textcircled{{\footnotesize $6$}}}} \ar@{~>}[dr]_{\mbox{\textcircled{{\footnotesize $5$}}}}&
	&& -x_{uv}x_{ij}x_{ik}&	\\
	&x_{ij}x_{ik}x_{uv}\ar@{~>}[r]^{\mbox{\textcircled{{\footnotesize $6$}}}} & -x_{ij}x_{uv}x_{ik}  \ar@{~>}[ru]_{\mbox{\textcircled{{\footnotesize $6$}}}} &&
}\ ;
\]
for $u=i<j<v<k$: 
\[
\xymatrix@C=0.5cm@R=0.3cm{
	& -x_{ik}x_{iv}x_{jk}\ar@{~>}[r]_{\mbox{\textcircled{{\footnotesize $3$}}}}  & x_{iv}x_{vk}x_{jk} \ar@{~>}[dr]^{\mbox{\textcircled{{\footnotesize $2$}}}} &  \\
	x_{ik}x_{jk}x_{iv} \ar@{~>}[ur]^{\mbox{\textcircled{{\footnotesize $6$}}}} \ar@{~>}[dr]_{\mbox{\textcircled{{\footnotesize $5$}}}}&&& -x_{iv}x_{jv}x_{vk} \ar@{~>}[dl]_{\mbox{\textcircled{{\footnotesize $5$}}}} &	\\
	&x_{ij}x_{ik}x_{iv}\ar@{~>}[r]^{\mbox{\textcircled{{\footnotesize $3$}}}} & x_{ij}x_{iv}x_{vk}&  &
} \ ;
\]
for $u=i<v<j<k$:
\[
\xymatrix@C=0.5cm@R=0.3cm{
	& -x_{ik}x_{iv}x_{jk}\ar@{~>}[r]_{\mbox{\textcircled{{\footnotesize $3$}}}}  & x_{iv}x_{vk}x_{jk} \ar@{~>}[dr]^{\mbox{\textcircled{{\footnotesize $5$}}}} &  \\
	x_{ik}x_{jk}x_{iv} \ar@{~>}[ur]^{\mbox{\textcircled{{\footnotesize $6$}}}} \ar@{~>}[dr]_{\mbox{\textcircled{{\footnotesize $5$}}}}&&& x_{iv}x_{vj}x_{vk}  &	\\
	&x_{ij}x_{ik}x_{iv}\ar@{~>}[r]^{\mbox{\textcircled{{\footnotesize $3$}}}} & -x_{ij}x_{iv}x_{vk} \ar@{~>}[ru]_{\mbox{\textcircled{{\footnotesize $3$}}}}&  &
} \ ;
\]
for $u<i=v<j<k$:
\[
\xymatrix@C=0.5cm@R=0.3cm{
	& -x_{ik}x_{ui}x_{jk}\ar@{~>}[r]_{\mbox{\textcircled{{\footnotesize $4$}}}}  & x_{ui}x_{uk}x_{jk} \ar@{~>}[rd]^{\mbox{\textcircled{{\footnotesize $5$}}}} &  \\
	x_{ik}x_{jk}x_{ui} \ar@{~>}[ur]^{\mbox{\textcircled{{\footnotesize $6$}}}} \ar@{~>}[dr]_{\mbox{\textcircled{{\footnotesize $5$}}}}&&& x_{ui}x_{uj}x_{uk}  &	\\
	&x_{ij}x_{ik}x_{ui}\ar@{~>}[r]^{\mbox{\textcircled{{\footnotesize $4$}}}} & -x_{ij}x_{ui}x_{uk} \ar@{~>}[ru]_{\mbox{\textcircled{{\footnotesize $4$}}}}&  &
} \ .
\]

Finally, consider the diagrams \textcircled{\footnotesize  $6$}-\textcircled{\footnotesize $\beta$}. We start with the case \textcircled{\footnotesize  $6$}-\textcircled{\footnotesize $2$}; there are two cases:
$i<j<k<v$ and $u\ne i,j,k$:
\[
\xymatrix@C=0.5cm@R=0.3cm{
	& -x_{uv}x_{ij}x_{jk}\ar@{~>}[r]_{\mbox{\textcircled{{\footnotesize $6$}}}}  & x_{ij}x_{uv}x_{jk} \ar@{~>}[rd]^{\mbox{\textcircled{{\footnotesize $6$}}}} &   \\
	x_{uv}x_{jk}x_{ik} \ar@{~>}[ur]^{\mbox{\textcircled{{\footnotesize $2$}}}} \ar@{~>}[dr]_{\mbox{\textcircled{{\footnotesize $6$}}}}&&&  -x_{ij}x_{jk}x_{uv};	\\
	&-x_{jk}x_{uv}x_{ij}\ar@{~>}[r]^{\mbox{\textcircled{{\footnotesize $6$}}}} & x_{jk}x_{ij}x_{uv}\ar@{~>}[ru]_{\mbox{\textcircled{{\footnotesize $2$}}}} & 
}
\]
and the case $i<j<k<v$ and $u= i$:
\[
\xymatrix@C=0.5cm@R=0.3cm{
	& -x_{iv}x_{ij}x_{jk}\ar@{~>}[r]_{\mbox{\textcircled{{\footnotesize $3$}}}}  & x_{ij}x_{jv}x_{jk} \ar@{~>}[rd]^{\mbox{\textcircled{{\footnotesize $3$}}}} &   \\
	x_{iv}x_{jk}x_{ik} \ar@{~>}[ur]^{\mbox{\textcircled{{\footnotesize $2$}}}} \ar@{~>}[dr]_{\mbox{\textcircled{{\footnotesize $6$}}}}&&&  -x_{ij}x_{jk}x_{kv}.	\\
	&-x_{jk}x_{iv}x_{ik}\ar@{~>}[r]^{\mbox{\textcircled{{\footnotesize $3$}}}} & x_{jk}x_{ik}x_{kv}\ar@{~>}[ru]_{\mbox{\textcircled{{\footnotesize $2$}}}} & 
}
\]
For the case \textcircled{\footnotesize  $6$}-\textcircled{\footnotesize $3$}, we begin with the case where $i<j<k<v$ and $u\ne i,j,k$:
\[
\xymatrix@C=0.5cm@R=0.3cm{
	& -x_{uv}x_{ij}x_{jk}\ar@{~>}[r]_{\mbox{\textcircled{{\footnotesize $6$}}}}  & x_{ij}x_{uv}x_{jk} \ar@{~>}[rd]^{\mbox{\textcircled{{\footnotesize $6$}}}} &   \\
	x_{uv}x_{ik}x_{ij} \ar@{~>}[ur]^{\mbox{\textcircled{{\footnotesize $3$}}}} \ar@{~>}[dr]_{\mbox{\textcircled{{\footnotesize $6$}}}}&&&  -x_{ij}x_{jk}x_{uv};	\\
	&-x_{ik}x_{uv}x_{ij}\ar@{~>}[r]^{\mbox{\textcircled{{\footnotesize $6$}}}} & x_{ik}x_{ij}x_{uv}\ar@{~>}[ru]_{\mbox{\textcircled{{\footnotesize $3$}}}} & 
}
\]
and the case $i<j<k<v$ and $u= j$:
\[
\xymatrix@C=0.5cm@R=0.3cm{
	& -x_{iv}x_{ij}x_{jk}\ar@{~>}[r]_{\mbox{\textcircled{{\footnotesize $3$}}}}  & x_{ij}x_{jv}x_{jk} \ar@{~>}[rd]^{\mbox{\textcircled{{\footnotesize $3$}}}} &   \\
	x_{iv}x_{jk}x_{ik} \ar@{~>}[ur]^{\mbox{\textcircled{{\footnotesize $3$}}}} \ar@{~>}[dr]_{\mbox{\textcircled{{\footnotesize $6$}}}}&&&  -x_{ij}x_{jk}x_{kv}.	\\
	&-x_{jk}x_{iv}x_{ik}\ar@{~>}[r]^{\mbox{\textcircled{{\footnotesize $3$}}}} & x_{jk}x_{ik}x_{kv}\ar@{~>}[ru]_{\mbox{\textcircled{{\footnotesize $2$}}}} & 
}
\]
For the case \textcircled{\footnotesize  $6$}-\textcircled{\footnotesize $4$}, we begin with the case where $i<j<k<v$ and $u\ne i,j,k$:
\[
\xymatrix@C=0.5cm@R=0.3cm{
	& -x_{uv}x_{ij}x_{ik}\ar@{~>}[r]_{\mbox{\textcircled{{\footnotesize $6$}}}}  & x_{ij}x_{uv}x_{ik} \ar@{~>}[rd]^{\mbox{\textcircled{{\footnotesize $6$}}}} &   \\
	x_{uv}x_{jk}x_{ij} \ar@{~>}[ur]^{\mbox{\textcircled{{\footnotesize $4$}}}} \ar@{~>}[dr]_{\mbox{\textcircled{{\footnotesize $6$}}}}&&&  -x_{ij}x_{ik}x_{uv};	\\
	&-x_{jk}x_{uv}x_{ij}\ar@{~>}[r]^{\mbox{\textcircled{{\footnotesize $6$}}}} & x_{jk}x_{ij}x_{uv}\ar@{~>}[ru]_{\mbox{\textcircled{{\footnotesize $4$}}}} & 
}
\]
and the case $i<j<k<v$ and $u=i$ :
\[
\xymatrix@C=0.5cm@R=0.3cm{
	& -x_{iv}x_{ij}x_{ik}\ar@{~>}[r]_{\mbox{\textcircled{{\footnotesize $3$}}}}  & x_{ij}x_{jv}x_{ik} \ar@{~>}[rd]^{\mbox{\textcircled{{\footnotesize $6$}}}} &   \\
	x_{iv}x_{jk}x_{ij} \ar@{~>}[ur]^{\mbox{\textcircled{{\footnotesize $5$}}}} \ar@{~>}[dr]_{\mbox{\textcircled{{\footnotesize $6$}}}}&&&  -x_{ij}x_{ik}x_{jv}.	\\
	&-x_{jk}x_{iv}x_{ij}\ar@{~>}[r]^{\mbox{\textcircled{{\footnotesize $3$}}}} & x_{jk}x_{ij}x_{jv}\ar@{~>}[ru]_{\mbox{\textcircled{{\footnotesize $4$}}}} & 
}
\]
For the case \textcircled{\footnotesize  $6$}-\textcircled{\footnotesize $5$},
we begin with the sub-case where $i<j<k<v$ and $u\ne i,j,k$:
\[
\xymatrix@C=0.5cm@R=0.3cm{
	& x_{uv}x_{ij}x_{ik}\ar@{~>}[r]_{\mbox{\textcircled{{\footnotesize $6$}}}}  & -x_{ij}x_{uv}x_{ik} \ar@{~>}[rd]^{\mbox{\textcircled{{\footnotesize $6$}}}} &   \\
	x_{uv}x_{ik}x_{jk} \ar@{~>}[ur]^{\mbox{\textcircled{{\footnotesize $5$}}}} \ar@{~>}[dr]_{\mbox{\textcircled{{\footnotesize $6$}}}}&&&  x_{ij}x_{ik}x_{uv};	\\
	&-x_{ik}x_{uv}x_{jk}\ar@{~>}[r]^{\mbox{\textcircled{{\footnotesize $6$}}}} & x_{ik}x_{jk}x_{uv}\ar@{~>}[ru]_{\mbox{\textcircled{{\footnotesize $5$}}}} & 
}
\]
and the case $i<j<k<v$ and $u=j$ :
\[
\xymatrix@C=0.5cm@R=0.3cm{
	& x_{jv}x_{ij}x_{ik}\ar@{~>}[r]_{\mbox{\textcircled{{\footnotesize $4$}}}}  & -x_{ij}x_{iv}x_{ik} \ar@{~>}[rd]^{\mbox{\textcircled{{\footnotesize $3$}}}} &   \\
	x_{jv}x_{ik}x_{jk} \ar@{~>}[ur]^{\mbox{\textcircled{{\footnotesize $5$}}}} \ar@{~>}[dr]_{\mbox{\textcircled{{\footnotesize $6$}}}}&&&  x_{ij}x_{ik}x_{kv}.	\\
	&-x_{ik}x_{jv}x_{jk}\ar@{~>}[r]^{\mbox{\textcircled{{\footnotesize $3$}}}} & x_{ik}x_{jk}x_{kv}\ar@{~>}[ru]_{\mbox{\textcircled{{\footnotesize $5$}}}} & 
}
\]
For the case \textcircled{\footnotesize  $6$}-\textcircled{\footnotesize $6$}, there are two sub-cases: the first one is for $b<j<v$ and $u\ne a$:
\[
\xymatrix@C=0.5cm@R=0.3cm{
	& x_{uv}x_{ab}x_{ij}\ar@{~>}[r]_{\mbox{\textcircled{{\footnotesize $6$}}}}  & -x_{ab}x_{uv}x_{ij} \ar@{~>}[rd]^{\mbox{\textcircled{{\footnotesize $6$}}}} &   \\
	x_{uv}x_{ij}x_{ab} \ar@{~>}[ur]^{\mbox{\textcircled{{\footnotesize $6$}}}} \ar@{~>}[dr]_{\mbox{\textcircled{{\footnotesize $6$}}}}&&&  -x_{ab}x_{ij}x_{uv};	\\
	&-x_{ij}x_{uv}x_{ab}\ar@{~>}[r]^{\mbox{\textcircled{{\footnotesize $6$}}}} & x_{ij}x_{ab}x_{uv}\ar@{~>}[ru]_{\mbox{\textcircled{{\footnotesize $6$}}}} & 
}
\]
and the second is for $b<j<v$ and $u=a$:
\[
\xymatrix@C=0.5cm@R=0.3cm{
	& x_{av}x_{ab}x_{ij}\ar@{~>}[r]_{\mbox{\textcircled{{\footnotesize $3$}}}}  & -x_{ab}x_{bv}x_{ij} \ar@{~>}[rd]^{\mbox{\textcircled{{\footnotesize $6$}}}} &   \\
	x_{av}x_{ij}x_{ab} \ar@{~>}[ur]^{\mbox{\textcircled{{\footnotesize $6$}}}} \ar@{~>}[dr]_{\mbox{\textcircled{{\footnotesize $6$}}}}&&& x_{ab}x_{ij}x_{bv}.\\
	&-x_{ij}x_{av}x_{ab}\ar@{~>}[r]^{\mbox{\textcircled{{\footnotesize $3$}}}} & x_{ij}x_{ab}x_{bv}\ar@{~>}[ru]_{\mbox{\textcircled{{\footnotesize $6$}}}} & 
}
\]
Since all diagrams are confluent, the algebra $W_n$ is Koszul. Hence, for all integers $n \geqslant 2$, the algebra $\scrA(\DLie,n)$ is Koszul.
\end{proof}



\nocite{Ler17,Val03,CMW16}
\bibliographystyle{alpha}
\bibliography{biblio_these}

\end{document}